\documentclass{amsart}

\usepackage[margin=3cm]{geometry}

\usepackage{mathtools}
\usepackage{amsfonts}
\usepackage{amssymb}
\usepackage{amsthm}
\usepackage{mathrsfs}
\usepackage[all]{xy}
\usepackage{amsmath}
\usepackage{color}
\usepackage{graphicx}
\usepackage{stmaryrd}
\usepackage{xspace}
\usepackage{xfrac}
\usepackage{etex,etoolbox}

\mathtoolsset{showonlyrefs} % Only number referenced equations.

\theoremstyle{plain}
\newtheorem{proposition}[subsubsection]{Proposition}
\newtheorem{lemma}[subsubsection]{Lemma}
\newtheorem{theorem}[subsubsection]{Theorem}
\newtheorem{corollary}[subsubsection]{Corollary}

\theoremstyle{definition}

\newtheorem{example}[subsubsection]{Example}

\theoremstyle{remark}

\newtheorem{remark}[subsubsection]{Remark}

\newcommand{\PF}{Poisson\xspace} %-Furstenberg } 
\newcommand{\superdiag}{super-diagonal}
\newcommand{\Aff}{\operatorname{Aff}} %Affine group
\newcommand{\UT}{\operatorname{UT}} %Upper triangular matrix group
\newcommand{\NN}{\mathbb{N}} %Natural numbers
\newcommand{\PI}{\mathbb{N}_0} %Non-negative integers
\newcommand{\I}{\mathbb{Z}} %The integers
\newcommand{\R}{\mathbb{R}} %The real numbers
\newcommand{\dyad}{\I[\sfrac{1}{2}]}
\newcommand{\pyad}[1]{\I[\sfrac{1}{{#1}}]}
\newcommand{\Qp}[1]{\mathbb{Q}_{{#1}}} %p-adic numbers

\newcommand{\supp}{\operatorname{supp}}
\newcommand{\sgr}{\operatorname{sgr}} % closed semigroup generated by the support of $\mu$.
\newcommand{\pf}[2]{{#1}_*{#2}} %pushforward measure % #1 is the function, #2 is the measure.
\newcommand{\mm}[1]{m_1 \left ( {#1} \right )} %notation for the mean of the measure #1 
\newcommand{\convpow}[2]{{#1}^{\ast #2}}
\newcommand{\Gmu}{(G,  \mu)} %this saves a lot of typing. Just a random walk (G,  \mu)
\newcommand{\BHluc}{H^\infty_{luc}} % left uniformly continuous essentially bounded harmonic functions
\newcommand{\vect}[1]{{#1}} %vector
\newcommand{\vects}[2]{{{#1}^{({#2})}}} %vector series
\newcommand{\vecti}[2]{#1_#2} %index of vector
\newcommand{\fract}{\operatorname{frac}} %old fractional part notation to be depreciated
\newcommand{\fracc}[1]{\left \{ {#1} \right \}} %new fractional part notation 
\newcommand{\bracc}[1]{\left ( {#1} \right ) } %throw stuff in brackets
\newcommand{\tie}[1]{\flr{m \mm{\imesi{{#1}}}} } %m times the drift, rounded (currently by truncation) to the nearest integer

\newcommand{\mt}[1]{#1} %matrix only with no reference to entries
\newcommand{\mti}[2]{#1_{#2}} %matrix with etries
\newcommand{\ms}[2]{{#1^{(#2)}}} %Matrix series, First argument is matrix name, second argument is sequence position, third argument is subscript
\newcommand{\msi}[3]{#1^{(#2)}_{#3}} %Matrix series referring to a particular entry, First argument is matrix name, second argument is sequence position, third argument is subscript

\newcommand{\diagm}[2]{\theta_{{#1}{#1}}\left ({#2} \right )} %a diagonal matrix the same as the identity except at #1, where it has value #2
\newcommand{\utm}[3]{\theta_{{#1}{#2}}\left ({#3} \right )} %an upper triangular matrix the same as the identity except at #1,#2 where it has value #3
\newcommand{\utmone}[2]{\theta_{{#1}{#2}}\left (1 \right )} %an upper triangular matrix the same as the identity except at #1,#2 where it has value #3

\newcommand{\MFI}[1]{\operatorname{MFI}\left ( {#1} \right )} %set of monotone finite increasing sequences from 0 to #1

\newcommand{\md}[1]{\vect{#1}} %diagonal matrix notation without index
\newcommand{\mdr}[1]{#1} %diagonal matrix notation without forcing vector entry
\newcommand{\mds}[2]{{\vect{#1}^{(#2)}}} %diagonal matrix notation series without index
\newcommand{\mdsi}[3]{#1^{(#2)}_{#3}} %diagonal matrix notation series with index

\newcommand{\ylij}[3]{\mdsi{y}{#1}{#2} - \mdsi{y}{#1}{#3}}

\newcommand{\tci}[2]{T^{(#1)}_{#2}}
\newcommand{\tc}[1]{T^{(#1)}}

\newcommand{\tcinversei}[2]{S^{(#1)}_{#2}}
\newcommand{\tcinverse}[1]{S^{(#1)}}

\newcommand{\flr}[1]{\left \lfloor {#1}  \right \rfloor} %floor function
\newcommand{\sgn}[1]{\operatorname{sgn} \left ( {#1} \right )} %Sign of a number

\newcommand{\D}[0]{D} %find displacement matrix without indicies
\newcommand{\Di}[1]{\mti{D}{{{#1}}}} %displacement matrix with indicies

\newcommand{\hproji}[1]{\pi^{{#1}}}
\newcommand{\imesi}[1]{\mu_{{#1}}}

\newcommand{\nproji}[1]{\pi^{{#1}}}
\newcommand{\nmesi}[1]{\mu_{{#1}}}

\newcommand{\fgseq}[2]{\operatorname{FG}_{#1}({#2})} %upper triangular matrix group of dimension #1 generated by enerated by primes in the sequence #2 on the diagonal and generators the same as the identity everywhere except one upper triangular entry, where 1 is subsituted.
\newcommand{\wlfgseq}[3]{\left | {#1} \right |_{#2,#3}} %_{\fgseq{#2}{#3}}} % Word length of element #1 in \fgseq{#2}{#3}. %Some of the information given in arguments isn't actually displayed.
\newcommand{\mefgseq}[3]{\left \llbracket {#1} \right \rrbracket_{#2,#3} } %_{\fgseq{#2}{#3}}} % Metric estimate of element #1 in \fgseq{#2}{#3}. %Some of the information given in arguments isn't actually displayed.
\newcommand{\adfgseq}[2]{\left \llbracket {#1} \right \rrbracket_{#2,a} } %_{\fgseq{#2}{#3}}} % Metric estimate of element #1 in \fgseq{#2}{#3}. %Some of the information given in arguments isn't actually displayed.
\newcommand{\brofopnorm}[2]{ \left \| {#1}  \right \|_{#2}} %_{\fgseq{#2}{#3}}} % Word length of element #1 in \fgseq{#2}{#3}. %Some of the information given in arguments isn't actually displayed.
\newcommand{\maxonedmdp}[2]{\left [ {#1} \right ]_{#2}} %_{\fgseq{#2}{#3}}} % Word length of element #1 in \fgseq{#2}{#3}. %Some of the information given in arguments isn't actually displayed.

\newcommand{\Gex}{G_{n}^p}
\newcommand{\Gexnp}[2]{G_{{#1}}^{{#2}}}

\newcommand{\Hn}{\I^n} % The normal subgroup in my semidirect product.
\newcommand{\Nn}{\UT_n(\pyad{p})} % The normal subgroup in my semidirect product.
\newcommand{\pathspace}{G^\NN}

\newcommand{\pathmeasure}{\mathbb{P}^\mu}

\newcommand{\why}[3]{\mdsi{y}{{#1}}{{#2}} - \mdsi{y}{{#1}}{{#3}}}

\newcommand{\bnd}{\mathbf{bnd}}

\newcommand{\bP}{{b_P}} %Product of primes in p.

\usepackage{amsmath,amssymb,graphicx,url} 
\usepackage{blindtext}
\usepackage{enumerate}

\begin{document}

\title[Random walks on solvable matrix groups]{Random walks on solvable matrix groups}
\author{John J. Harrison}
\address{School of Mathematical and Physical Sciences, 
The University of Newcastle,
Callaghan NSW 2308, Australia}
\email{John.Harrison@newcastle.edu.au}
\date{\today}
%\thanks{Thanks to ...}

\begin{abstract}
We define matrix groups $\fgseq{n}{P}$ for each natural number $n$ and finite set of primes $P$, such that every rational-valued upper triangular matrix group is a (possibly distorted) subgroup. Brofferio and Schapira \cite{brofferio2011poisson}, described the \PF boundary of $GL_n (\mathbb{Q})$ for measures of finite first moment with respect to adelic length. We show that adelic length is a word metric estimate on $\fgseq{n}{P}$ by constructing another, intermediate, word metric estimate which can be easily computed from the entries of any matrix in the group. In particular, finite first moment of a probability measure with respect to adelic length is an equivalent condition to requiring finite first moment with respect to word length in $\fgseq{n}{P}$. 

We also investigate random walks in the case that $P$ is a length one sequence. Conditions for pointwise convergence in $\R$ or $\Qp{p}$ are given. When these conditions are satisfied, we give path estimates from boundary points, discuss boundary triviality, show that the resulting space is a $\mu$-boundary and give cases where the $\mu$-boundary is the \PF boundary, as conjectured by Kaimanovich in \cite{kaimanovich91}.
\end{abstract}
\maketitle
\tableofcontents

\section{Introduction and preliminaries}

This paper is concerned with solvable matrix groups, $\fgseq{n}{P}$, defined for each natural number $n$ and each finite set of primes $P$. In Section 
\ref{sec:groupsfgnp}, we define these groups and prove some basic properties. Any finitely generated group of upper triangular matrices with rational entries is a subgroup of $\fgseq{n}{P}$ for a suitable choice of $n$ and $P$.

In Section \ref{sec:metricestimate}, we give an estimate of word length on $\fgseq{n}{P}$. Unlike the word length, the estimate can be efficiently computed from the entries of a given group element. 

Brofferio and Schapira \cite{brofferio2011poisson} described the \PF boundary of $GL_n(\mathbb{Q})$ for measures which have finite first moment with respect to \emph{adelic length}. 
In Section \ref{sec:ffmconds}, we show that adelic length is a word metric estimate on $\fgseq{n}{P}$. We also discuss finite moment conditions on probability measures with respect to each notion of length. 

If $P$ is a singleton, then $\fgseq{n}{P}$ is a semi-direct product,  $\I^n \ltimes \UT_n (\pyad{p})$, where $\UT_n (\pyad{p})$ is the group of upper unitriangular matrices with entries in $\pyad{p}$. In Section \ref{sec:prelim} we give formulae for the multiplication of many elements and computation of inverses for semi-direct products of this form.

In the last section, we change our focus to random walks on $\fgseq{n}{P}$ when $P$ is a singleton. We discuss conditions on $\mu$ which allow almost sure pointwise convergence of the right random walk to elements of a $\mu$-boundary which is a product of copies of $\R$ and $\Qp{p}$, as Kaimanovich \cite{kaimanovich91} suggested might be the case. When a path converges to a point $b$ on this $\mu$-boundary, we show that the path may be estimated from $b$ with at most linear error. We also give cases where the $\mu$-boundary is the \PF boundary.  We conclude the chapter by discussing when the \PF boundary is trivial. 

\subsection{Random walks and the Poisson boundary}
Suppose that $G$ is a second countable locally compact group with identity $e$ and that $\mu$ is a probability measure on $G$. Any such pair $(G, \mu)$ is called a \emph{random walk}.

The \emph{space of trajectories} is the set $G^\NN$ with the product $\sigma$-algebra, where $G^\NN$ is infinite Cartesian product of countably many copies of $G$. An element $\omega$ in $G^\NN$ is a \emph{trajectory} or \emph{path}. 

We denote by $\pathmeasure$ the pushforward of  $\mu^\NN$  with respect to the map $S$ on $\pathspace$,  $\pathmeasure = \mu^\NN \circ S^{-1}$, given by
\[ S(\omega_1, \omega_2, \omega_3, \dots, \omega_k, \dots) = (\omega_1, \omega_1 \omega_2, \omega_1 \omega_2 \omega_3, \dots,  \omega_1 \dots \omega_k , \dots ). \] 
The measure  $\pathmeasure$ is called the \emph{path measure}, and the pair $(\pathspace, \pathmeasure)$ is the \emph{path space}. We identify the random walk $\Gmu$ with a discrete time-homogeneous Markov chain $\{R_i\}_{i \in \PI }$, called the \emph{right random walk}, where each random variable $R_n$ is the projection from the path space,
\[ R_n(\omega) = \omega_n   \]
The group $G$ acts diagonally on elements of the path space. This action extends to one on probability measures on $G^\NN$, namely, if $m$ is any probability measure on $G^\NN$, then 
\[ g \cdot m(E)  = m(g^{-1} E) \]
for each measurable set $E$ and $g$ in $G$. 

There are many equivalent definitions of the \PF boundary of a random walk. See e.g. Erschler \cite{erschler10}, Furstenberg \cite{furstenberg63,furstenberg71}, or Kaimanovich and Vershik \cite{kaimanovich83}.

If $X$ is a topological space, then the pair $(X, \cdot)$ is a \emph{$G$-space} if $G$ acts on $X$ and the map $(g,f) \mapsto g \cdot f$ from $G \times X$ to $X$ is continuous with respect to the product topology on $G \times X$. 

A measure $\nu$ on a $G$-space $B$ is said to be \emph{$\mu$-stationary}  if $\mu \ast \nu = \nu$. A $G$-space $B$ equipped with a $\mu$-stationary measure $\nu$ is called a \emph{$\Gmu$-space}.

Let $B$ be a second countable $G$-space and $(B, \nu)$ be a $\Gmu$-space.  Then, $(B, \nu)$ is said to be a \emph{$\mu$-boundary} if there exists a random variable $\bnd$ from $\pathspace$ to $B$,  called the \emph{boundary map}, such that $R_n(\omega) \cdot \nu$ converges in the weak* topology to a point measure $\delta_{\bnd(\omega)}$ for $\pathmeasure$-almost every $\omega$ in $\pathspace$.

A function $f$ in $L^\infty (G)$ is a bounded $\mu$-harmonic function if it satisfies the convolution identity $f = f \ast \mu$. The set of all bounded $\mu$-harmonic functions with pointwise addition, complex conjugation and the multiplication
\[ \lim_{n \rightarrow \infty}  \left ( (f\cdot g) \ast \convpow{\mu}{n} \right )  (x)\]
is a $C^*$ algebra, which we denote by $\BHluc\Gmu$. If $B$ is a compact $\Gmu$-space with a $\mu$-stationary measure $\nu$, then 
\[ P_\nu (\varphi)(g) = \int_B \varphi(g b) \, d \nu(b).  \] 
is an element of  $\BHluc\Gmu$. The map $ P_\nu$ is called the \emph{Poisson transform}. A second countable $\Gmu$-space $(B, \nu)$ is a $\mu$-boundary if and only if Poisson transformation $P_\nu$ is a $\ast$-homomorphism.

A continuous map $\gamma$ from a $G$-space $B$ to a $G$-space $B'$ is \emph{equivariant} if $\gamma(g \cdot b) = g \cdot \gamma (b)$ for all $b$ in $B$ and $g$ in $G$. If $(B, \nu)$ and $(\bar{B}, \bar{\nu})$ are $\mu$-boundaries of the random walk $(G, \nu)$ then $(\bar{B}, \bar{\nu})$ is an \emph{equivariant image of $(B, \nu)$} if there exists an equivariant map $\gamma$ from $B$ to $B'$, such that the pushforward measure $\pf{\gamma}{\nu}$ is equal to $\bar{\nu}$.   

Given any random walk $\Gmu$,  there is a $\mu$-boundary $(\Pi_\mu, \nu)$, such that the Poisson map $P_\nu$ is an isometric *-isomorphism and every other $\mu$-boundary $(B, \eta)$ is an equivariant image of $(\Pi_\mu, \nu)$. 

A \emph{gauge} is an increasing sequence $\mathcal{A}$ of measurable sets $\mathcal{A}_j$ which exhaust $G$. A \emph{gauge function} is a non-negative integer-valued function $\delta$ for which there exists a non-negative constant $K$, such that 
\[ \delta(g h) \leq \delta (g) + \delta (h) + K \]
for all $g$ and $h$ in $G$. A gauge function is subadditive if 
\[ \delta(g h) \leq \delta (g) + \delta (h) \]
for all $g$ and $h$ in $G$.  Let $\mathcal{A} = \{ \mathcal{A} \}_{i=1}^\infty$ be a gauge.  Then, $\mathcal{A}$ is \emph{subadditive} if the \emph{gauge map}
\[ | \gamma |_{\mathcal{A}} = \min \left \{k \in \NN: \gamma \in \mathcal{A}_k \right \} \]
is a subadditive gauge function. If $\delta$ is a gauge function, then the sequence $\mathcal{A}^\delta = \{ \mathcal{A}^\delta_i \}_{i=1}^\infty$ given by
\[ \mathcal{A}^\delta_j = \{ g \in G : \delta(g) \leq j \} \] 
is a gauge. If $\delta$ is a subadditive gauge function, then $\mathcal{A}^\delta_j$ is a subadditive gauge. We say that the gauge $\mathcal{A}$ is \emph{$C$-temperate}, or just \emph{temperate}, if $\lambda_G (\mathcal{A}_j) \leq e^{Cj}$ for all natural numbers $k$ and some positive real number $C$.  A sequence of gauges $\mathcal{A}^{(j)}$ is \emph{uniformly temperate} if there is a positive real number $C$, such that $\mathcal{A}^{(j)}$ is $C$-temperate for each natural number $j$.  See Kaimanovich \cite{kaimanovich91,kaimanovich00} for more details about gauges and gauge functions. 

We make use of the following geometric criterion for boundary maximality. 
\begin{theorem}[Kaimanovich's ray criterion \cite{kaimanovich85}]
	\label{thm:kaimanovichapproxthm}
	Let $\mu$ be a probability measure of finite first moment on a finitely generated group discrete $G$, and $(B, \nu)$ be a $\mu$-boundary of $\Gmu$. Let $d$ be the word length metric corresponding to some finite generating set on $G$. Let $\bnd$ be the boundary map associated with $B$. % from $\pathspace$ to $G$ as it was defined in Proposition \ref{prop:boundarymap}.
	If there exists a sequence of measurable \emph{approximation maps} $\Pi_m$, from $B$ to $G$, such that

	\[ \frac{1}{m} \, d(\Pi_m (\bnd (\omega)), \omega_1 \omega_2 \dots \omega_n ) \rightarrow 0 \]
	for almost every path $\omega = (\omega_1, \omega_2, \dots) $, then $(B, \nu)$ is the \PF boundary of the pair $\Gmu$. 
\end{theorem}
\section{The groups, $\fgseq{n}{P}$}
\label{sec:groupsfgnp}
Let $P$ be a non-empty, finite set of primes.  Let $\theta_{rs}(g)$ be the matrix which is equal to a rational number  $q$ at the $rs$th entry and the identity everywhere else, i.e.
\[ \mti{(\theta_{rs}(q))}{ij} = 
\begin{cases}   
q & \textrm{if} \ (r,s) = (i,j), \\
1 & \textrm{if} \ i=j \ \textrm{and} \ (r,s) \neq (i,j), \\
0 & \textrm{if} \ i\neq j \ \textrm{and} \ (r,s) \neq (i,j),
\end{cases}  \]
for all natural numbers $r$, $s$, $i$ and $j$ less than or equal to $n$. Let 
\[ U = \{ \utmone{r}{s}, \utmone{r}{s}^{-1} : r, s \in \NN,  1 \leq r < s \leq n \} \]
and let
\[ J_P = \left \{ \diagm{r}{p},  \diagm{r}{p}^{-1} : p \in P, r \in \NN,  1 \leq r \leq n \right \}. \]
Let $\fgseq{n}{P}$ be the subgroup of $GL_n(\mathbb{Q})$ generated by $K_P := U \cup J_P$. Let $\wlfgseq{\cdot}{n}{P}$ be the word length function on $\fgseq{n}{P}$ with respect to $K_P$. Identify each prime $p$ with the set $\{ p \}$ in our notation, so that, for example, if $p$ is prime, then 
\[ J_p = \left \{ \diagm{r}{p},  \diagm{r}{p}^{-1} : r \in \NN,  1 \leq r \leq n \right \} \]
and $\fgseq{n}{p}$ is the subgroup of $GL_n(\mathbb{Q})$ generated by $K_p := U \cup J_p$.

\begin{lemma}
	\label{lem:fgequivalence}
	Let $n$ be a natural number and $P = \{p_r\}_{r=1}^{k}$ be a non-empty, finite set of primes. Let $\mt{f}$ be an element of $\fgseq{n}{P}$. Then, $f$ is upper triangular and there are integers $\mti{r}{ij}$ and $\msi{c}{r}{ij}$ for natural numbers $r$ less than or equal to $k$, such that
	\[ \mti{f}{ij} = 
	\begin{cases} \prod_{r=1}^k p_r^{\msi{c}{r}{ij}}
	& \textrm{if} \ i = j, \ \textrm{and} 
	\\ 
	\mti{r}{ij} \prod_{r=1}^k p_r^{\msi{c}{r}{ij}}
	&\textrm{if} \ i < j \end{cases}\]
	and
	$ \mti{r}{ij} $ is not divisible by any of the primes in $P$, 
	for all natural numbers $i$ and $j$, such that $i \leq j \leq n$. Conversely, every matrix which may be written in this form is an element of $\fgseq{n}{P}$. 
\end{lemma}
\begin{proof}
	The forward direction follows from the well known fact that $U$ generates $\UT_n(\I)$, see Elder, Elston and Ostheimer \cite{elder13}. To prove the reverse direction, suppose that $f$ is an upper triangular matrix with entries 
	\[ \mti{f}{ij} = 
	\begin{cases} \prod_{r=1}^k p_r^{\msi{c}{r}{ij}}
	& \textrm{if} \ i = j, \ \textrm{and} 
	\\ 
	\mti{r}{ij} \prod_{r=1}^k p_r^{\msi{c}{r}{ij}}
	&\textrm{if} \ i < j. \end{cases}\]
	Let $\mt{g}$ be the diagonal matrix satisfying  $\mti{g}{ij} = \mti{f}{ij}$ whenever $i = j$.
	Then, $f = g g^{-1} f$, the matrix $g^{-1} f$ is an upper unitriangular and
	\[ \mti{(g^{-1} f)}{ij} = r_{ij} \prod_{r=1}^k p_r^{\msi{c}{r}{ij} - \msi{c}{r}{i,i+1}} \] 
	whenever $i$ is less than $j$. With this in mind, let $M =  - \min_{i \leq j} \left \{ \msi{c}{r}{ij} - \msi{c}{r}{i,i+1} \right \}.$ and let $\mt{A}$ be the diagonal matrix with entries
	\[ \mti{A}{jj} = \begin{cases} p^{jM}
	&  \textrm{if} \ j < n, \ \textrm{and} 
	\\ 
	1
	&\textrm{if} \ j = n. \end{cases}
	\]
	Then, $f = g A^{-1} z A $
	where $z = A g^{-1} f A^{-1}$ is in $\UT_n(\I)$. Since $g A^{-1}$ and $A$ are diagonal matrices which may be written as a finite product of the diagonal generators in $J_P$ and we know that $U$ generates $\UT_n(\I)$, we have shown that $f$ is a finite product of elements in $\fgseq{n}{P}$. The matrix $f$ is upper triangular because every generator is upper triangular.
\end{proof}
\begin{corollary}
	\label{cor:subgrp}
	Suppose that $G$ is a finitely generated upper triangular matrix group with rational entries. Then, $G$ is a subgroup of $\fgseq{n}{P}$ for some set of primes $P$.
\end{corollary}
\begin{proof}
	Let $P = \{ p_r \}_{r=1}^k$ be the minimal set of primes such that each generator $f$ in $G$ can be written in the form
	\[ \mti{f}{ij} = 
	\begin{cases} \prod_{r=1}^k p_r^{\msi{c}{r}{ij}}
	& \textrm{if} \ i = j, \ \textrm{and} 
	\\ 
	\mti{r}{ij} \prod_{r=1}^k p_r^{\msi{c}{r}{ij}}
	&\textrm{if} \ i < j \end{cases}\]
	as in Lemma \ref{lem:fgequivalence}. Then, $P$ is finite because $G$ has finitely many generators and because the absolute value of every entry of each generator can be written as a product of powers of primes.
\end{proof}

\section{A word metric estimate on $\fgseq{n}{P}$}
\label{sec:metricestimate}
Let $b$ be a natural number greater than $1$. For any non-zero rational number $x$ with finite base $b$ representation, $\sum_{i=k}^l \epsilon_i b^i$. Let
\begin{align*} 
d_+^b (x) &= k, \\
d_-^b (x) &= l, \ \textrm{and} \\
\maxonedmdp{x}{b} = 1 + \max \{ &|d_-^b(x)|, |d_+^b(x)| \}.
\end{align*}
If $x$ is zero, let $\maxonedmdp{x}{b} = 0$. Kaimanovich used $\maxonedmdp{\cdot}{2}$ to construct a word metric estimate on $\Aff(\pyad{2})$ in \cite{kaimanovich91}.
\begin{lemma}
	\label{lem:binnormtriangleineq}
	Suppose that $b$ is a natural number greater than $1$. Then, 
	\[ \maxonedmdp{x + y}{b} \leq \maxonedmdp{x}{b} + \maxonedmdp{y}{b}, \]
	whenever $x$ and $y$ are rational numbers with finite base $b$ representation. The equality holds if and only if $x$, $y$ or $x + y$ are zero.
\end{lemma}
\begin{proof}
	Recall from the definition of $\maxonedmdp{\cdot}{b}$ that
	\[ \maxonedmdp{x + y}{b} = \begin{cases}
	1 + \max \left \{ |d_-^b (x + y) |, |d_+^b (x + y) | \right \} &\textrm{if} \ x + y \neq 0, \\
	0  &\textrm{if} \ x + y = 0. 
	\end{cases} \]
	If $x$, $y$ or $x + y$ are zero, then the result is clear. Suppose that $x$, $y$ or $x + y$ are all non-zero. If \[ |d_+^b(x + y)| \geq |d_-^b(x+y)|, \] then by considering addition in terms of the base $b$ expansions of $x$ and $y$, 
	\begin{align*}
	\maxonedmdp{x + y}{b} &\leq 1 + |d_+^b (x + y)| \\ 
	&\leq 2 + \max( |d_+^b(x)|, |d_+^b(y)|) \\
	&< 1 + |d_+^b(x)| + 1 + |d_+^b(y)| \\
	&< \maxonedmdp{x}{b} + \maxonedmdp{y}{b}. 
	\end{align*}
	Similarly if, \[ |d_+^b(x + y)| < |d_-^b(x+y)|, \] then,
	\begin{align*}
	\maxonedmdp{x + y}{b} &\leq 1 + |d_-^b(x + y)|  \\ 
	&\leq 1 + \max ( |d_-^b(x)|, |d_-^b(y)| ) \\ 
	&< 1 + |d_-^b(x)| + 1 + |d_-^b(y)| \\
	&< \maxonedmdp{x}{b} + \maxonedmdp{y}{b}.
	\end{align*}
	Which completes the proof.
\end{proof}
Let $x$ and $y$ be rational numbers with finite base $b$ representation. Then, 
\begin{align}
|d_+^b(x)| &\leq  \log_{b} |x| \leq 1 + |d_+^b(x)| \label{eqn:dpluslogineq}
\end{align}
and similarly, 
\begin{align}
|d_+^b(xy)| &\leq  \log_{b} |xy| \nonumber \\
&=  \log_{b} |x| + \log_{b}|y| \nonumber \\
&\leq 2 + |d_+^b(x)| + |d_+^b(y)|. 
\end{align}
from consideration of base $b$ representations. It is also easy to see that
\[ d_+^b(x) = - \log_b |x|_b \]
where $| \cdot |_p$ is the $b$-adic absolute value. Here, $b$ is not necessarily a prime. 
\begin{lemma}
	\label{lem:binnormmulttriangleineq}
	Suppose that $b$ is a natural number greater than $1$. Then, 
	\[ \maxonedmdp{xy}{b} \leq 3(\maxonedmdp{x}{b} + \maxonedmdp{y}{b}), \]
	whenever $x$ and $y$ are rational numbers with finite base $b$ representation.
\end{lemma}
\begin{proof}
	The statement is true if $x = 0$ or $y=0$. Suppose that $x \neq 0$ and $y \neq 0$. Then,  
	\begin{align*}
	\maxonedmdp{xy}{b} &= 1 + \max \{ |d_-^b(xy)|, |d_+^b(xy)| \} \\
	&\leq 1 + \max \{ 2+ |d_-^b(x)| + |d_-^b(y)|, 2 + |d_+^b(x)| + |d_+^b(y)| \} \\
	&\leq 3 + \max \{ |d_-^b(x)| + |d_-^b(y)|, |d_+^b(x)| + |d_+^b(y)| \} \\
	&\leq 3 + 2 \max \{ |d_-^b(x)|, |d_+^b(x)| \} + 2 \max \{ |d_-^b(y)|, |d_+^b(y)| \} \\
	&\leq 3(\maxonedmdp{x}{b} + \maxonedmdp{y}{b})
	\end{align*}
	which completes the proof.
\end{proof}
\begin{lemma}
	\label{lem:chgbase}
	Let $b_1$ and $b_2$ be natural numbers both greater than $2$. Suppose that $x$ is a rational number with finite base $b_1$ representation. Let $b = b_1 b_2$. Then, 
	%\[ \maxonedmdp{x}{b_1} \leq A \maxonedmdp{x}{b} + B \]
	\begin{align*}
	|d_+^{b_1}(x)| &\leq   2 + \left ( \frac{\ln(b)}{\ln(b_1)} \right )|d_+^{b}(x)|.
	\end{align*}
	and $x$ has a finite base $b$ representation.
\end{lemma}
\begin{proof}
	Using Equation \eqref{eqn:dpluslogineq} and changing base, we have
	\begin{align*}
	|d_+^{b_1}(x)| &\leq  1+ \log_{b_1} |x| \\
	&= 1 + \left ( \frac{\ln(b)}{\ln(b_1)} \right ) \log_b |x| \\
	&\leq 2 + \left ( \frac{\ln(b)}{\ln(b_1)} \right )|d_+^{b}(x)|.
	\end{align*}
	Which is the desired inequality.
\end{proof}
\begin{lemma}
	\label{lem:dminusbound}
	Suppose that $x$ is a non-zero rational number with finite base $b$ representation and that $k$ is a  non-zero integer. Let $y = kx$. Then, $d_-^b(y) \geq d_-^b(x)$. 
\end{lemma}
\begin{proof}
	Obvious.
\end{proof}
\begin{corollary}
	\label{cor:dminuscorollarybound}
	Let $P = \{p_r\}_{r=1}^{k}$ be a non-empty, finite set of primes. Suppose that $x$ is a non-zero rational number which is a product
	\[ x = y \prod_{r=1}^k p_r^{c_r} \]
	for some integer $y$ which does not divide any prime in $P$, where $p_r$ is prime and $c_r$ is an integer for each natural number $r \leq k$. Then, \[ d_-^\bP(x) \geq -\max_{r=1}^k \{ |c_r |\}, \]
	where $\bP$ is the product of all primes in $P$. 
\end{corollary}
\begin{proof}
	For brevity, $M = -\max_{r=1}^k \{ |c_r |\}$. Notice that $y \left ( \prod_{r=1}^k p_r^{c_r - M} \right ) $ is an integer and
	\[ x = y \left ( \prod_{r=1}^k p_r^{c_r - M} \right )  (\bP)^{M}, \]
	then applying Lemma \ref{lem:dminusbound},
	\[ d_-^\bP(x) \geq d_-^\bP \left ( ( \bP)^{M} \right ) = M \]
	which concludes the proof of the lemma.	
\end{proof}
\begin{proposition}[Metric estimate on $\fgseq{n}{P}$]
	\label{prop:rationalmetricestimate}
	Let $P = \{p_r\}_{r=1}^{k}$ be a non-empty set of primes. Let $n$ be a natural number. Let $\bP$ be the product of all primes in $P$.  Suppose that $f$ is an element of $\fgseq{n}{P}$ with diagonal entries 
	\[ \mti{f}{ii} = \prod_{r=1}^k p_r^{\msi{c}{r}{ii}}, \]
	for integers $\msi{c}{r}{ii}$, where $i$ is a natural number less than or equal to $n$. Let 
	\[ \mefgseq{f}{n}{P} =  \sum_{i=1}^n  \sum_{r=1}^k \left | \msi{c}{r}{ii} \right | + \sum_{i=1}^{n-1} \sum_{j=i+1}^n \maxonedmdp{\mti{f}{ij}}{\bP}, \]
	where $|z|$ is the ordinary absolute value of an integer $z$. Then, $\mefgseq{\cdot}{n}{P}$ is a word metric estimate on $\fgseq{n}{P}$. In particular, there is a positive, real constant $J$, such that
	\[ \frac{1}{J} \mefgseq{f}{n}{P} \leq  \wlfgseq{f}{n}{P} \leq J  \mefgseq{f}{n}{P}, \]
	where $|f|$ is the word length of $f$ with respect to the generating set $K_P$.
\end{proposition}
\begin{proof}
	If $\mt{f}$ is the identity, then 
	\[
	\wlfgseq{f}{1}{P} = 0 = \mefgseq{f}{n}{P}.
	\]
	Suppose that $\mt{f}$ is not the identity.  If $n=1$, then
	\[  \wlfgseq{f}{1}{P} = \sum_{r=1}^k \left | \msi{c}{r}{11} \right | = \mefgseq{f}{n}{P}. \]
	
	Suppose $n > 1$.  We begin by proving the upper bound and argue by induction. For each natural number $r$, identify $\fgseq{r}{P}$ with its isomorphic subgroup in $\fgseq{r+1}{P}$, consisting of all elements in  $\fgseq{r+1}{P}$ for which the first row and column are the same as the identity,  $I_{r+1}$. For induction, suppose for every $g$ in $\fgseq{r}{P}$ that there is a positive constant $A_r$, such that
	\[ \wlfgseq{g}{r}{P} \leq A_{r} \mefgseq{g}{r}{P}. \]
	Let $f$ be in $\fgseq{r+1}{P}$. Let $g$ and $h$ be the $(r+1) \times (r+1)$ matrices which are the same as the identity, except that $f$ and $h$ agree on the first row and $f$ and $g$ agree on all remaining rows. 
	%	
	%	
	%	Let $\mt{h}$ and $\mt{g}$ be the $n\times n$ matrices
	%	%
	%	\[ \mti{h}{ij} = \begin{cases} \mti{f}{ij} &\textrm{if} \ i=1, \\ 1 \ &\textrm{if} \ i=j, \\ 0 &\textrm{otherwise} \end{cases} \]
	%	%
	%	and
	%	%
	%	\[ \mti{g}{ij} = \begin{cases} \mti{f}{ij} &\textrm{if} \ i\neq1, \\ 1 \ &\textrm{if} \ i=j, \\ 0 &\textrm{otherwise}. \end{cases}  \]
	%
	Since $f$ is upper triangular, $f = gh$. Subadditivity of the word length function implies 
	\begin{align}
	\wlfgseq{f}{r+1}{P} &= \wlfgseq{gh}{r+1}{P}  \nonumber \\
	&\leq \wlfgseq{g}{r+1}{P} + \wlfgseq{h}{r+1}{P} \nonumber \\ 
	&= \wlfgseq{g}{r}{P} + \wlfgseq{h}{r+1}{P}. \nonumber \\ 
	\intertext{Hence, by the induction hypothesis}
	\wlfgseq{f}{r+1}{P}&\leq A_r \mefgseq{g}{r}{P} + \wlfgseq{h}{r+1}{P} \nonumber \\
	&=  A_r \mefgseq{g}{r+1}{P} + \wlfgseq{h}{r+1}{P}.  \label{eqn:inductionineqrational}
	\end{align}
	The word length of $h$ is bounded by $ \sum_{j=0}^{n-1} \wlfgseq{ \utm{1}{,n-j}{h_{1,n-j}}}{r+1}{P}$, because 
	\[
	h = \prod_{j=0}^{n-1} \utm{1}{,n-j}{h_{1,n-j}},
	\] 
	where the product is taken from left to right. We will now bound the word length of each term in this product. We can write $\utm{1}{j}{h_{11}}$ in terms of the generators as
	\begin{align}
	\utm{1}{j}{h_{11}} = \prod_{r=1}^k \diagm{1}{p_r}^{\msi{c}{r}{11}}.
	\end{align}
	Hence, the word length of $\utm{1}{1}{h_{11}}$ is bounded above by $\sum_{r=1}^k \left | \msi{c}{r}{11} \right |$. 

	We now bound the word length of $\utm{1}{j}{h_{1j}}$ in the case that $j$ is greater than or equal to $2$. If $h_{1j}$ is zero, then $\utm{1}{j}{h_{1j}}$ is the identity, so $|\utm{1}{j}{h_{1j}}| = 0$. Suppose $h_{1j}$ is not zero. We use the unique base $\bP$ expansion, $p^{n_j} \sum_{r=0}^{t_j} \epsilon_r^j {\bP}^r$, of $\mti{h}{1j}$ to construct a product equal to $\utm{1}{j}{h_{1j}}$:
	\begin{equation}
	\utm{1}{j}{h_{1j}}
	= \diagm{1}{\bP}^{n_j} \utmone{1}{j}^{\epsilon^0_j}  \left ( \prod_{k=1}^{t_j} \diagm{1}{\bP} \utmone{1}{j} ^{\epsilon_j^k} \right )  \diagm{1}{\bP} ^{-n_j-t_j + 1}. \label{eqn:prodnearlygens}
	\end{equation}
	Each term $\diagm{1}{\bP}$ is equal to $ \prod_{p \in P} \diagm{1}{p}$, which is a product of $k$ generators. Hence,
	\[  \wlfgseq{\utm{1}{j}{h_{1j}}}{r+1}{P} \leq (2 n_j + 2 t_j + 1) k + 1 + 2 t_j. \]
	By uniqueness of the base $\bP$ expansion of $\mti{h}{1j}$, we have $n_j = d_-^{\bP}(h_{1j})$ and $t_j = d_+^{\bP}(h_{1j}) - d_-^{\bP}(h_{ij})$. Hence,
	\begin{align*}
	\wlfgseq{\utm{1}{j}{h_{1j}}}{r+1}{P} &\leq 2kd_-^{\bP}(h_{1j}) + 2(k + 1)(d_+^{\bP}(h_{1j}) - d_-^{\bP}(h_{ij}) + 1). \\ \intertext{Since, $2(k+1)$ is greater than $1$, }
	\wlfgseq{\utm{1}{j}{h_{1j}}}{r+1}{P} &\leq 6(k + 1)\maxonedmdp{h_{1j}}{\bP}.
	\end{align*}
	Therefore, we have the upper bound, 
	\begin{align*}
	\wlfgseq{h}{r+1}{P} &\leq  6(k + 1) \left ( \sum_{j=2}^n \maxonedmdp{h_{1j}}{\bP}  + \sum_{r=k}^l \left | \msi{c}{r}{1j} \right | \right ) =  6(k + 1) \mefgseq{h}{r+1}{P}.
	\end{align*}
	In conjunction with Equation \eqref{eqn:inductionineqrational}, we have
	\begin{align*}
	\wlfgseq{f}{r+1}{P}&\leq A_r \mefgseq{g}{r+1}{P} + 6(k + 1) \mefgseq{h}{r+1}{P} \\
	&\leq A_{r+1} \left ( \mefgseq{g}{r+1}{P} + \mefgseq{h}{r+1}{P} \right ) \\
	&= A_{r+1} \mefgseq{f}{r+1}{P}.
	\end{align*}
	where $A_{r+1} = 6(k + 1)A_r$. Hence,  %Choosing $A_1 = 1$ and $A = A_n = 6^{n-1}(k + 1)^{n-1}$ gives the upper bound
	\[ \wlfgseq{f}{n}{P} \leq 6^{n-1}(k + 1)^{n-1}  \mefgseq{f}{n}{P},  \]
	which is an upper bound for  $\wlfgseq{f}{n}{P}$ in terms of $\mefgseq{f}{n}{P}$.
	
	We will now find a lower bound for $\wlfgseq{f}{n}{P}$. Since $f$ is not the identity by assumption, $1 \leq \wlfgseq{f}{n}{P}$. It is clear from the choice of generating set, that
	\[  \sum_{r=1}^k \left | \msi{c}{r}{1j} \right |  \leq \wlfgseq{f}{n}{P}, \label{eqn:dminusineq}  \]
	that 
	\begin{equation}
	\left | d_-^{\bP}(\mti{f}{ij}) \right | \leq  \wlfgseq{f}{n}{P}
	\end{equation}
	and that
	\[ |\mti{f}{ij}| \leq \bP^{ \wlfgseq{f}{n}{P}} \label{eqn:dplusineq}\]
	hence, 
	\begin{equation}
	|d_+^{\bP}(\mti{f}{ij})| \leq \log_{\bP} |\mti{f}{ij}| \leq \wlfgseq{f}{n}{P}, 
	\end{equation}
	for all upper off-diagonal entries $f_{ij}$ of the matrix $f$. Equation \eqref{eqn:dminusineq} implies
	\[ \maxonedmdp{\mti{f}{ij}}{\bP} \leq 2 \wlfgseq{f}{n}{P}. \] As there are $\frac{n(n-1)}{2}$ entries in the upper triangular part of $f$, 
	\[\sum_{r=1}^k \left | \msi{c}{r}{1j} \right | + \sum_{i=1}^{n-1} \sum_{j=i+1}^n 	\maxonedmdp{\mti{f}{ij}}{\bP} \leq n(n-1) \wlfgseq{f}{n}{P}, \]
	which means that 
	\[ \frac{1}{J}  \mefgseq{f}{n}{P}  \leq \wlfgseq{f}{n}{P} \leq J \mefgseq{f}{n}{P}, \]
	where $J = \max \left \{ n(n-1), 6^{n-1}(k + 1)^{n-1} \right \}$. This completes the proof.
\end{proof}

\section{Adelic length is a word metric estimate on $\fgseq{n}{P}$}
\label{sec:ffmconds}
Let $\mathcal{P}^*$ be the set of all primes. Let $\mathcal{P} = \mathcal{P}^* \cup \{ \infty \}$. For each vector $v = (v_1, \dots, v_n)$ in $\mathbb{Q}_p^n$, let 
\[|v|_p = \max_{1 \leq i \leq n} |v_i|_p. \]
For each prime $p$ and for each vector $v = (v_1, \dots, v_n)$ in $\R^n$, let
\[ |v|_\infty = \sqrt{\sum_{i=1}^n |v_i|^2}. \]
If $f$ is an element of $GL_n(\mathbb{Q})$, define the operator norm
\[ \brofopnorm{f}{p} = \sup_{|v|_p = 1} |fv|_p. \] 
for each $p$ in $\mathcal{P}$. Note that $\brofopnorm{\cdot}{\infty} $ is just the real spectral norm. For any $f,g$ in $GL_n(\mathbb{Q})$, let
\[ d_p^a(f,g) = \ln^+ \brofopnorm{f^{-1} g}{p} + \ln^+  \brofopnorm{g^{-1} f}{p} , \]
for $p$ prime or $p = \infty$, where $\ln^+$ is the positive part of the natural logarithm function. The map $d_p$ is symmetric, satisfies the triangle inequality and is left-invariant. For all $f,g$ in $GL_n(\mathbb{Q})$, let
\[ d^a(f,g) = \sum_{p \in \mathcal{P}} d_p^a(f,g).  \]
The map $d^a$ is a left invariant pseudometric on $GL_n(\mathbb{Q})$, called the \emph{adelic pseudometric}.
The \emph{adelic length} of an element $f$ in $GL_n(\mathbb{Q})$ is the quantity 
\[ \adfgseq{f}{n} =  d^a(I_n,f) = \sum_{p \in \mathcal{P}} \left ( \ln^+ \brofopnorm{f}{p} + \ln^+  \brofopnorm{f^{-1}}{p} \right ), \]
where $I_n$ is the $n \times n$ identity matrix.

Let $\mu$ be a probability measure on $(\operatorname{GL}_n(\mathbb{Q}), \mu)$. For $p$ and each natural number $j \leq n$, let $\lambda_j (p)$ be the Lyapunov coefficients of $\mu$, given recursively by
\[ \sum_{i=1}^k \lambda_i(p) = \lim_{r \rightarrow \infty} \frac{1}{r} \int \ln \| {\wedge}^k	 g \|_p \, d \mu^{\ast r} (g) \]
where $\wedge $ is the exterior product, as defined in Section 2.2 of Winitzki \cite{winitzki10}. Let
\[ B_p = \operatorname{GL}_n(\mathbb{K_p}/{P_p}), \]
be the associated flag manifolds, where $\mathbb{K_p} = \Qp{p}$ if $p$ is prime or $\R$ if $p = \infty$, where $P_p$ is the normal subgroup of $\operatorname{GL}_n(\mathbb{K_p})$ consisting of all matrices $f$, such that $f_{ij}$ is zero whenever $\lambda_i (p) < \lambda_j (p)$.

Brofferio and Schapira described the adelic pseudometric in \cite{brofferio2011poisson}.  They showed that if a measure $\mu$ on $GL_n(\mathbb{Q})$ has finite first moment with respect to the adelic length, i.e.
\[  \int_{GL_n(\mathbb{Q})} \adfgseq{f}{n} \, d \mu (f) \]
is finite, then there is a probability measure $\nu$ supported on the product
\[ \mathbf{B} = \prod_{p \in \mathcal{P}} B_p, \]
such that $(\mathbf{B}, \nu)$ is the \PF boundary of $(\operatorname{GL}_n(\mathbb{Q}), \mu)$.

In this section, we show that the adelic length is a word metric estimate on $\fgseq{n}{P}$. We make use of the word metric estimate, $\mefgseq{\cdot}{n}{P}$, from the previous section in many of the arguments. The argument is split into three subsections. The first subsection contains the technical lemmas and propositions needed to construct an upper bound of the form
\[ \wlfgseq{f}{n}{P} \leq S + T \adfgseq{f}{n}, \]
for positive, real constants $S$ and $T$. The second subsection is similar, aiming towards constructing a lower bound of the form
\[Q + R  \adfgseq{f}{n}  \leq \wlfgseq{f}{n}{P}, \]
for positive, real constants $Q$ and $R$. The final section brings together the earlier sections, stating and proving that adelic length is a word metric estimate on $\fgseq{n}{P}$, before making other concluding remarks. 

\subsection{Upper bounds}
\label{subsec:ub}
\begin{lemma}
	\label{lem:chgbasedm}
	Let $p_1$ and $p_2$ be distinct primes. Suppose that $x$ is a non-zero rational number of the form
	\[ x = p_1^{c_1} p_2^{c_2} y  \]
	for some integers $c_1, c_2$ and $y$, such that $p_1$ and $p_2$ do not divide $y$.  Let $b = p_1 p_2$. Then, 
	%\[ \maxonedmdp{x}{b_1} \leq A \maxonedmdp{x}{b} + B \]
	\begin{align*}
	d_-^{b}(x) &= \min \left \{ c_1, c_2 \right \} = \min \left \{ -\log_{p_1} | x |_{p_1}, -\log_{p_2} |x|_{p_2} \right \}.
	\end{align*}
	and $x$ has a finite base $b$ representation.
\end{lemma}
\begin{proof} Suppose that $c_1 \leq c_2$. Then,
	\[ x = b^{c_1}  p_2^{c_2 - c_1} y  \]
	and $p_2^{c_2 - c_1} y$ is an integer which is not divisible by $p_1$, so
	\[ c_1 = d_-^{p_1}(p_1^{c_1}) = d_-^{b}(x). \]
	A similar argument shows that 
	\[ c_2 = d_-^{p_2}(p_2^{c_2}) = d_-^{b}(x)\]
	if $c_2 \leq c_1$. It is clear that, in either case, $x$ has finite base $b$ representation.
\end{proof} 
\begin{corollary}
	\label{cor:mineq}
	Let $P = \{p_r\}_{r=1}^{k}$ be a non-empty, finite set of primes. Let $\bP$ be the product of all primes in $P$.	Suppose that $x$ is a non-zero rational number which is a product
	\[ x = y \prod_{r=1}^k p_r^{c_r} \]
	for some integer $y$, such that no prime in $P$ divides $y$. Then, 
	\begin{align*}
	d_-^{b}(x) &= \min_{1 \leq r \leq k}  \left \{ { c_r} \right \} = \min_{1 \leq r \leq k}  \left \{ - \log_{p_r} |x|_{p_r}  \right \}  = \min_{1 \leq r \leq k}  \left \{ d_-^{p_r}(p_r^{c_r})   \right \}.
	\end{align*}
	and $x$ has a finite base $b$ representation.
\end{corollary}
\begin{proof}
	Apply Lemma \ref{lem:chgbasedm} recursively.
\end{proof}

\begin{corollary}
	\label{cor:dmkasfkks}
	Let $P = \{p_r\}_{r=1}^{k}$ be a non-empty, finite set of primes. Let $\bP$ be the product of all primes in $P$.	Suppose that $x$ is a non-zero rational number which is a product
	\[ x = y \prod_{r=1}^k p_r^{c_r} \]
	for some integer $y$, such that no prime in $P$ divides $y$. If $d_-^{\bP}(x) \leq 0$, then
	\begin{equation} |d_-^{\bP}(x)| \leq \sum_{p \in \mathcal{P}^*} \ln^+ |x|_{p} \label{eqn:dkdkdkkdk} \end{equation}
	where $\mathcal{P}^*$ is the set of all primes. In particular, if $d_+^{\bP}(x) \leq 0$, then
	$d_-^{\bP}(x) \leq 0$ and \[ |d_+^{\bP}(x)| \leq |d_-^{\bP}(x)|, \] so Equation \eqref{eqn:dkdkdkkdk} holds with $d_+^{\bP}(x)$ in the place of $d_-^{\bP}(x)$.
\end{corollary}
\begin{proof}
	By the previous corollary, there is a natural number $r \leq k$, such that
	\begin{align*}
	|d_-^{\bP}(x)| = -d_-^{\bP}(x) = \log_{p_r} |x|_{p_r} \leq \frac{\ln^+ |x|_{p_r}}{\ln p_r} \leq \ln^+ |x|_{p_r} \leq \sum_{p \in \mathcal{P}^*} \ln^+ |x|_{p},
	\end{align*}
	because the summand $\ln^+ |x|_{p}$ is zero whenever $p$ is a prime larger than $p_r$.
\end{proof}
\begin{lemma}
	\label{lem:dpposlogineq}
	Let $P = \{p_r\}_{r=1}^{k}$ be a non-empty, finite set of primes. Let $\bP$ be the product of all the primes in $P$. Suppose that $x$ is a non-zero rational number and that $d_+^{\bP}(x) \geq 0$. Then,
	\[|d_+^{\bP}(x)| \leq \log_{\bP} (e)  \ln^+ |x|  \label{eqn:fjfjjfj} \]
	for every $p$ in $P$. In particular, if $d_-^{\bP}(x) \geq 0$, then $d_+^{\bP}(x) \geq 0$ and \[ |d_-^{\bP}(x)| \leq |d_+^{\bP}(x)|, \] so Equation \eqref{eqn:fjfjjfj} holds with $d_-^{\bP}(x)$ in the place of $d_+^{\bP}(x)$.
\end{lemma}
\begin{proof}
	By consideration of decimal representations, we have 
	\[ |d_+^{\bP}(x)| = d_+^{\bP}(x) \leq \log_{\bP} |x| \leq \log^+_{\bP} |x| = \log_{\bP} (e)  \ln^+ |x|\]
	The right hand equality is just the change of base formula.
\end{proof}
\begin{proposition}
	%Perhaps use J and K?
	\label{prop:asjkasjdhjre}
	Let $P = \{p_r\}_{r=1}^{k}$ be a non-empty, finite set of primes. Let $\bP$ be the product of all primes in $P$.	Suppose that $f$ is an element of $\fgseq{n}{P}$. Then, there are positive, real constants $K$ and $M$, such that 
	\[  \maxonedmdp{\mti{f}{ij}}{\bP} \leq K + M \sum_{p \in \mathcal{P}} \ln^+ \| f \|_p,  \]
	for every pair of natural numbers $i$ and $j$, such that $i < j  \leq n$, where $\mathcal{P} = \mathcal{P}^* \cup \{ \infty \}$ and $\mathcal{P}^*$ is the set of all primes.
\end{proposition}
\begin{proof}
	First, we relate $\maxonedmdp{\mti{f}{ij}}{\bP}$ to $ \ln^+ \| f \|_{p}$ for each $p$ in $\mathcal{P}$. It follows from the max norm inequality,  Equation (2.3.8) in Golub \cite{golub1996matrix}, gives that
	\begin{equation}
	\ln^+ |\mti{f}{ij}| \leq \ln \sqrt{n} +  \ln^+ \| f \|_\infty. \label{eqn:ntrjerje}
	\end{equation}
	Let $p$ be any prime. Choose the vector $w = (w_1, \dots, w_n)$ in $\Qp{p}^n$, such that
	\[ w_t = \begin{cases} p^{-1} &\textrm{if} \ t=j, \\ 0 \ &\textrm{otherwise.} \end{cases}  \]
	for every natural number $t \leq n$. Then, $|w|_p = 1$, so
	\begin{align}
	\| f \|_p &= \sup_{|v|_p = 1} |f v|_p \nonumber \\
	&\geq \sup_{|v|_p = 1} \max_s \left |\sum_{t=1}^n f_{st} v_t \right |_p \nonumber \\
	&\geq \sup_{|v|_p = 1} \left |\sum_{t=1}^n f_{it} v_t \right |_p \nonumber \\
	&\geq \left |\sum_{t=1}^n f_{it} w_t \right |_p \nonumber \\
	&= \ln^+ \left | f_{ij} p^{-1} \right |_p \nonumber \\
	&\geq \ln^+ \left | f_{ij} \right |_p. \label{eqn:dfgfdg}
	\end{align}

	On the other hand, from the definition of $\maxonedmdp{.}{\bP}$,
	\[ \maxonedmdp{\mti{f}{ij}}{\bP} \leq 1 + |d_-^{\bP}(\mti{f}{ij})| + |d_+^{\bP}(\mti{f}{ij})|. \]
	We now seek an upper bound on the right hand side of this equation. There are four cases to consider, depending on the signs of $d_-^{\bP}(\mti{f}{ij})$ and $d_-^{\bP}(\mti{f}{ij})$. If they are both non-positive, then we can use Corollary \ref{cor:dmkasfkks}. Similarly, if they are both non-negative, then we can use Lemma \ref{lem:dpposlogineq}.	If $d_+^{\bP}(\mti{f}{ij})$ is non-negative and  $d_+^{\bP}(\mti{f}{ij})$ is non-positive, then we can use the combination of those lemmas. The final case, where $d_+^{\bP}(\mti{f}{ij}) < d_-^{\bP}(\mti{f}{ij})$,  is impossible by definition! In any case, we have
	\begin{equation}
	\maxonedmdp{\mti{f}{ij}}{\bP} \leq 1 + \log_{\bP} (e) \sum_{p \in \mathcal{P}} \ln^+ |\mti{f}{ij}|_{p}, \label{eqn:asdsasdad}
	\end{equation}
	Substituting equations \eqref{eqn:ntrjerje} and \eqref{eqn:dfgfdg} into Equation \eqref{eqn:asdsasdad}  gives
	\[  \maxonedmdp{\mti{f}{ij}}{\bP} \leq K + M \sum_{p \in \mathcal{P}} \ln^+ \| f \|_p,  \]
	where \[ K = 1 + \log_{\bP} (e) + \log_{\bP} (e) \ln \sqrt{n} \] and $M = \log_{\bP} (e)$.
\end{proof}

\begin{lemma}
	\label{lem:cibnd}
	Let $P = \{p_r\}_{r=1}^{k}$ be a finite set of primes. Let $r$ be a natural number less than or equal to $n$. Let $f$ be an element of $\fgseq{n}{P}$.  Let  $\mti{r}{ii}$ and $\msi{c}{r}{ii}$ be integers, as in Lemma \ref{lem:fgequivalence}, such that
	\[ \mti{f}{ii} =  \prod_{r=1}^k p_r^{\msi{c}{r}{ii}} \]
	and such that each $r_{ij}$ is not divisible by any of the primes in $P$. Then,
	\[  \sum_{r=1}^k  |\msi{c}{r}{ii}|  \leq  \frac{3 \ln n}{2 \ln 2}   +  \frac{3}{\ln 2} \sum_{p \in \mathcal{P}} \ln^+ \| f \|_p, \]
	where $\mathcal{P} = \mathcal{P}^* \cup \{ \infty \}$ and $\mathcal{P}^*$ is the set of all primes.
\end{lemma}
\begin{proof}
	We begin by observing that
	\begin{align*}
	\ln^+ |f| &\geq \ln f \\
	&= \ln \left ( \prod_{r=1}^k p_r^{\msi{c}{r}{ii}} \right ) \\
	%	&= \sum_{r=1}^k \ln \left (  p_r^{\msi{c}{r}{ii}} \right ) \\
	&= - \sum_{r=1}^k \ln \left (  p_r^{-\msi{c}{r}{ii}} \right ) \\
	&= - \sum_{r=1}^k \ln \left |  \mti{f}{ii} \right |_{p_r}.
	\end{align*}
	Since $\ln^+ |\mti{f}{ii}|_p$ is zero for all primes not in $P$,
	\begin{align*}
	\sum_{p \in \mathcal{P}} \ln^+ |\mti{f}{ii}|_p = \ln^+ |f| + \sum_{r=1}^k \ln^+ \left |  \mti{f}{ii} \right |_{p_r} = - \sum_{r=1}^k \ln^{-} \left |  \mti{f}{ii} \right |_{p_r} \geq - \sum_{r=1}^k \ln \left |  \mti{f}{ii} \right |_{p_r},
	\end{align*}
	where $\ln^-$ is the negative part of the natural logarithm function. Now, 	
	\begin{align*}
	- \sum_{r=1}^k \ln \left |  \mti{f}{ii} \right |_{p_r} 
	= \sum_{r=1}^k \msi{c}{r}{ii} 
	= \sum_{r=1}^k |\msi{c}{r}{ii}| - 2\sum_{\substack{1\leq r \leq k \\ \textrm{s.t.}  \ \msi{c}{r}{ii} < 0}} |\msi{c}{r}{ii}|
	\end{align*}
	Rearranging the last expression, we have
	\begin{align}
	\sum_{r=1}^k |\msi{c}{r}{ii}| \leq \sum_{p \in \mathcal{P}} \ln^+ |\mti{f}{ii}|_p + 2\sum_{\substack{1\leq r \leq k \\ \textrm{s.t.}  \ \msi{c}{r}{ii} < 0}} |\msi{c}{r}{ii}|. \label{eqn:gfruiewhufdjhsfd}
	\end{align}
	Notice that $p^{-\msi{c}{r}{ii}} = |\mti{f}{ii}|_{p_r}$. So $-\msi{c}{r}{ii} \ln (p_r ) = \ln |\mti{f}{ii}|_{p_r}$. Hence, if $\msi{c}{r}{ii}$ is negative,
	\[ \frac{\ln |\mti{f}{ii}|_{p_r}}{\ln 2}  \geq |\msi{c}{r}{ii}| = -\msi{c}{r}{ii}, \]
	as $2$ is the smallest prime. Returning to Equation \eqref{eqn:gfruiewhufdjhsfd}, we have
	\[ \sum_{r=1}^k |\msi{c}{r}{ii}| \leq \frac{3}{\ln 2} \sum_{p \in \mathcal{P}} \ln^+ |\mti{f}{ii}|_p \]
	The same argument in Proposition \ref{prop:asjkasjdhjre} used to arrive at equations \eqref{eqn:ntrjerje} and \eqref{eqn:dfgfdg} gives 
	\[ \ln^+ \left | f_{ii} \right |_p \leq	\ln^+ \| f \|_p, \]
	for every prime $p$ and
	\[ \ln^+ |\mti{f}{ii}| \leq \ln \sqrt{n} +  \ln^+ \| f \|_\infty, \]
	respectively. Hence, 
	\[ \sum_{r=1}^k  |\msi{c}{r}{ii}|  \leq  \frac{3}{\ln 2} \sum_{p \in \mathcal{P}} \ln^+ |\mti{f}{ii}|_p \leq \frac{3 \ln n}{2 \ln 2}  +  \frac{3}{\ln 2} \sum_{p \in \mathcal{P}} \ln^+ \| f \|_p, \]
	as desired.
\end{proof}

\subsection{Lower bounds}
\label{subsec:lb}

\begin{lemma}\label{lem:maxsupineq}
	Let $P = \{p_r\}_{r=1}^{k}$ be a finite set of primes. Let $f$ be an element of $\fgseq{n}{P}$.  Let  $\mti{r}{ij}$ and $\msi{c}{r}{ij}$ be integers, as in Lemma \ref{lem:fgequivalence}, such that
	\[ \mti{f}{ij} = 
	\begin{cases} \prod_{r=1}^k p_r^{\msi{c}{r}{ij}}
	& \textrm{if} \ i = j, \ \textrm{and} 
	\\ 
	\mti{r}{ij} \prod_{r=1}^k p_r^{\msi{c}{r}{ij}}
	&\textrm{if} \ i < j. \end{cases}\]
	and such that each $r_{ij}$ is not divisible by any of the primes in $P$. 
	Then,
	\[ \ln^+ \brofopnorm{f}{p_r}   \leq \max_{1 \leq i \leq j \leq n} \left \{  |\msi{c}{r}{ij}| \ln^+ \left ( p_r\right ) \right \} \]
	for all natural numbers $r$ less than or equal to $k$. 
\end{lemma}
\begin{proof}
	Let $r$ be a natural number. Suppose that $v = (v_1, \dots, v_n)$ is a vector in $\Qp{p_r}$ with $|v|_{p_r} = 1$. The ultrametric inequality and the definition of the $p_r$-adic norm imply that
	\begin{align*}
	|f v|_{p_r} &\leq \max_{1 \leq i \leq j \leq n} \left |p_r^{\msi{c}{r}{ij}} v_j \right |_{p_r} = \max_{1 \leq i \leq j \leq n} \left | p_r^{\msi{c}{r}{ij}} \right |_{p_r} \left | v_j \right |_{p_r}. 
	\end{align*}
	If $|v|_p = 1$, then
	\[
	|f v|_{p_r} \leq \max_{1 \leq i \leq j \leq n} \left | p_r^{\msi{c}{r}{ij}} \right |_{p_r} 
	= \max_{1 \leq i \leq j \leq n} p_r^{- \msi{c}{r}{ij}} 
	\leq \max_{1 \leq i \leq j \leq n} p_r^{\left | \msi{c}{r}{ij} \right | }. \]
	Since $p_r^{\left | \msi{c}{r}{ij} \right |} $ is always greater than or equal to $1$, we have
	\[ \ln^+ \brofopnorm{f}{p_r} = \ln^+  \left ( \sup_{|v|_p = 1} |fv|_p \right ) \leq
	\ln^+  \left ( \max_{1 \leq i \leq j \leq n} p_r^{\left | \msi{c}{r}{ij} \right | } \right ) =
	\max_{1 \leq i \leq j \leq n} \left \{  |\msi{c}{r}{ij}| \ln^+ \left ( p_r\right ) \right \} \] 
	Which concludes the proof.

\end{proof}
\begin{proposition}
	\label{prop:metricestimatebound}
	Let $f$ be an element of $\fgseq{n}{P}$ for a non-empty, finite set of primes $P = \{p_r\}_{r=1}^{k}$. Then,
	\[ \sum_{p \in P} \left ( \ln^+ \brofopnorm{f}{p}  + \ln^+  \brofopnorm{f^{-1}}{p} \right ) \leq K \mefgseq{f}{n}{P}  \]
	where $K$ is a positive, real constant.
\end{proposition}
\begin{proof}
	The claim is certainly true if $f$ is the identity $I_n$. Let $\mti{r}{ij}$ and $\msi{c}{r}{ij}$ be integers, as in Lemma \ref{lem:fgequivalence}, such that
	\[ \mti{f}{ij} = 
	\begin{cases} \prod_{r=1}^k p_r^{\msi{c}{r}{ij}}
	& \textrm{if} \ i = j, \ \textrm{and} 
	\\ 
	\mti{r}{ij} \prod_{r=1}^k p_r^{\msi{c}{r}{ij}}
	&\textrm{if} \ i < j. \end{cases}\]
	and such that each $r_{ij}$ is not divisible by any of the primes in $P$. 
	
	Since $f$ is not the identity, we may apply Lemma \ref{lem:maxsupineq},
	\begin{align*}
	\sum_{r=1}^k \ln^+ \brofopnorm{f}{p_r}  &= \sum_{r=1}^k \max_{1 \leq i \leq j \leq n} \left \{  |\msi{c}{r}{ij}| \ln^+ \left ( p_r\right ) \right \} \\ 
	&\leq \sum_{r=1}^k \left ( \ln(p_r)  \sum_{i=1}^n | \msi{c}{r}{ii} | \right ) + \sum_{1 \leq i < j \leq n} \sum_{r=1}^k   |\msi{c}{r}{ij}| \ln^+ \left ( p_r\right ).
	\end{align*} 
	Let $M_{ij} = -\max_{r=1}^{k} \left \{  |\msi{c}{r}{ij} | \right \}$. Then, using Corollary \ref{cor:dminuscorollarybound}, 
	\begin{align}
	\sum_{r=1}^k  |\msi{c}{r}{ij}| \ln^+ \left ( p_r\right ) &\leq \sum_{r=1}^k |M_{ij}| \ln^+ \left ( p_r\right ) \nonumber \\
	&= |M_{ij}| \sum_{r=1}^k \ln \left ( {p_r} \right ) \nonumber \\
	&\leq |d_-^\bP(f_{ij})| \sum_{r=1}^k \ln \left ( {p_r} \right ). \label{eq:lnpdmineq}
	\end{align} 
	Hence, we have, 
	\begin{align*}
	\sum_{r=1}^k \ln^+ \brofopnorm{f}{p_r} 
	&\leq \left ( \sum_{r=1}^k \ln(p_r) \right ) \left (  \sum_{i=1}^n \sum_{r=1}^k | \msi{c}{r}{ii} | + \sum_{1 \leq i < j \leq n}  |d_-^\bP(f_{ij})| \right ) \\
	&\leq \left ( \sum_{r=1}^k \ln(p_r) \right ) \left (  \sum_{i=1}^n \sum_{r=1}^k | \msi{c}{r}{ii} | + \sum_{1 \leq i < j \leq n} \maxonedmdp{\mti{f}{ij}}{\bP} \right ), \\
	&= \left ( \sum_{r=1}^k \ln(p_r) \right ) \mefgseq{f}{n}{P}.
	\end{align*}
	where $\bP$ is the product of all the primes in $P$. Since $\mefgseq{\cdot}{n}{P}$ is a word metric estimate, 
	\[  \frac{1}{J} \mefgseq{f^{-1}}{n}{P}  \leq \wlfgseq{f^{-1}}{n}{P}  = \wlfgseq{f}{n}{P} \leq J \mefgseq{f}{n}{P}, \]
	where $J$ is the positive, real constant we found in Proposition \ref{prop:rationalmetricestimate}. This yields the desired inequality, namely,
	\[ \sum_{p \in P} \left ( \ln^+ \brofopnorm{f}{p} + \ln^+  \brofopnorm{f^{-1}}{p}  \right ) \leq K \mefgseq{f}{n}{P}  \]
	where $K =  \left ( \sum_{r=1}^k \ln(p_r) \right )  \left (1 + J^2 \right )$.
\end{proof}
\begin{lemma}
	\label{lem:dminustodiag}
	Let $P = \{p_r\}_{r=1}^{k}$ be a finite set of primes. If $f$ is a rational number equal to a product of powers of primes in $P$, \[ f = \prod_{r=1}^k p_r^{\ms{c}{r}}, \]
	then, \[ 
	d_+^\bP (f) \leq 1 + \sum_{r=1}^k | \ms{c}{r} |,  \]
	where $\bP$ is the product of all the primes in $P$.
\end{lemma}
\begin{proof}
	We begin by noting that 
	\begin{align*}
	d_+^\bP (f) &\leq d_+^\bP \left ( \prod_{r=1}^k  p_r^{|\ms{c}{r}|} \right ) \\
	&\leq 1 + \ln_\bP \left ( \prod_{r=1}^k  p_r^{|\ms{c}{r}|} \right ).
	\end{align*}
	Application of logarithm laws yields, 
	\begin{align*}
	d_+^\bP (f) &= 1 + \sum_{r=1}^k  {|\ms{c}{r}|} \ln_\bP \left (  p_r \right ).
	\end{align*}
	Since $\bP$ is the product of all the primes in $P$, $ \ln_\bP \left (  p_r \right ) \leq 1 $ for each $r$ in the above summation. So, 
	\begin{align*}
	d_+^\bP (f) &\leq 1 + \sum_{r=1}^k | \ms{c}{r} |.
	\end{align*}
	Which concludes the proof.
\end{proof}
\begin{proposition}
	\label{prop:metricestimateboundR}
	Let $f$ be an element of $\fgseq{n}{P}$ for a non-empty, finite set of primes $P = \{p_r\}_{r=1}^{k}$. Then,
	\[  \ln^+ \brofopnorm{f}{\infty}  + \ln^+ \brofopnorm{f^{-1}}{\infty}  \leq L + M \mefgseq{f}{n}{P}  \]
	where $L$ and $M$ are positive, real constants.
\end{proposition}
\begin{proof}
	The max norm inequality, Equation (2.3.8) in Golub \cite{golub1996matrix}, gives that
	\begin{align*}
	\ln^+ \brofopnorm{f}{\infty}  &\leq \ln \sqrt{n} + \ln \left ( \max_{1 \leq i \leq j \leq n} |f_{ij}| \right ), \intertext{where $|\cdot|$ is the ordinary absolute value. Hence,}
	\ln^+ \brofopnorm{f}{\infty}  &\leq \ln \sqrt{n} + \ln (\bP) \left ( 1 + d_+^\bP \left ( \max_{1 \leq i \leq j \leq n} |f_{ij}| \right ) \right ) \\
	&\leq \ln \sqrt{n} + \ln (\bP) \left (
	\sum_{i=1}^n  \left ( d_+^\bP (|f_{ii }|) \right ) + \sum_{1\leq i < j \leq n}  \left ( d_+^\bP (|f_{ij}|) \right ) \right ). \intertext{Applying Lemma \ref{lem:dminustodiag}, we have }\\
	\ln^+ \brofopnorm{f}{\infty} &\leq \ln \sqrt{n} + \ln (\bP) \left ( n +
	\sum_{i=1}^n \sum_{r=1}^k | \msi{c}{r}{ii} | + \sum_{1\leq i < j \leq n}  \left ( d_+^\bP (|f_{ij}|) \right ) \right ) \\
	&\leq \frac{L}{2} + \ln (\bP) \mefgseq{f}{n}{P},
	\end{align*}
	where $L = \ln n + 2\ln (\bP)$. Let $J$ be as in Proposition \ref{prop:rationalmetricestimate}. Then, since $\mefgseq{f}{n}{P}$ is a word metric estimate,
	\begin{align*}
	\ln^+ \brofopnorm{f^{-1}}{\infty} &\leq \frac{L}{2}  +  \mefgseq{f}{n}{P} \\
	&\leq \frac{L}{2}  + J \wlfgseq{f^{-1}}{n}{P} \\
	&= \frac{L}{2}  +  J \wlfgseq{f}{n}{P} \\
	&\leq \frac{L}{2}  +  J^2 \mefgseq{f^{-1}}{n}{P}.
	\end{align*}
	Hence,
	\[ \ln^+ \brofopnorm{f}{\infty} + \ln^+  \brofopnorm{f^{-1}}{\infty} \leq L + M \mefgseq{f}{n}{P}  \]
	where $M = 1 + J^2$, which concludes the proof.
\end{proof}
\begin{lemma}
	\label{lem:adeliclengthsimplification}
	Let $P = \{p_r\}_{r=1}^{k}$ be a non-empty, finite set of primes. Let $\mu$ be a probability measure on $\fgseq{n}{P}$ that has finite first moment with respect to word length. Then, 
	\[ \ln^+ \brofopnorm{f}{p} + \ln^+  \brofopnorm{f^{-1}}{p}  \] is zero unless $p$ is in $P$, or $p = \infty$.
\end{lemma}
\begin{proof}
	Suppose $p$ is any prime not in $P$. By Lemma \ref{lem:fgequivalence},  there are integers $\mti{r}{ij}$ and $\msi{c}{r}{ij}$ for natural numbers $r$ less than or equal to $k$, such that
	\begin{align}
	\mti{f}{ij} = 
	\begin{cases} \prod_{r=1}^k p_r^{\msi{c}{r}{ij}}
	& \textrm{if} \ i = j, \ \textrm{and} 
	\\ 
	\mti{r}{ij} \prod_{r=1}^k p_r^{\msi{c}{r}{ij}}
	&\textrm{if} \ i < j \end{cases} \label{eq:decrep}
	\end{align}
	and the integers $r_{ij}$ are not divisible by any of the primes in $P$.  
	
	Suppose that $v = (v_1, \dots, v_n)$ is a vector in $\Qp{p_r}$ with $|v|_p = 1$. The ultrametric inequality and the definition of the $p$-adic norm imply that
	\begin{align*}
	|f v|_{p_r} &\leq \max_{1 \leq i \leq j \leq n} \left | f_{ij} v_j \right |_{p} \\
	&= \max_{1 \leq i \leq j \leq n} \left | f_{ij} \right |_{p} \left | v_j \right |_{p} \\
	&\leq \max_{1 \leq i \leq j \leq n} \left | f_{ij} \right |_{p}. 
	\end{align*}
	By considering Equation \eqref{eq:decrep}, we see that $\left | f_{ij} \right |_{p}$ is less than or equal to $1$. Hence,
	\[ \brofopnorm{f}{p} = \sup_{|v|_p = 1 } |f v|_{p_r} \leq 1, \]
	which implies that $\ln^+ \brofopnorm{f}{p} $ is zero. The argument to show that $\ln^+ \brofopnorm{f^{-1}}{p} $  is zero is similar, since $\fgseq{n}{P}$ is a group.
\end{proof}
\subsection{Conclusions}
We will now prove that the adelic length is a word metric estimate.
\begin{proposition}
	\label{prop:ame}
	Let $P = \{p_r\}_{r=1}^{k}$ be a finite set of primes. Suppose that $f$ is an element of $\fgseq{n}{P}$. Then	
	\[ Q + R \adfgseq{f}{n} \leq \wlfgseq{f}{n}{P} \leq S + T \adfgseq{f}{n}, \]
	for some positive, real constants $Q, R, S$ and $T$.
	
\end{proposition}
\begin{proof}
	We will first prove the upper bound using statements in Subsection \ref{subsec:lb}. Proposition \ref{prop:rationalmetricestimate} states that there is a positive, real constant $J$, such that
	\begin{align*}
	\wlfgseq{f}{n}{P}   &\leq J \mefgseq{f}{n}{P}  \\
	&= J \sum_{i=1}^n \sum_{r=1}^k  |\msi{c}{r}{ii}|  + J \sum_{1 \leq i < j \leq n} \maxonedmdp{\mti{f}{ij}}{\bP}.
	\end{align*}
	Lemma \ref{lem:cibnd} gives that
	\[  \sum_{r=1}^k  |\msi{c}{r}{ii}|  \leq  \frac{3 \ln n}{2 \ln 2}  +  \frac{3}{\ln 2} \sum_{p \in \mathcal{P}} \ln^+ \| f \|_p. \]
	Since there are $\frac{(n-1)(n-2)}{2}$ upper off-diagonal entries in $f$, Proposition \ref{prop:asjkasjdhjre} gives,
	\[ \sum_{1 \leq i < j \leq n} \maxonedmdp{\mti{f}{ij}}{\bP} \leq \frac{(n-1)(n-2)}{2} \left (  K + M \sum_{p \in \mathcal{P}} \ln^+ \| f \|_p \right ) \]
	for positive, real constants $K$ and $M$. Therefore, if we let 
	\[ S = JK \frac{(n-1)(n-2)}{2} + J \frac{ 3 \ln n}{2 \ln 2}    \]
	and let
	\[ T =  J\frac{3}{\ln 2}  + JM \frac{(n-1)(n-2)}{2} ,  \]
	then,
	\[ \wlfgseq{f}{n}{P} \leq S + T \adfgseq{f}{n}, \]
	which is the upper bound.
	
	Now we will prove the lower bound, making use of the results in Subsection \ref{subsec:ub}. By definition, the adelic length of $f$ is  
	\[ \adfgseq{f}{n} = \sum_{p \in \mathcal{P}} \left ( \ln^+ \brofopnorm{f}{p} + \ln^+  \brofopnorm{f^{-1}}{p} \right ), \]
	where $\mathcal{P} = \mathcal{P}^* \cup \{ \infty \}$ and $\mathcal{P}^*$ is the set of all primes. By Lemma \ref{lem:adeliclengthsimplification}, 
	\[ \ln^+ \brofopnorm{f}{p} + \ln^+  \brofopnorm{f^{-1}}{p}  \] is zero unless $p$ is in $P$ or $p = \infty$. Hence,
	\[ \adfgseq{f}{n} = \sum_{p \in P \cup \{ \infty \}} \left ( \ln^+ \brofopnorm{f}{p} + \ln^+  \brofopnorm{f^{-1}}{p} \right ). \]
	By \ref{prop:metricestimatebound}, there is a positive, real constant $K$, such that 
	\[ \sum_{p \in P} \left ( \ln^+ \brofopnorm{f}{p}  + \ln^+  \brofopnorm{f^{-1}}{p} \right ) \leq K \mefgseq{f}{n}{P}  \]
	and by Proposition \ref{prop:metricestimateboundR}, there are positive, real constants $L$ and $M$, such that
	\[  \ln^+ \brofopnorm{f}{\infty}  + \ln^+ \brofopnorm{f^{-1}}{\infty}  \leq L + M \mefgseq{f}{n}{P}.  \]
	As before, let $J$ be as in Proposition \ref{prop:rationalmetricestimate}. Then, taking $Q = \frac{L}{J}$ and $R = \frac{K+M}{J}$,
	\[ Q + R \adfgseq{f}{n} \leq \wlfgseq{f}{n}{P},  \]
	which completes the proof.
\end{proof}
\begin{corollary}
	Let $P = \{p_r\}_{r=1}^{k}$ be a non-empty, finite set of primes. Let $\mu$ be a probability measure on $\fgseq{n}{P}$. Then, $\mu$ has finite first moment with respect to word length if and only if it has finite first moment with respect to adelic length.
\end{corollary}
\begin{proof}
	Suppose that $\mu$ has finite first moment with respect to adelic length. The measure $\mu$ has finite first moment with respect to word length if 
	\[ \int_{\fgseq{n}{P}} \wlfgseq{f}{n}{P} \, d \mu (f) \]
	is finite. By Proposition \ref{prop:ame}, this quantity is bounded above by 
	\[  S + T \int_{\fgseq{n}{P}} \adfgseq{f}{n}  \, d \mu (f) \]
	which is finite by assumption. The converse is similar.
\end{proof}
As mentioned in Corollary \ref{cor:subgrp}, any finitely generated upper triangular matrix group, $H$ with rational entries is a subgroup of $\fgseq{n}{P}$ for suitable $n$ and $P$. Translating the first moment condition to the word length taken with respect to the generating set on the subgroup $H$ is then mostly a matter of computing the distortion of $H$ in $\fgseq{n}{P}$. See Davis and Olshanskii \cite{davis2012relative}.

\begin{example}[Affine group over the dyadic rationals]
	Denote by $\Aff(\pyad{2}) $ the affine group over the dyadic rationals,
	\[  \Aff(\pyad{2}) = \left \{ \begin{pmatrix} 2^n & f \\ 0 & 1 \end{pmatrix} : n \in \I, f \in \pyad{2}  \right \}. \]
	The group $\Aff(\pyad{p})$ is isomorphic to the solvable Baumslag-Solitar group $BS(1,p)$, see  McLaury \cite{mclaury12}. The group is a subgroup of $FG_n(2)$ for any $n$ greater than or equal to $2$ and is generated by the elements $\theta_{11}(2)$ and $\theta_{12}(1)$. Hence, a probability measure $\mu$ on $\Aff(\pyad{2}) $  has finite first moment with respect to word length if and only if it has finite first moment with respect to adelic length. 
\end{example}

\begin{example}
	According to Burillo and Plat{\'o}n \cite{burillo2014metric}, The map, $E$, from $\operatorname{UT}_n (\I)$ to the natural numbers given by 
	\[ E(\mt{f}) = \sum_{1 \leq i < j \leq n} |\mti{f}{ij} |^{\frac{1}{j-i}} \]
	is a word metric estimate on that group. If $f$ is a matrix in $\operatorname{UT}_n (\I)$, then
	\begin{align*}
	\ln E(\mt{f}) 
	&\leq \sum_{\substack{ 1 \leq i < j \leq n \\ |\mti{f}{ij} | \neq 0}} \frac{1}{j-i} \ln \left ( |\mti{f}{ij} | \right )\\
	&\leq \sum_{\substack{ 1 \leq i < j \leq n \\ |\mti{f}{ij} | \neq 0}} \left ( 1  + d_+ \left ( \mti{f}{ij}  \right ) \right ) \\
	&= \mefgseq{f}{n}{P}.
	\end{align*}
	Since both $\mefgseq{\cdot}{n}{P}$ and $\adfgseq{\cdot}{n}$ are word metric estimates, a probability measure $\mu$ on $\operatorname{UT}_n (\I)$ has finite first logarithmic moment with respect to word length if it has finite first moment with respect to adelic length.
\end{example}

\section{Formulae for products and inverses}
\label{sec:prelim}
If $P$ is a singleton, then $\fgseq{n}{P}$ is a semi-direct product,  $\I^n \ltimes \UT_n (\pyad{p})$, where $\UT_n (\pyad{p})$ is the group of upper unitriangular matrices with entries in $\pyad{p}$. In this section, we give formulae for the multiplication of many elements and computation of inverses for semi-direct products of this form.

Let $n$ be a natural number and let $p$ be a fixed prime.  Denote by $\mathbb{F}$ the field of real numbers, $p$-adic numbers $\Qp{p}$, rational numbers $\mathbb{Q}$, or the ring
\[ \pyad{p} = \left \{ \frac{f}{p^m} : f,m \in \I \right \}.\] 
Let $A_n(\mathbb{F})$ be the group of invertible upper triangular matrices of dimension $n$ over $\mathbb{F}$ and let $UT_n(\mathbb{F})$ be the subgroup of all upper unitriangular matrices in $A_n(\mathbb{F})$.

Unless otherwise noted, let $G$ be the group $\I^n \ltimes A_n(\mathbb{F})$, where the product of two elements $(\md{x}, \mt{f})$ and $(\md{y}, \mt{g})$ in $G$ is given by
\[ (\md{x}, \mt{f})(\md{y}, \mt{g}) = \left ( \md{x} \md{y}, \mt{f} \zeta_\md{x} (\mt{g}) \right ), \]
where $\zeta_\md{x} (y) = XyX^{-1}$,
\[ X = \begin{pmatrix} p^{x_1} & 0 & \cdots & 0 \\
0  & p^{x_2} & \cdots & 0 \\
\vdots  & \vdots & \ddots  & \vdots \\
0  & 0 & \cdots  & p^{x_n}
\end{pmatrix} \] 
and $\md{x} = (x_1, \dots, x_n )$.

The identity element, $e$, in $G$ is $(\mathbf{0}, I_n)$, where $\mathbf{0}$ is the zero vector and $I_n$ is the $n \times n$ identity matrix. 

The product of $m$ elements, 
\begin{align}
(\mds{y}{m},\ms{\varphi}{m}) &= \prod_{i=1}^m  \left ( \mds{x}{i}, \ms{f}{i} \right ) = \left ( \mds{x}{1}, \ms{f}{1} \right ) \dots \left ( \mds{x}{m}, \ms{f}{m} \right ) 
\end{align}
is given by the relations
\[
\mds{y}{m} = \sum_{r=1}^m \mds{x}{m}, \ \textrm{and} \ 
\ms{\varphi}{m}  = \ms{\varphi}{m-1} \zeta_{\mds{y}{m-1}} \left ( \ms{f}{m} \right ), \label{eqn:semidirectrecurrence}
\]
where $\ms{\varphi}{0}$ and $\mds{y}{0}$ are taken to be the identity and $\mds{y}{m} = \left ( \mdsi{y}{m}{1}, \dots, \mdsi{y}{m}{n} \right )$. It follows from the definition of matrix multiplication that
\begin{equation}
\msi{\varphi}{m}{ij} = \sum_{k=1}^n \msi{\varphi}{m-1}{ik} \msi{f}{m}{kj} p^{\mdsi{y}{m-1}{k} - \mdsi{y}{m-1}{j}},
\end{equation}
whenever $m$ is greater than or equal to $2$. Expanding this recurrence relation,
\begin{equation}
\msi{\varphi}{m}{i,i+r} = \sum_{\{a_k\} \in Q(m,r) } \prod_{n=1}^{m} \msi{f}{n}{i+a_n,i+a_{n+1}} p^{\mdsi{y}{n-1}{i + a_n} - \mdsi{y}{n-1}{i + a_{n+1}}}. \label{eqn:semidirectrecurrenceexpansion}
\end{equation}
where $m$ is greater than or equal to $r$ and $Q(l,s)$ is the set of all integer sequences $\{ a_1, \dots, a_l \}$, of length $l$ which are non-strictly monotone, start at $0$ and finish at $s$. The case where $r = 0$ and $m$ is any natural number is easily verified. It is similarly easy to show for the case where $m = 1$ and $r$ is any appropriate natural number. The remaining cases are shown by induction. Namely, if we suppose that Equation \eqref{eqn:semidirectrecurrenceexpansion} is satisfied for all natural numbers $r$ and $m$ less than a sufficiently large natural number $q$, then
\begin{align}
\msi{\varphi}{q}{i,i+r} &= \sum_{k=1}^n \msi{\varphi}{q-1}{i,k} \msi{f}{q}{k,i+r} p^{\mdsi{y}{q-1}{k} - \mdsi{y}{q-1}{i+r}}  \nonumber \\
&= \sum_{k=i}^{i+r} \msi{\varphi}{q-1}{i,k} \msi{f}{q}{k,i+r} p^{\mdsi{y}{q-1}{k} - \mdsi{y}{q-1}{i+r}}  \nonumber \\
&= \sum_{k=0}^{r} \msi{\varphi}{q-1}{i,i+k} \msi{f}{q}{i+k,i+r} p^{\mdsi{y}{q-1}{i+k} - \mdsi{y}{q-1}{i+r}}  \nonumber \\
&= \sum_{k=0}^{r} \left ( \sum_{\{a_k\} \in Q(q-1,k) } \prod_{n=1}^{q-1} \msi{f}{n}{i+a_n,i+a_{n+1}} p^{\mdsi{y}{n-1}{i + a_n} - \mdsi{y}{n-1}{i + a_{n+1}}}  \right ) \msi{f}{q}{i+k,i+r} p^{\mdsi{y}{q-1}{i+k} - \mdsi{y}{q-1}{i+r}}  \nonumber \\
& = \sum_{\{a_k\} \in Q(q,r) } \prod_{n=1}^{q} \msi{f}{n}{i+a_n,i+a_{n+1}} p^{\mdsi{y}{n-1}{i + a_n} - \mdsi{y}{n-1}{i + a_{n+1}}}. \nonumber
\end{align}

If $G$ is instead equal to $\I^n \ltimes UT_n(\mathbb{F})$ with the same action, where $UT_n(\mathbb{F})$ is the group of all upper unitriangular matrices with entries in $\mathbb{F}$, then the entries on the diagonal of $\ms{f}{m}$ are all equal to $1$. Equation \eqref{eqn:semidirectrecurrenceexpansion} then gives
\begin{align}
\msi{\varphi}{m}{ij} &= \sum_{k=i}^{j} \msi{\varphi}{m-1}{ik} \msi{f}{m}{kj} p^{\mdsi{y}{m-1}{k} - \mdsi{y}{m-1}{j}}  \nonumber \\
&= \sum_{k=i}^{j-1} \left (  \msi{\varphi}{m-1}{ik} \msi{f}{m}{kj} p^{\mdsi{y}{m-1}{k} - \mdsi{y}{m-1}{j}} \right ) +  \msi{\varphi}{m-1}{ij}   \nonumber \\
&= \sum_{l=1}^{m} \sum_{k=i}^{j-1} \left (  \msi{\varphi}{l-1}{ik} \msi{f}{l}{kj} p^{\mdsi{y}{l-1}{k} - \mdsi{y}{l-1}{j}} \right ) \nonumber \\
&= \sum_{l=1}^{m} \sum_{k=i}^{j-1} \left (  \msi{\varphi}{l-1}{ik} \msi{f}{l}{kj} p^{\mdsi{y}{l-1}{k} - \mdsi{y}{l-1}{j}} \right ). \label{eqn:semidirectrecurrenceexpansionforindiction}
\end{align}
For sufficiently large $m$, an alternate expansion of the recurrence relation in this case is 
\begin{equation}
\msi{\varphi}{m}{i,i+r} = \sum_{a \in \MFI{r}} S_{a}^{(m)} \label{eqn:expandedrecurrence}
\end{equation}
where $\MFI{r}$  is the set of all strictly monotone increasing finite integer valued sequences, $a = \{a_n\}$, with first term equal to $0$ and last term equal to $r$ indexed by the non-negative integers and 
\begin{equation}
S_a^{(m)} = \sum_{0 \leq b_1 < \dots < b_{|a|-1} < m} \prod_{n=0}^{|a|-2} \msi{f}{b_{n+1}+1}{i+a_n,i+a_{n+1}}  p^{\mdsi{y}{b_{n+1}}{i + a_n} - \mdsi{y}{b_{n+1}}{i + a_{n+1}}}. \end{equation}
The validity of this formula may be checked by induction on Equation \eqref{eqn:semidirectrecurrenceexpansion}.

\begin{example}
	The set $\MFI{2}$ contains only the finite sequences $(0,2)$ and $(0,1,2)$. Equation \eqref{eqn:expandedrecurrence} gives
	\begin{align*}
	\msi{\varphi}{m}{i,i+2} &= \sum_{a \in \MFI{2}} S_{a}^{(m)} \\
	%    &= \sum_{0 \leq b_1 < m} \msi{f}{b_{1}+1}{i,i+2} 2^{\mdsi{y}{b_{1}}{i} - \mdsi{y}{b_{1}}{i + 2}} %a = (0,2), |a|=2
	%    +  \sum_{0 \leq b_1 < b_2 < m} \msi{f}{b_{1}+1}{i,i+1} \msi{f}{b_{2}+1}{i+1,i+2} p^{\mdsi{y}{b_{1}}{i} - \mdsi{y}{b_{1}}{i + 1}} p^{\mdsi{y}{b_{2}}{i + 1} - \mdsi{y}{b_{2}}{i + 2}}
	%    %a = (0,1,2), |a|=3
	&= \sum_{0 \leq l < m} \msi{f}{l+1}{i,i+2} p^{\mdsi{y}{l}{i} - \mdsi{y}{l}{i + 2}} %a = (0,2), |a|=2
	+  \sum_{0 \leq r < l < m} \msi{f}{r+1}{i,i+1} \msi{f}{l+1}{i+1,i+2} p^{\mdsi{y}{r}{i} - \mdsi{y}{r}{i + 1}} p^{\mdsi{y}{l}{i + 1} - \mdsi{y}{l}{i + 2}}
	%a = (0,1,2), |a|=3
	\end{align*}
	when $m$ is greater than or equal to $2$. $\MFI{3}$ contains $(0,3), (0,1,3), (0,2,3)$ and $(0,1,2,3)$. When $m$ is greater than or equal to $3$, Equation \eqref{eqn:expandedrecurrence} gives
	\begin{align*}
	\msi{\varphi}{m}{i,i+3} =& \sum_{a \in \MFI{3}} S_{a}^{(m)} \\
	%    &= \sum_{0 \leq b_1 < m} \msi{f}{b_{1}+1}{i,i+2} p^{\mdsi{y}{b_{1}}{i} - \mdsi{y}{b_{1}}{i + 2}} %a = (0,2), |a|=2
	%    +  \sum_{0 \leq b_1 < b_2 < m} \msi{f}{b_{1}+1}{i,i+1} \msi{f}{b_{2}+1}{i+1,i+2} p^{\mdsi{y}{b_{1}}{i} - \mdsi{y}{b_{1}}{i + 1}} p^{\mdsi{y}{b_{2}}{i + 1} - \mdsi{y}{b_{2}}{i + 2}}
	%    %a = (0,1,2), |a|=3
	=& \sum_{0 \leq l < m} \msi{f}{l+1}{i,i+3} p^{\mdsi{y}{l}{i} - \mdsi{y}{l}{i + 3}} %a = (0,2), |a|=2 \\
	\\
	&+  \sum_{0 \leq r < l < m} \msi{f}{r+1}{i,i+1} \msi{f}{l+1}{i+1,i+3} p^{\mdsi{y}{r}{i} - \mdsi{y}{r}{i + 1}} p^{\mdsi{y}{l}{i + 1} - \mdsi{y}{l}{i + 3}}
	\\
	&+  \sum_{0 \leq r < l < m} \msi{f}{r+1}{i,i+2} \msi{f}{l+1}{i+2,i+3} p^{\mdsi{y}{r}{i} - \mdsi{y}{r}{i + 2}} p^{\mdsi{y}{l}{i + 2} - \mdsi{y}{l}{i + 3}}
	\\
	&+  \sum_{0 \leq s < r < l < m} \msi{f}{s+1}{i,i+1} \msi{f}{r+1}{i+1,i+2} \msi{f}{l+1}{i+2,i+3} p^{\mdsi{y}{s}{i} - \mdsi{y}{s}{i + 1}} p^{\mdsi{y}{r}{i+1} - \mdsi{y}{r}{i + 2}} p^{\mdsi{y}{l}{i + 2} - \mdsi{y}{l}{i + 3}}.
	%a = (0,1,2), |a|=3
	\end{align*}
\end{example}
The inverse of an element  $(\md{x}, \mt{f})$ in $G$ is
\[ (\md{x}, \mt{f})^{-1} = \left (\mdr{\vect{x}^{-1}}, \mdr{\vect{x}}^{-1} \mt{f}^{-1} \md{x} \right ). \]
The propositions which follow give useful explicit formulae for $f^{-1}$ in the upper unitriangular and upper triangular cases.
%
%a = (0,2), |a|=2
%a = (0,1,2), |a|=3

\begin{proposition}
	\label{prop:recursiveutinverseformula}
	Suppose that $\mt{f}$ is an $n \times n$ upper unitriangular matrix. Then, for each pair of natural numbers $i$ and $s$, such that $i $ is less than $ n$ and $s $ is less than or equal to $n-i$, the inverse matrix $\mt{f^{-1}}$ satisfies the recurrence relations
	\[ (\mti{f}{i,i+1})^{-1} = \mti{-f}{i,i+1} \quad \textrm{and} \quad (\mti{f}{i,i+s})^{-1} = -\sum_{k=1}^{s}  \mti{f}{i,i+k} (\mti{f}{i+k,i+s})^{-1}. \]
	Hence we have
	\begin{equation}
	(\mti{f}{i,i+s})^{-1} = \sum_{l=1}^s (-1)^l \left ( \sum_{\{h_\zeta \} \in H(l,s)} \left (\prod_{\xi = 0}^{l-1} \mti{f}{i+h_\xi,i+h_{\xi + 1}} \right ) \right ), \label{eqn:utinverseformula}
	\end{equation}
	where $H(l,s)$ is the collection of finite sequences of integers $\left \{h_0, \dots, h_l \right \}$, such that $h_k$ is less than $h_{k+1}$, $h_1 = 0$ and $h_l = s$.
\end{proposition}
%\begin{proof}
\begin{proof}
	Consider the system of linear equations $\mt{f}\mt{f^{-1}} = I_n$. This system states that
	\[ 0 = \sum_{k=1}^n \mti{f}{ik} (\mti{f}{k,i+s})^{-1} = \sum_{k=i}^{i+s} \mti{f}{ik} (\mti{f}{k,i+s})^{-1} =  (\mti{f}{i,i+s})^{-1} + \sum_{k=i+1}^{i+s}  \mti{f}{ik} (\mti{f}{k,i+s})^{-1}  \]
	whenever $s$ is a natural number. Rearranging and reindexing, we arrive at the recurrence relations in the statement of the proposition, which shows that Equation \eqref{eqn:utinverseformula} is valid for $s=1$.  Suppose that Equation \eqref{eqn:utinverseformula} is satisfied for all natural numbers $r$ less than $s$. Then, 
	\begin{align*}
	(\mti{f}{i,i+s+1})^{-1} &= - \sum_{k=1}^{s+1} \mti{f}{{i},i+k} (\mti{f}{i+k,i+s+1})^{-1}  \\
	%	&= - \sum_{k=1}^{s} \mti{f}{{i},i+k} \mti{f^{-1}}{i+k,i+s+1}  - \mti{f}{{i},i+s+1} \\
	&= - \sum_{k=1}^{s} \mti{f}{{i},i+k} \sum_{l=1}^{s-k+1} (-1)^l \left ( \sum_{\{h_\zeta \} \in H(l,s-k+1)} \left (\prod_{\xi = 0}^{l-1} \mti{f}{{i}+k+h_\xi,i+k+h_{\xi + 1}} \right ) \right )  - \mti{f}{{i},i+s+1} \\
	&= \sum_{l=1}^{s+1} (-1)^l \left ( \sum_{\{h_\zeta \} \in H(l,s+1)} \left (\prod_{\xi = 0}^{l-1} \mti{f}{{i}+h_\xi,i+h_{\xi + 1}} \right ) \right ). \\
	\end{align*}
	Hence, Equation \eqref{eqn:utinverseformula} is satisfied for $s+1$.
	
	%	 This means that our formula for $f^{-1}$ gives us at least a right inverse for $f$.  If $g$ is the left inverse of $f$ then it is a left inverse of $f$, as 
	%	%
	%	\[ g = g I_n = g f f^{-1} = I_n f^{-1} = f^{-1} \]
	%	%
	%	which completes the proof.
\end{proof}
\begin{remark}
	A more elegant proof, pointed out to the author by George Willis, involves writing $f = I_n + g$ for a nilpotent matrix $g$ and noting that since $g$ is nilpotent and all eigenvalues are zero, $(I_n + g)^{-1} = \sum_{k=0}^{n-1} g^k$. 
\end{remark}
\begin{corollary}
	\label{cor:recursiveAinverseformula}
	Suppose that $\mt{f}$ is an $n \times n$ upper triangular matrix. Then, for each pair of natural numbers $i$ and $s$, such that $i$ is less than $n$ and $s$ is less than or equal to $n-i$, the inverse matrix $\mt{f^{-1}}$ satisfies the recurrence relations
	\begin{equation}
	(\mti{f}{i,i+s})^{-1} = \sum_{l=1}^s (-1)^l \left ( \sum_{\{h_\zeta \} \in H(l,s)} \left (\prod_{\xi = 0}^{l-1} \mti{f}{i+h_\xi, i+h_\xi} \mti{f}{i+h_\xi,i+h_{\xi + 1}} \right ) \right ) (\mti{f}{i+s, i+s})^{-1}, 
	\end{equation}
	where $H(l,s)$ is the collection of finite sequences of integers $\left \{h_0, \dots, h_l \right \}$, such that $h_k$ is less than $h_{k+1}$, $h_1 = 0$ and $h_l = s$.
\end{corollary}
\begin{proof}
	Let $g$ be the matrix that agrees with $f$ on the diagonal and is otherwise the same as the identity. Then, $f^{-1} = (g^{-1}f)^{-1}g$. The statement of the proposition then follows from Proposition \ref{prop:recursiveutinverseformula} because $(g^{-1}f)^{-1}$ is upper unitriangular and $g$ is diagonal.
\end{proof}

\section{Random walks on $\fgseq{n}{p}$}
\label{sec:rwfg}
Now we turn our attention to random walks on $\fgseq{n}{p}$ for a single prime $p$. We define a \emph{displacement matrix} associated with every measure $\mu$ on $\fgseq{n}{p}$ that has finite first moment with respect to word length. We construct a $\mu$-boundary, $(\Gamma, \nu)$, dependent on the displacement matrix, so that the right random walk converges almost surely to elements of $\Gamma$. We use this boundary to estimate paths in the right random walk with at most linear error and prove that it is a valid description of the \PF boundary in some cases. We conclude by discussing conditions for triviality of the \PF boundary. 

Let $p$ be prime and let $n$ be a natural number.  Then, $\fgseq{n}{p}$ is the discrete solvable group of upper triangular $n \times n$ matrices whose entries are integer powers of $p$ on the main diagonal and elements of $\pyad{p}$ otherwise. It is possible to write any element $f$  of $\fgseq{n}{p}$ in the form  $\mt{f} = \mt{g}\md{x}$, where $\md{x}$ is the diagonal matrix

\[ \md{x} = \begin{pmatrix} p^{x_1} & 0 & \cdots & 0 \\
0  & p^{x_2} & \cdots & 0 \\
\vdots  & \vdots & \ddots  & \vdots \\
0  & 0 & \cdots  & p^{x_n}
\end{pmatrix}. \] 
and $\mt{g}$ is the upper uni-triangular matrix with entries
\[ \mti{g}{ij} = \begin{cases} 1 &\textrm{if} \ i=j, \\ \frac{\mti{f}{ij}}{p^{x_j}} &\textrm{if} \ i<j, \textrm{and} \\ 0 &\textrm{if} \ i>j, \end{cases}\]
where $f_{ij}$ and $p$ are coprime. Hence, $\fgseq{n}{p}$ is an internal semi-direct product, isomorphic to $\Hn \ltimes \Nn$. We denote $\Hn \ltimes \Nn$ by $\Gex$.

\label{sec:gensetsforGn}
\begin{proposition}[Metric estimate on $\Gex$] 
	\label{prop:wordlengthbound}
	For each element $(\md{x},\mt{f}) $ in $\Gex$, let 
	\[ \left \llbracket (\md{x},\mt{f}) \right \rrbracket = \sum_{i=1}^n \left | \vecti{x}{i} \right | + \sum_{i=1}^{n-1} \sum_{j=i+1}^n \maxonedmdp{\mti{f}{ij}}{p}, \]
	where $|z|$ is the ordinary absolute value of  an integer $z$. Then, there are positive, real constants $J'$ and $L'$ (dependent on $n$) such that
	\[ J' \llbracket (\md{x},\mt{f}) \rrbracket \leq |(\md{x},\mt{f}) | \leq L' \llbracket (\md{x},\mt{f}) \rrbracket, \]
	i.e., the gauge function $\llbracket \cdot \rrbracket$ is a metric estimate on $\Gex$.
\end{proposition}
\begin{proof}
	See Proposition \ref{prop:rationalmetricestimate}. The generating set is instead the subset of $\Gex$ consisting of the elements $(\theta_{ii}(p), I_n)$ for natural numbers $i$ less than $n$ and $(\mathbf{0},\theta_{ij}(1))$ for natural numbers $i$ and $j$, such that $i$ is less than $j$ which is less than or equal to $n$, where $\theta_{ij}$ and $\theta_{ii}$ is as defined at the start of Section \ref{sec:groupsfgnp}.
\end{proof}
\subsection{The displacement matrix}
Suppose that $\mu$ is a measure on $\Gex$ that has finite first moment with respect to word length and that
\[ \left ( \mds{y}{m}, \ms{\varphi}{m} \right ) = \prod_{i=1}^m \left ( \mds{x}{m}, \ms{f}{m} \right ) \]
is a path in the random walk associated with $\mu$ on $\Gex$. For every natural number $p$ less than $n$, let $\hproji{p}$ be the map from $\Gex$ to $\I$ given by 
\[ \hproji{p} (\md{x}, \mt{f}) = x_p, \]
where $x = (x_1, \dots, x_n)$. Let $\imesi{p}$ be the pushforward measure given by 
\[ \imesi{p} (E) = \pf{\hproji{p}}{\mu} \]
and let $\mm{\imesi{p}} = \int_{z \in \I} z \, d \imesi{p}(z)$ be the first moment of each measure $\imesi{p}$.

Similarly, for natural numbers $p$ and $q$ less than $n$, let $\nproji{pq}$ be the map from $\Gex$ to $\dyad$ given by 
\[\nproji{pq}(\md{x}, \mt{f}) = [f]_{pq} \]
and let $\nmesi{pq}$ be the pushforward measure:  
\[ \imesi{pq} (E) = \pf{\nproji{pq}}{\mu}. \]
Finally, let $\mm{\imesi{pq}} = \int_{\dyad} z \, d \imesi{pq}$ be the first moment of $\imesi{pq}$.

%
%\begin{theorem}
%    \label{thm:pushforwardintegral}
%    Let $T \colon X \rightarrow X'$ be a measurable map between a measure space $(X, \Sigma, \mu)$ and a measurable space $(X', \Sigma')$. For every measurable and $\pf{T}{\mu}$-integrable function $f \colon X' \rightarrow \R$ the function $f \circ T$ is measurable, $\mu$-integrable and satisfies
%    %
%    \[ \int_{X'} f \, d(\pf{T}{\mu}) = \int_X f \circ T \, d \mu. \]
%    %
%\end{theorem}
%\begin{proof}
%    The proof follows from the fact that compositions of measurable functions are measurable and by considering simple functions and then applying the monotone convergence theorem. See e.g.  Schilling \cite{schilling05} for details. 
%\end{proof}

\begin{lemma}
	\label{lem:meanimeasureisfinite}
	Suppose that $\mu$ is a measure on $\Gex$ that has finite first moment with respect to word length. Then, the first moment, $\mm{\imesi{p}}$, of each measure $\imesi{p}$ is finite. 
\end{lemma}
\begin{proof}
	It is obvious that 
	\begin{align*}
	\mm{\imesi{p}} &= \int_{\I} z \, d \pf{\hproji{p}}{\mu} (z) \\
	&\leq \int_{\I} |z| \, d \pf{\hproji{p}}{\mu} (z). \\ \intertext{Because the first moment of $\mu$ is finite, Theorem 3.6.1 in Bogachev \cite{bogachev07} is applicable, stating that}
	&\int_{\I} |z| \, d \pf{\hproji{p}}{\mu} (z) = \int_{\Gex} |\hproji{p}(z)| \, d \mu (z) \\
	&\leq \int_{\Gex} |z| \, d  \mu (z),
	\end{align*}	
	which is finite since $|\cdot|$ is equivalent to every metric estimate on $\Gex$. 
\end{proof}

Let $\mt{D}$ be the matrix in $\UT_n(\R)$ given by 
\[ \Di{ij} = \mm{\imesi{i}} - \mm{\imesi{j}}. \]
We refer to $\D$ as the \emph{displacement matrix associated with $(\Gex, \mu)$}. The entries of $D$ are finite by Lemma \ref{lem:meanimeasureisfinite}. Each \superdiag{} entry $\Di{ij}$ shall be referred to as a \emph{displacement associated with $(\Gex, \mu)$}. If $i$, $j$ and $k$ are natural numbers such that $i \leq j \leq k \leq n$, then  
\[ \Di{ik} = \Di{ij} + \Di{jk}. \]
Furthermore, %, by a telescoping sum argument,\
if  $\{a_k\} $ is an element of $\MFI{r}$, as defined for Equation \eqref{eqn:expandedrecurrence}, then 
\[ \Di{i,i+r} = \sum_{k = 0}^{|a_k| - 1} \Di{a_k,a_{k+1}}. \]
In particular, by taking $\{a_k \}$  = $\{0, 1,   \dots, j-i \}$ in this expression, it is evident that $D$ is determined exactly by its entries on the first \superdiag{}. 

If every \superdiag{} entry of $D$ is non-zero, then we say that $D$ is \emph{non-zero}. If every \superdiag{} entry in a column of $D$ has the same sign, then $D$ is \emph{column homogeneous}. If every \superdiag{}  entry in a row of $D$ has the same sign, then $D$ is \emph{row homogeneous}. If all \superdiag{}  entries in $D$ have the same sign, then we say that $D$ is \emph{homogeneous}. 

\begin{lemma}
	Suppose that $\mu$ is a measure on $\Gex$ that has finite first moment with respect to word length. Let $D$ be the displacement matrix associated with $\mu$. Let $\check{\mu}$ be the reflected measure,  
	\[  \check{\mu}(E) := \mu(E^{-1}) \]
	for all measurable subsets $E$ of $G$. Let $\check{D}$ be the displacement matrix associated with $\check{\mu}$. Then
	\[ D_{ij} = \check{D}_{ji} \]
	for all natural numbers $i < j \leq n$. In particular,   
	\begin{enumerate}[(i)] 
		\item $D$ is non-zero if and only if $\check{D}$ is non-zero,
		\item $D$ is column consistent if and only if $\check{D}$ is row consistent and
		\item $D$ is row consistent if and only if $\check{D}$ is column consistent.
		\item $D$ is homogeneous if and only if $\check{D}$ is homogeneous.
	\end{enumerate}
\end{lemma}
\begin{proof}
	If $i$ is less than $j$, then
	\begin{align*}
	\Di{ij} &= \mm{\imesi{i}} - \mm{\imesi{j}} \\
	&=  \int_{z \in \I} z \, d \imesi{i}(z) - \int_{z \in \I} z \, d \imesi{j}(z) \\
	&=  - \int_{z \in \I} z \, d \check{\imesi{i}}(z) + \int_{z \in \I} z \, d \check{\imesi{i}}(z) \\
	&= \mm{\check {\imesi{j}}} - \mm{\check {\imesi{i}}} \\
	&= \check{D}_{ji},
	\end{align*}
	which gives the statement of the lemma.
\end{proof}
\subsection{Pointwise convergence to elements of a $\mu$-boundary}
\label{sec:convergence}
For the remainder, we suppose that $\mu$ is a measure on $\Gex$ that has finite first moment with respect to word length, such that the group generated by $\supp \mu$ is non-abelian. Let $D$ be the displacement matrix associated with $\mu$. Suppose that $D$ is non-zero and column consistent. Let
\[ \left ( \mds{y}{m}, \ms{\varphi}{m} \right ) = \prod_{i=1}^m \left ( \mds{x}{m}, \ms{f}{m} \right ) \]
be a path in the random walk associated with $\mu$ on $\Gex$. 

In the first part of this section we discuss convergence of entries in either $\R$ or $\Qp{p}$, depending on the displacement matrix. In the second part we use this information to construct a $\mu$-boundary, such that the right random walk converges almost surely to elements of $\Gamma$.

\begin{lemma}
	\label{lem:subexpentries}
	For every pair of natural numbers $i$ and $j$, such that $i \leq j \leq n$,  
	\[ \log \left ( 1 + \left | \msi{f}{m}{ij} \right | \right ) \
	\textrm{and} \ \log \left ( 1 + \left | \msi{f}{m}{ij} \right |_p \right )  \] 
	are both in $o(m)$ almost surely.
\end{lemma}
\begin{proof}
	Let $i$ and $j$ be natural numbers such that $i<j\leq n$. Then, 
	\begin{align*}
	\left | \log \left ( 1 + \left | \msi{f}{m}{ij} \right | \right ) \right | &\leq \left |  d_+^{p} \left ( 1 + \left | \msi{f}{m}{ij} \right | \right ) \right | \\
	&\leq 1 + \left | d_+^{p} \left ( \msi{f}{m}{ij} \right ) \right | \\
	&\leq 1 + \max \left \{ \left | d_-^{p} \left ( \msi{f}{m}{ij} \right ) \right |, \left | d_+^{p} \left ( \msi{f}{m}{ij} \right ) \right | \right \} \\
	&\leq \maxonedmdp{\msi{f}{m}{ij}}{p} \\
	&\leq \frac{1}{K} \left | \left (\mds{x}{m},\ms{f}{m} \right ) \right |,
	\end{align*}
	for some positive, real constant $K$. Similarly,  
	\begin{align*}
	\left | \log \left ( 1 + \left | \msi{f}{m}{ij} \right |_p \right ) \right | &\leq  1 + \left | d_-^{p} \left (\msi{f}{m}{ij} \right ) \right | \\
	&\leq 1 + \max \left \{ \left | d_-^{p} \left (\msi{f}{m}{ij} \right ) \right |, \left | d_+^{p} \left ( \msi{f}{m}{ij} \right ) \right | \right \} \\
	&\leq \maxonedmdp{\msi{f}{m}{ij}}{p}  \\
	&\leq \frac{1}{K} \left | \left ( \mds{x}{m},\ms{f}{m} \right ) \right |. 
	\end{align*}
	Suppose that $ \left |\left ( \mds{x}{m}, \ms{f}{m} \right ) \right | $ is not in $o(m)$ almost surely. The integral
	\[ \int_{\Gex} \left | \left ( \vect{x} , f \right ) \right | \, d \mu (\vect{x},f) \]
	is not bounded, which is a contradiction, because we assumed $\mu$ had finite first moment with respect to word length. 
\end{proof}

\begin{lemma}
	\label{lem:driftbounds} For all positive, real numbers $\varepsilon$, there is a natural number $N$, such that
	\[ (\mti{D}{ij} - \varepsilon) l < \ylij{l}{i}{j} < (\mti{D}{ij} + \varepsilon)l \]
	whenever $l$ is greater than $N$.
\end{lemma}
\begin{proof}
	This statement is just the strong law of large numbers, which states the existence of a natural number $N$, such that
	\[ \left | \frac{\ylij{l}{i}{j}}{l} - \mti{D}{ij} \right | < \varepsilon \]
	whenever $l$ is greater than $N$. 
\end{proof}
\begin{corollary}
	\label{cor:bndonrw}
	If $\mti{D}{ij}$ is positive and $x$ is a real number, then there is a natural number $N$, such that 
	\[ \ylij{l}{i}{j} > x \]
	whenever $l$ is greater than $N$. Similarly, if $\mti{D}{ij}$ is negative, then there is a natural number $N$, such that
	\[ \ylij{l}{i}{j} < x \]
	whenever $l$ is greater than $N$.
\end{corollary}
\begin{proof}
	Choose $\varepsilon = \frac{\mti{D}{ij}}{2}$ in Lemma \ref{lem:driftbounds}. If $\mti{D}{ij}$ is positive, then $\frac{1}{2} l \mti{D}{ij}$ is greater than $x$ for all sufficiently large $l$. Similarly, if $\mti{D}{ij}$ is negative,  $\frac{3}{2}l \mti{D}{ij}$ is less than $x$ for all sufficiently large $l$.
\end{proof}
\begin{corollary} If $\mti{D}{ij}$ is positive, then the sequence 
	\[ \left \{ \max_{1 \leq r \leq l} \left ( \ylij{l}{i}{j} \right ) \right \}_{l \in \NN} \]
	is eventually constant. 
\end{corollary}
\begin{proof}
	By Corollary \ref{cor:bndonrw}, $\ylij{l}{i}{j}$ is less than $\ylij{1}{i}{j}$ for all sufficiently large $l$. 
\end{proof}

\begin{proposition}
	Suppose that $\mu$ is a measure on $\Gex$ that has finite first moment with respect to word length, such that the group generated by $\supp \mu$ is non-abelian. Let $D$ be the displacement matrix associated with $\mu$. Let $i$ and $j$ be natural numbers such that $i < j \leq n$. Suppose that $\mu$ is a measure such that $\mti{D}{ij}$ positive and $\sgn{\mti{D}{ij}} = \sgn{\mti{D}{kj}}$ for all natural numbers $k$, such that $i < k < j$. Let 
	\[ \left ( \mds{y}{m}, \ms{\varphi}{m} \right ) = \prod_{i=1}^m \left ( \mds{x}{m}, \ms{f}{m} \right ) \]
	be a path in the random walk associated with $\mu$ on $\Gex$.  The sequence $\msi{\varphi}{m}{ij}$ is almost surely convergent in $\Qp{p}$.
\end{proposition}
\begin{proof}
	Let $r=j-i$. Suppose that $m$ is large, so that Equation \eqref{eqn:expandedrecurrence} applies. Let $a= \{ a_k \}$ be an element of $\MFI{r}$.   To show that the sequence
	\[ S_{a}^{(m)} = \sum_{0 \leq b_1 < \dots < b_{\left | a \right |-1} < m} \prod_{n=0}^{\left | a \right | -2} \msi{f}{b_{n+1}+1}{i+a_n,i+a_{n+1}} p^{\mdsi{y}{b_{n+1}}{i + a_n} - \mdsi{y}{b_{n+1}}{i + a_{n+1}}} \]
	converges almost surely in the $p$-adic numbers, consider the limit
	\begin{align*}
	L_{a}^{(m)} &= \lim_{b_{| a |-1} \rightarrow \infty} \left | \sum_{0 \leq b_1 < \dots < b_{\left | a \right |-1}} \prod_{n=0}^{\left | a \right | -2} \msi{f}{b_{n+1}+1}{i+a_n,i+a_{n+1}} p^{\mdsi{y}{b_{n+1}}{i + a_n} - \mdsi{y}{b_{n+1}}{i + a_{n+1}}} \right | _p \\
	%	&\leq  \lim_{b_{| a |-1} \rightarrow \infty} \max_{0 \leq b_1 < \dots < b_{\left | a \right |-1}} \left | \prod_{n=0}^{\left | a \right | -2} \msi{f}{b_{n+1}+1}{i+a_n,i+a_{n+1}} p^{\mdsi{y}{b_{n+1}}{i + a_n} - \mdsi{y}{b_{n+1}}{i + a_{n+1}}} \right | _p
	%	\\
	%	&\leq  \lim_{b_{| a |-1} \rightarrow \infty}  \max_{0 \leq b_1 \leq \dots \leq b_{\left | a \right |-1}} \left | \prod_{n=0}^{\left | a \right | -2} \msi{f}{b_{n+1}+1}{i+a_n,i+a_{n+1}} p^{\mdsi{y}{b_{n+1}}{i + a_n} - \mdsi{y}{b_{n+1}}{i + a_{n+1}}} \right | _p
	%	\\
	&\leq \lim_{b_{| a |-1} \rightarrow \infty}  \max_{0 \leq b_1 \leq \dots \leq b_{\left | a \right |-1}}  \prod_{n=0}^{\left | a \right | -2} \left | \msi{f}{b_{n+1}+1}{i+a_n,i+a_{n+1}} \right | _p p^{-\mdsi{y}{b_{n+1}}{i + a_n} + \mdsi{y}{b_{n+1}}{i + a_{n+1}}}.
	\end{align*}
	By Lemma \ref{lem:subexpentries}, $\log \left ( 1 + \left |  \msi{f}{b_{n+1}+1}{i+a_n,i+a_{n+1}} \right |_p \right ) $ is almost surely in $o(b_{n+1})$ for each non-negative integer $n$ less than $|a|-1$. So
	\begin{align*}
	L_{a}^{(m)} &\leq  \lim_{b_{| a |-1} \rightarrow \infty}  \max_{0 \leq b_1 \leq \dots \leq b_{\left | a \right |-1}}  \prod_{n=0}^{\left | a \right | -2} b_{n+1} p^{-\mdsi{y}{b_{n+1}}{i + a_n} + \mdsi{y}{b_{n+1}}{i + a_{n+1}}} \\
	&\leq  \lim_{b_{| a |-1} \rightarrow \infty} \left ( b_{\left | a \right |-1} \right )^{|a|-1} \max_{0 \leq b_1 \leq \dots \leq b_{\left | a \right |-1}}  \prod_{n=0}^{\left | a \right | -2}  p^{-\mdsi{y}{b_{n+1}}{i + a_n} + \mdsi{y}{b_{n+1}}{i + a_{n+1}}} 
	\end{align*}
	%    be the absolute value of the maximum pairwise difference of distinct entries in the displacement matrix $D$ divided by $|a|$, that is
	%    %
	%    \[ \varepsilon = \frac{1}{|a|-1}\max \left ( \left | |\mti{D}{gh}| - |\mti{D}{ij}| \right | : g,h,i,j \in \NN, g < h \leq n, i < j \leq n, (g,h) \neq (i,j) \right ). \]
	%
	Then, Lemma $\ref{lem:driftbounds}$ gives
	\begin{align*}
	L_{a}^{(m)} &\leq \lim_{b_{| a |-1} \rightarrow \infty}  \left ( b_{\left | a \right |-1} \right )^{|a|-1} \max_{0 \leq b_1 \leq \dots \leq b_{\left | a \right |-1}}  \prod_{n=0}^{\left | a \right | -2}  p^{-b_n(\mti{D}{i+a_n, i+a_{n+1}} - \varepsilon)} \\
	&\leq \lim_{b_{| a |-1} \rightarrow \infty}  p^{b_{\left | a \right |-1}(|a|-1)\varepsilon} \left ( b_{\left | a \right |-1} \right )^{|a|-1} \max_{0 \leq b_1 \leq \dots \leq b_{\left | a \right |-1}}\prod_{n=0}^{\left | a \right | -2}  p^{-b_n \mti{D}{i+a_n, i+a_{n+1}}} \\
	&= \lim_{b_{| a |-1} \rightarrow \infty}  p^{b_{\left | a \right |-1}(|a|-1)\varepsilon} \left ( b_{\left | a \right |-1} \right )^{|a|-1} \max_{0 \leq n \leq |a| - 2} p^{- l D_{i+a_n, i+r}}
	\end{align*}
	for any positive, real numbers $\varepsilon$. The displacement conditions give that $D_{i+a_n, i+r}$ is positive for all non-negative integers $n$ less than $|a|-1$, choosing any sufficiently small $\varepsilon$, we see that the limit $L_{a}^{(m)}$ is zero, hence $S_{a}^{(m)}$ is convergent. 
\end{proof}

\begin{corollary}
	\label{prop:positivedisplacementsconvergenceproof}
	If entry in the displacement matrix corresponding to $(\Gex, \mu)$ is negative, 
	%and that
	%%
	%\[ ([\mds{y}{m}], [\msi{\varphi}{m}{ij}]) = \prod_{i=1}^m ([\mdsi{x}{m}{i}], [\msi{f}{m}{ij}]) \]
	%%
	%is a path in the random walk associated with $\mu$ on $\Gex$
	then the sequence of matrices $\ms{\varphi}{m}$ converges pointwise almost surely to a matrix $\ms{\varphi}{\infty}$ in the $p$-adic valued unitriangular matrices,
	$\UT_n(\Qp{p})$.  
\end{corollary}
\begin{lemma}
	\label{lem:powersofpbound}
	Suppose that $\mu$ is a measure on $\Gex$ that has finite first moment with respect to word length, such that the group generated by $\supp \mu$ is non-abelian. Let $D$ be the displacement matrix associated with $\mu$. Let $i$ and $j$ be natural numbers such that $i < j \leq n$. Suppose that $\mu$ is a measure such that $\mti{D}{ij}$ is negative. Let $r=j-i$. Let $a$ be an element of $\MFI{r}$, then there is a positive, real constant $K$, such that 
	\[ U_a^{(m)} :=  \sum_{0 \leq b_1 < \dots < b_{\left | a \right |-1} } \prod_{n=0}^{\left | a \right | -2} p^{b_{n+1} \mti{D}{i + a_n, i + a_{n+1}} } \leq K p^{b_{\left | a \right |-1}  D_{i+a_n,i+r}} \]
	for each non-negative integer $n$ less than $|a|$.
\end{lemma}
\begin{proof}
	The base case is straightforward. The series
	\begin{equation}
	\sum_{0 \leq b_1 < b_2} p^{b_1 \mti{D}{i, i + a_1}} \label{eqn:baseinductionbound}
	\end{equation}
	is geometric and bounded above by a positive, real constant $K_1$ if $\mti{D}{i, i + a_1}$ is negative. If $\mti{D}{i, i + a_1}$ is positive, then  \eqref{eqn:baseinductionbound} is equal to $p^{b_2\mti{D}{i, i + a_1}} - 1$, which is bounded above by $p^{b_2\mti{D}{i, i + a_1}}$. 
	
	Suppose that $k$ is a natural number less than $|a|-2$ and that there is a positive, real constant $K_k$ and a natural number $s$ less than $k$, such that
	\[ \sum_{0 \leq b_1 < \dots < b_k < b_{k+1}} \prod_{n=0}^{k} p^{b_{n+1} \mti{D}{i + a_n, i + a_{n+1}} } \leq K p^{b_{k+1} D_{i+a_s,i+a_k}}. \]
	If $\mti{D}{i,a+a_{k+1}}$ is negative, then $\mti{D}{i+a_n,a+a_k} + \mti{D}{i+a_k,a+a_{k+1}} = \mti{D}{i+a_n,a+a_{k+1}}$ is negative by the condition given in the lemma, implying the convergence of \[ K_{k+1} \sum_{k-1 \leq b_{k+1} < b_{k+2}} p^{b_{k+1} (D_{i+a_s,i+a_k} + D_{i+a_k,i+a_{k+1}})}\] to a limit $K_{k+1}$ and that
	\begin{align*}
	\sum_{0 \leq b_1 < \dots < b_k < b_{k+1}} \prod_{n=0}^{k} p^{b_{n+1} \mti{D}{i + a_n, i + a_{n+1}} } &\leq K_k p^{m D_{i+a_s,i+a_k}} \\ &\leq K_k \sum_{k-1 \leq b_{k+1} < b_{k+2}} p^{b_{k+1} (D_{i+a_s,i+a_k} + D_{i+a_k,i+a_{k+1}})}
	\\ &\leq K_{k+1} = K_{k+1} p^{b_{k+2} \mti{D}{i + a_{k+1}, i + a_{k+1} }}.
	\end{align*}
	If $\mti{D}{i,a+a_{k+1}}$ is instead negative, such that $\mti{D}{i+a_n,a+a_k} + \mti{D}{i+a_k,a+a_{k+1}}$ is positive, then
	\begin{align*}
	\sum_{0 \leq b_1 < \dots < b_k < b_{k+1}} \prod_{n=0}^{k} p^{b_{n+1} \mti{D}{i + a_n, i + a_{n+1}} } &\leq K_k p^{m D_{i+a_s,i+a_k}} \\ &\leq K_k \sum_{k-1 \leq b_{k+1} < b_{k+2}} p^{b_{k+1} (D_{i+a_s,i+a_k} + D_{i+a_k,i+a_{k+1}})} \\
	&\leq K_k p^{b_{k+2} \mti{D}{i + a_s, i + a_{k+1} }}.
	\end{align*}
	The lemma then follows by induction.
\end{proof}
\begin{proposition}
	Suppose that $\mu$ is a measure on $\Gex$ that has finite first moment with respect to word length, such that the group generated by $\supp \mu$ is non-abelian. Let $D$ be the displacement matrix associated with $\mu$. Let $i$ and $j$ be natural numbers such that $i < j \leq n$. Suppose that $\mu$ is a measure such that $\mti{D}{ij}$ negative and $\sgn{\mti{D}{ij}} = \sgn{\mti{D}{kj}}$ for all natural numbers $k$ between $i$ and $j$. Let 
	\[ \left ( \mds{y}{m}, \ms{\varphi}{m} \right ) = \prod_{i=1}^m \left ( \mds{x}{m}, \ms{f}{m} \right ) \]
	be a path in the random walk associated with $\mu$ on $\Gex$.  Then, the sequence $\msi{\varphi}{m}{ij}$ is almost surely convergent in $\R$.
\end{proposition}
\begin{proof}	
	Let $r=j-i$. Suppose that $m$ is large, such that Equation \eqref{eqn:expandedrecurrence} applies. Let $a= \{ a_k \}$ be an element of $\MFI{r}$. For each positive, real $\varepsilon$, there is the following bound on the summands of $ S_{a}^{(m)}$:
	\begin{align*}
	&\left | \sum_{0 \leq b_1 < \dots < b_{\left | a \right |-1}} \prod_{n=0}^{\left | a \right | -2} \msi{f}{b_{n+1}+1}{i+a_n,i+a_{n+1}} p^{\mdsi{y}{b_{n+1}}{i + a_n} - \mdsi{y}{b_{n+1}}{i + a_{n+1}}} \right |\\
	&\leq 
	\sum_{0 \leq b_1 < \dots < b_{\left | a \right |-1}} \prod_{n=0}^{\left | a \right | -2} \left |  \msi{f}{b_{n+1}+1}{i+a_n,i+a_{n+1}} \right | p^{\mdsi{y}{b_{n+1}}{i + a_n} - \mdsi{y}{b_{n+1}}{i + a_{n+1}}} \\
	&\leq \left (
	\prod_{n=0}^{\left | a \right | -2} \left |  \msi{g}{m}{i+a_n,i+a_{n+1}} \right | \right ) \sum_{0 \leq b_1 < \dots < b_{\left | a \right |-1}} \prod_{n=0}^{\left | a \right | -2} p^{\mdsi{y}{b_{n+1}}{i + a_n} - \mdsi{y}{b_{n+1}}{i + a_{n+1}}} \\
	&\leq \left (
	\prod_{n=0}^{\left | a \right | -2} \left |  \msi{g}{m}{i+a_n,i+a_{n+1}} \right | \right ) \sum_{0 \leq b_1 < \dots < b_{\left | a \right |-1}} \prod_{n=0}^{\left | a \right | -2} p^{b_{n+1} \left ( \mti{D}{i + a_n, i + a_{n+1}} + \frac{\varepsilon}{|a|-2} \right ) } 
	\end{align*}
	Lemma \ref{lem:powersofpbound} gives
	\begin{align*}
	&\left | \sum_{0 \leq b_1 < \dots < b_{\left | a \right |-1}} \prod_{n=0}^{\left | a \right | -2} \msi{f}{b_{n+1}+1}{i+a_n,i+a_{n+1}} p^{\mdsi{y}{b_{n+1}}{i + a_n} - \mdsi{y}{b_{n+1}}{i + a_{n+1}}} \right | \\
	&\leq K \left (
	\prod_{n=0}^{\left | a \right | -2} \left |  \msi{g}{m}{i+a_n,i+a_{n+1}} \right | \right ) p^{b_{|a|-1} \varepsilon} p^{b_{\left | a \right |-1}  D_{i+a_n,i+r}}
	\end{align*}
	for some positive, real constant $K$ and non-negative integer $n$ less than $|a|$. Choosing $\varepsilon$ so that it is less than $D_{i+a_n,i+r}$  this bound converges absolutely by the ratio test using Lemma \ref{lem:subexpentries}. Absolute convergence	of  $ S_{a}^{(m)}$ follows from the limit comparison test. Since $a$ was an arbitrary member of the finite set $\MFI{r}$, the sum
	\begin{equation}
	\msi{\varphi}{m}{i,i+r} = \sum_{a \in \MFI{r}} S_{a}^{(m)} \tag{\ref{eqn:expandedrecurrence}}
	\end{equation}
	converges in $\R$.
\end{proof}
\begin{remark}
	\label{prop:negativedisplacementsconvergenceproof}
	Suppose that every entry in the displacement matrix corresponding to $(\Gex, \mu)$ is negative. 
	%and that
	%%
	%\[ ([\mds{y}{m}], [\msi{\varphi}{m}{ij}]) = \prod_{i=1}^m ([\mdsi{x}{m}{i}], [\msi{f}{m}{ij}]) \]
	%%
	%is a path in the random walk associated with $\mu$ on $\Gex$
	Then, the sequence of matrices $\ms{\varphi}{m}$ converges pointwise almost surely to a matrix $\ms{\varphi}{\infty}$ in the real valued unitriangular matrices,
	$\UT_n(\R)$.  
\end{remark}

\begin{lemma}
	Suppose that $\mu$ is a measure on $\Gex$ that has finite first moment with respect to word length, such that the group generated by $\supp \mu$ is non-abelian. Let $D$ be the displacement matrix associated with $\mu$. Let $i$ and $j$ be natural numbers such that $i < j \leq n$. Let 
	\[ \left ( \mds{y}{m}, \ms{\varphi}{m} \right ) = \prod_{i=1}^m \left ( \mds{x}{m}, \ms{f}{m} \right ) \]
	be a path in the random walk associated with $\mu$ on $\Gex$. Let $x$ be a positive, real number. Then, for every positive, real number $\varepsilon$ less than $D_{ij}$, there is a positive, real constant $K$ and a natural number $N$, such that
	\[ \sum_{l=0}^{m-1} | \msi{f}{l+1}{ij} | p^{\ylij{l}{i}{ij} + l x} \leq  K \left ( p^{m(\mti{D}{ij} + x) + \varepsilon} -1 \right ) \]
	whenever $m-1$ is greater than $N$.
\end{lemma}
\begin{proof}
	Let $\msi{g}{l}{ij} =  \sum_{s=1}^{l} | \msi{f}{l}{ij} |  $. Then, 
	
	\begin{align*}
	\sum_{l=1}^{m-1} | \msi{f}{l+1}{ij} | p^{\ylij{l}{i}{ij} + l x} &\leq \sum_{l=1}^{m-1} \msi{g}{l+1}{ij} p^{\ylij{l}{i}{ij} + l x} \\
	&\leq \sum_{l=1}^{m-1} l \frac{\msi{g}{l+1}{ij}}{l} p^{\ylij{l}{i}{ij} + l x} \\
	&\leq p^\varepsilon \sum_{l=1}^{m-1} l |\mu_{ij}| p^{\ylij{l}{i}{ij} + l x} \\
	&\leq p^\varepsilon m \sum_{l=1}^{m-1} |\mu_{ij}| p^{\ylij{l}{i}{ij} + l x} \\
	\end{align*}
	Which completes the proof.
\end{proof}

\begin{lemma}
	Suppose that $\mu$ is a measure on $\Gex$ that has finite first moment with respect to word length, such that the group generated by $\supp \mu$ is non-abelian. Let $D$ be the displacement matrix associated with $\mu$. Let $i$ and $j$ be natural numbers such that $i < j \leq n$. Let 
	\[ \left ( \mds{y}{m}, \ms{\varphi}{m} \right ) = \prod_{i=1}^m \left ( \mds{x}{m}, \ms{f}{m} \right ) \]
	be a path in the random walk associated with $\mu$ on $\Gex$. Let $x$ be a positive, real number. Then, for every positive, real $\varepsilon$ less than $| D_{ij} |$, there is a positive, real constant $K$ and a natural number $N$, such that
	
	\[ K \left ( p^{m(\mti{D}{ij} + x - \varepsilon)} -1 \right ) \leq \sum_{l=N+1}^{m-1} | \msi{f}{l+1}{ij} | p^{l\mti{D}{ij} + l x} \leq  K \left ( p^{m(\mti{D}{ij} + x + \varepsilon)} -1 \right )\]
	whenever $m-1$ is greater than $N$.
\end{lemma}
\begin{proof}
	Let $x$ be a positive, real number. By Lemma
	\ref{lem:subexpentries}, 
	\begin{align*} \mti{L}{ij}(x) := \lim_{l \rightarrow \infty} \left | \frac{| \msi{f}{l+1}{ij} | p^{l(\mti{D}{ij} +  x)}}{| \msi{f}{l}{ij} | p^{(l-1)(\mti{D}{ij} +  x)}} \right | &= p^{\mti{D}{ij}+x},
	\end{align*}
	almost surely. For brevity, let 
	\[ \msi{\theta}{l}{ij} = | \msi{f}{l+1}{ij} | p^{l\mti{D}{ij} + l x}. \]
	Let $\varepsilon$ be a positive, real number less than $\mti{D}{ij}$ and let $N$ be the natural number such that 
	\[ p^{\mti{D}{ij}+x - \varepsilon} \msi{\theta}{l}{ij} < \msi{\theta}{l+1}{ij} < p^{\mti{D}{ij}+x + \varepsilon} \msi{\theta}{l}{ij}  \]
	for all $l$ greater than $N$, such that 
	\[ p^{k(\mti{D}{ij}+x - \varepsilon)} \msi{\theta}{l}{ij} <  \msi{\theta}{l+k}{ij} < p^{k(\mti{D}{ij}+x + \varepsilon)}   \msi{\theta}{l}{ij}  \]
	for all $l$ greater than $N$ and natural numbers $k$. Then, 
	\[ \msi{\theta}{N}{ij} \sum_{l=N+1}^{m-1} p^{l(\mti{D}{ij}+x - \varepsilon)} 
	< \sum_{l=N+1}^{m-1} | \msi{f}{l+1}{ij} | p^{l\mti{D}{ij} + l x} 
	< \msi{\theta}{N}{ij} \sum_{l=N+1}^{m-1} p^{l(\mti{D}{ij}+x + \varepsilon)}.  \]
	This is the statement of the lemma where $K = \msi{\theta}{N}{ij}$.
	Since $\varepsilon < \mti{D}{ij}$ and $0 < \mti{D}{ij}$, there is a positive, real number $x$, such that
	\[ \frac{1}{x} \sum_{l=1}^{N} p^{l(\mti{D}{ij}+x - \varepsilon)} < \sum_{l=1}^{N} | \msi{f}{l+1}{ij} | p^{\ylij{l}{i}{ij} + l x} < x \sum_{l=1}^{N} p^{l(\mti{D}{ij}+x + \varepsilon)}. \]
	Hence, summing the geometric series, 
	\[  M_1 \left ( p^{m(\mti{D}{ij}+x - \varepsilon)} - 1 \right ) < \sum_{l=1}^{m-1} | \msi{f}{l+1}{ij} | p^{\ylij{l}{i}{ij} + l x} < M_2 \left ( p^{m(\mti{D}{ij}+x + \varepsilon)} - 1 \right ) \]
	where $M_1 = \min \left ( \frac{1}{x}, \msi{\theta}{N}{ij}  \right )$ and $M_2 = \min \left ( x, \msi{\theta}{N}{ij}  \right )$.
\end{proof}

We now construct a $\mu$-boundary for $(\Gex, \mu)$. Let $\Gamma$ be the set of all $n \times n$ matrices $M$ whose entries $\mti{M}{ij}$ are in $\R$ if $\mti{D}{ij}$ is negative and in $\Qp{p}$ if $\mti{D}{ij}$ is positive. Let $\ms{\varphi}{\infty}(\omega)$ be the matrix in $\Gamma$ with entries
\[ \msi{\varphi}{\infty}{ij}(\omega) = \lim_{m \rightarrow \infty} \msi{\varphi}{m}{ij}(\omega) \]
whenever the limit exists. Let $\Gex$ have the action on $\Gamma$ given by
\[ ( \md{x}, \mt{f} ) \cdot \mt{b} := \mt{f} \zeta_\md{x} (\mt{b})   \] 
for each $\bP$ in $\Gamma$ and $( \md{x}, \mt{f} ) $ in $G$, 
where $\zeta_\md{x} (y) = XyX^{-1}$,
\[ X = \begin{pmatrix} p^{x_1} & 0 & \cdots & 0 \\
0  & p^{x_2} & \cdots & 0 \\
\vdots  & \vdots & \ddots  & \vdots \\
0  & 0 & \cdots  & p^{x_n}
\end{pmatrix} \] 
and $\md{x} = (x_1, \dots, x_n )$. So 

%Since $\mti{\left ( \mt{f} \zeta_\md{x} (\mt{b})  \right )}{ij} = p^{\mti{x}{i} - \mti{x}{j}} \mti{b}{ij}$, the definition of matrix multiplication gives
%
\[ \left ( \mt{f} \zeta_\md{x} (\mt{b})  \right )_{ij} = \sum_{k=1}^n \mti{f}{ij} p^{\mti{x}{k} - \mti{x}{j}} \mti{b}{kj} \]
If $\mti{D}{ij}$ is negative, then the displacement condition forces $\mti{b}{kj}$ to be real for every number $k$ between $1$ and $n$. Similarly, if $\mti{D}{ij}$ is negative, then $\mti{b}{kj}$ is an element of the $p$-adics for every number $k$ between $1$ and $n$. Hence,  $\Gamma$ is closed under the $\Gex$-action. $\Gamma$ is second countable. In fact, $\Gamma$ is a $\Gex$-space, because if
\[ \lim_{m\rightarrow \infty}  \ms{\varphi}{m} = \ms{\varphi}{\infty} \]
then
\[ \lim_{m\rightarrow \infty} (\md{x}, \mt{f}) \ms{\varphi}{m} = (\md{x}, \mt{f}) = \ms{\varphi}{\infty}. \]
Let $\nu = \pf{{\varphi^{(\infty)}}}{\pathmeasure}$ be the hitting measure of $(\Gex, \mu)$ on $\Gamma$. Then, $\nu$ is $\mu$-stationary, because
\begin{align*}
\int_\Gamma f(b) \, d \mu \ast \nu(b) &= \int_{\Gex} \int_\Gamma f(gb) \, d\nu(b) \, d \mu(g) \\
&=   \int_{\Gex} \int_{(\Gex)^\NN} f(g \ms{\varphi}{\infty}(\omega)) \, d \pathmeasure(\omega) \, d \mu(g) \\
&=   \int_{\Gex} \int_{(\Gex)^\NN} f( \ms{\varphi}{\infty}(T \omega)) \, d \pathmeasure(\omega) \, d \mu(g) \\
&=   \int_{(\Gex)^\NN} f( \ms{\varphi}{\infty}(T \omega)) \, d \pathmeasure(\omega) \\
&=   \int_{(\Gex)^\NN} f( \ms{\varphi}{\infty}(\omega)) \, d \pathmeasure(\omega) \\
&= \int_\Gamma f(b) \, d \nu (b),
\end{align*}
where $T$ is the left shift map and $f$ is in $C(\Gamma)$. Then, $(\Gamma, \nu)$ is a $\mu$--boundary of $(\Gex, \mu)$, with boundary map $\varphi^{(\infty)}$, because $\varphi^{(\infty)}$ is shift invariant in the sense used by Kaimanovich in \cite{kaimanovich00} or  Brofferio and Schapira \cite{brofferio2011poisson} Proposition 2.1.

\subsection{Path estimates for column consistent right random walks}
\label{sec:mubdry}
As before, suppose $\mu$ is a measure on $\Gex$ that has finite first moment with respect to word length, such that the group generated by $\supp \mu$ is non-abelian. Let $D$ be the displacement matrix associated with $\mu$.  Suppose that $D$ is non-zero and column consistent. Let $\omega$ be a path in the random walk associated with $\mu$ with increments $\omega_m = \left ( \mds{x}{m}, \ms{f}{m} \right )$ and let
\[ \left ( \mds{y}{m}, \ms{\varphi}{m} \right ) = \prod_{i=1}^m \left ( \mds{x}{m}, \ms{f}{m} \right ). \]
Let $(\Gamma, \mu)$ be the $\mu$-boundary described in Section \ref{sec:convergence}.

In \cite{kaimanovich91}, Kaimanovich described the \PF boundary of $\Aff(\pyad{2})$ using his  ray criterion, see Theorem \ref{thm:kaimanovichapproxthm}. Depending on the sign of the \emph{drift} associated with the walk, the sequence of measurable `approximation' maps were left or right decimal truncations. We have seen that $\Gex$ has a $\mu$-boundary, $\Gamma$, with a similar description in appropriate displacement cases. It is natural to ask if similar mappings allow use of the ray criterion in that case. In this section, we give a sequence of measurable maps $\{ \Pi^{(m)} \}_{m \in \NN}$ from $\Gamma$ to $\Gex$ such that for almost every path $\left ( \mds{y}{m}, \ms{\varphi}{\infty} \right )$, such that the sequence
\[ \left | \bracc{\Pi^{(m)} \left ( \ms{\varphi}{\infty} \right ) }^{-1} \left ( \mds{y}{m}, \ms{\varphi}{m} \right ) \right | \]
is in $O(m)$ almost surely. Unfortunately, this bound is not quite tight enough to use the ray criterion. We now make the needed definitions.

Let $m$ be a natural number. Let $ \vects{t}{m} = \left ( \flr{ m \left ( \mm{\mu_{x_1}} \right )}, \dots, \flr{m \left ( \mm{\mu_{x_n}} \right ) }  \right )$. For brevity, let $\mti{M}{ij} = \tie{i}-\tie{j}$. Let $T^{(m)}$ be the map from $\Gamma$ to $\Gex$ given by
\[ T^{(m)}(b)_{ij} = \begin{cases} p^{\mti{M}{ij}} \flr{p^{-\mti{M}{ij}} \mti{b}{ij} } & \ \textrm{if} \ \mti{D}{ij} \ \textrm{is negative}  \\ p^{\mti{M}{ij}} \fracc{p^{-\mti{M}{ij}} \mti{b}{ij} } & \ \textrm{if} \ \mti{D}{ij} \ \textrm{is positive} \end{cases}  \] 
for all appropriate natural numbers $i$ and $j$, where $\flr{x}$ is the floor of $x$ and $\fracc{x}$ is the fractional part of $x$. If the boundary point $b$ is clear from the context, we write
%
% Let $T^{(m)}$ be the map from $\Gamma$ to $\Gex$ given by 
%%
%\[  \mti{T^{(m)}\left (\ms{\varphi}{\infty} \right )}{ij} =  p^{- \left | \flr{m D_{ij}} \right | }  \flr{p^{\left | \flr{m D_{ij}} \right |} \phi_{\left | \flr{m D_{ij}} \right |} \circ \pi_{\left | \flr{m D_{ij}} \right |} \left ( \msi{\varphi}{\infty}{ij} \right )  } \]
%
%%
%for each $\ms{\varphi}{\infty}$ in $\Gamma$, where $\flr{\cdot}$ is the floor function\nomenclature{$\flr{\cdot}$}{floor function}, $\pi_m$ is the canonical projection mapping from $\mathbb{S}_p$ to $\R / p^i \I$ and $\phi_m$ maps each coset $\R / p^i \I$ to its unique representative in $[0,p^i)$. If the boundary point $\varphi^{(\infty)}$ is clear from the context, we will use the following notation to

\begin{enumerate}[(i)]
	\item $\tc{m}$ to mean $\mt{T^{(m)}\left (b \right )}$ and
	\item $\tcinverse{m}$ to mean $\left ( \mt{T^{(m)}\left (b \right )} \right )^{-1}$.
\end{enumerate}
Let $\Pi^{(m)}$ be the map from $\Gamma$ to $\Gex$ given by
\[ \Pi^{(m)} \left ( \ms{\varphi}{\infty} \right )  = \left ( \vects{t}{m}, \tc{m} \right ). \]
Then, 
\[ \left ( \Pi^{(m)} \left ( \ms{\varphi}{\infty} \right )   \right )^{-1} = \left ( -\vects{t}{m}, \zeta_{-\vects{t}{m}} \left ( \tcinverse{m} \right ) \right ) \]
which means that
\begin{align*}
\left ( \Pi^{(m)} \left ( \ms{\varphi}{\infty} \right )   \right )^{-1} \left ( \mds{y}{m},  \ms{\varphi}{m} \right )
&= \left ( \vects{x}{m} - \vects{t}{m}, \zeta_{-\vects{t}{m}} \left (\tcinverse{m} \right )
\zeta_{-\vects{t}{m}} \left (\ms{\varphi}{m}  \right ) \right ) \\
&= \left ( \vects{x}{m} - \vects{t}{m}, \zeta_{-\vects{t}{m}} \left ( \tcinverse{m} \ms{\varphi}{m} \right ) \right ),
\end{align*}
For brevity, let $\ms{\gamma}{m}$ be equal to $ \zeta_{-\vects{t}{m}} \left ( \tcinverse{m} \ms{\varphi}{m} \right )$, such that
\[ \left ( \Pi^{(m)} \left ( \ms{\varphi}{\infty} \right )   \right )^{-1} \left ( \mds{y}{m},  \ms{\varphi}{m} \right ) = \left ( \vects{x}{m} - \vects{t}{m} , \ms{\gamma}{m} \right ). \]
From the bound given in Proposition \ref{prop:wordlengthbound}, there is a positive, real constant $L'$, not dependent on the particular path, such that 
\begin{align}
\left | \left ( \Pi^{(m)} \left ( \ms{\varphi}{\infty} \right )   \right )^{-1} \left ( \mds{y}{m},  \ms{\varphi}{m} \right ) \right | &= \left | \left ( \vects{x}{m}-\vects{t}{m} , \ms{\gamma}{m} \right ) \right | \nonumber \\
&\leq L' \left \llbracket \left ( -\vects{t}{m}  \vects{x}{m}, \ms{\gamma}{m} \right ) \right  \rrbracket \nonumber \nonumber \\
&\leq L' \left ( \sum_{i=1}^n \left | \mdsi{x}{m}{i} - \flr{m \left ( \mm{\mu_{x_i}} \right )} \right | + \sum_{i=1}^{n-1} \sum_{r=1}^{n-i} 
\maxonedmdp{\msi{\gamma}{m}{i,i+r}}{p} \right). \label{eqn:trivupperbound}
\end{align}
For each natural number $i$ less than $n$, Lemma \ref{lem:meanimeasureisfinite} implies that the mean of each measure $\mm{\imesi{i}}$ is finite, hence 
$ \left | \mdsi{x}{m}{i} - \flr{m \left ( \mm{\imesi{i}} \right )} \right | $
is in $o(m)$ almost surely by the strong law of large numbers. We now show that $ \maxonedmdp{\msi{\gamma}{m}{i,i+r}}{p} $ is almost surely in $O(m)$ for all natural numbers $i$ and $r$, such that $i$ is less than $n$ and $i+r$ is less than or equal to $n$. The next three lemmas allow an induction argument.

\begin{lemma}
	\label{lem:gammaupperbound}
	Suppose that $i$ and $r$ are natural numbers such that $i$ is less than $n$ and $i+r$ is less than $n$. Then, 
	\begin{align}
	\maxonedmdp{\msi{\gamma}{m}{i,i+r}}{p} &\leq  r + 3
	\maxonedmdp{p^{-\flr{m \Di{i,i+r}}} \left( \msi{\varphi}{m}{i, i+r} - \tci{m}{i,i+r} \right )}{p} + 9 \sum_{k=1}^{r-1} \left ( 
	\maxonedmdp{p^{- \flr{m \Di{i,i+k}}}}{p} + \maxonedmdp{\tci{m}{i,i+k}}{p} + \maxonedmdp{\msi{\gamma}{m}{i+k,i+r}}{p} \right ). \nonumber 
	\end{align}
\end{lemma}
\begin{proof}
	The definition of matrix multiplication may be used to compute  $\msi{\gamma}{m}{i,i+r}$,
	\[ 
	\msi{\gamma}{m}{i,i+r} =  p^{-\flr{m \Di{i,i+r}}} \sum_{s=0}^r   \tcinversei{m}{i,i+s} \msi{\varphi}{m}{i+s, i+r}. \]
	From Proposition \ref{prop:recursiveutinverseformula}, we have the following recurrence relation 
	\begin{align*}
	\tcinversei{m}{i,i+s} = -\sum_{k=1}^s  \tci{m}{i,i+k}  \tcinversei{m}{i+k,i+s} \quad \textrm{and} \quad \tcinversei{m}{i,i+1} = -\tci{m}{i,i+1},
	\end{align*}
	for all natural numbers $i$ less than $n$ and $s$, such that $i+s$ is less than or equal to $n$, such that after reindexing and using the fact that $ \tcinversei{m}{jj} = 1 = \tci{m}{jj} = \msi{\varphi}{m}{jj}$ for all natural numbers $j$, 
	\begin{align*}
	\msi{\gamma}{m}{i,i+r} %&=  p^{-\flr{m \Di{i,i+r}}} \left ( \msi{\varphi}{m}{i, i+r} - \sum_{s=1}^r   \sum_{k=1}^s  \tci{m}{i,i+k}  \tcinversei{m}{i+k,i+s} \msi{\varphi}{m}{i+s, i+r} \right ) \\
	%&=  p^{-\flr{m \Di{i,i+r}}} \left ( \msi{\varphi}{m}{i, i+r} - \sum_{k=1}^r   \sum_{k=s}^r  \tci{m}{i,i+k}  \tcinversei{m}{i+k,i+s} \msi{\varphi}{m}{i+s, i+r} \right ) \\
	&=  p^{-\flr{m \Di{i,i+r}}} \left ( \msi{\varphi}{m}{i, i+r} - \tcinversei{m}{i+r,i+r} - \sum_{k=1}^{r-1}   \sum_{s=k}^r  \tci{m}{i,i+k}  \tcinversei{m}{i+k,i+s} \msi{\varphi}{m}{i+s, i+r} \right ).
	\end{align*}
	
	Combining this statement with the triangle inequality given in Lemma  \ref{lem:binnormtriangleineq}  gives
	\begin{align*}
	\maxonedmdp{\msi{\gamma}{m}{i,i+r}}{p}  %&=  p^{-\flr{m \Di{i,i+r}}} \left ( \msi{\varphi}{m}{i, i+r} - \sum_{s=1}^r   \sum_{k=1}^s  \tci{m}{i,i+k}  \tcinversei{m}{i+k,i+s} \msi{\varphi}{m}{i+s, i+r} \right ) \\
	%&=  p^{-\flr{m \Di{i,i+r}}} \left ( \msi{\varphi}{m}{i, i+r} - \sum_{k=1}^r   \sum_{k=s}^r  \tci{m}{i,i+k}  \tcinversei{m}{i+k,i+s} \msi{\varphi}{m}{i+s, i+r} \right ) \\
	&\leq \maxonedmdp{p^{-\flr{m \Di{i,i+r}}} \left ( \msi{\varphi}{m}{i, i+r} - \tcinversei{m}{i+r,i+r} \right )}{p}+ \maxonedmdp{p^{-\flr{m \Di{i,i+r}}} \left ( \sum_{k=1}^{r-1}   \sum_{s=k}^r  \tci{m}{i,i+k}  \tcinversei{m}{i+k,i+s} \msi{\varphi}{m}{i+s, i+r} \right )}{p}.
	\end{align*}
	Since 
	\[ 
	\msi{\gamma}{m}{i+k,i+r} =  p^{-\flr{m \Di{i+k,i+r}}} \sum_{s=k}^r   \tcinversei{m}{i+k,i+s} \msi{\varphi}{m}{i+s, i+r}, \]
	and  $\left | \flr{m \Di{i,i+k}} + \flr{m \Di{i+k,i+r}} - \flr{m \Di{i,i+r}} \right | \leq 1 $, we have	
	\begin{align*}
	\maxonedmdp{\msi{\gamma}{m}{i,i+r}}{p}
	%&\leq \maxonedmdp{p^{-\flr{m \Di{i,i+r}}} \left ( \msi{\varphi}{m}{i, i+r} - \tcinversei{m}{i+r,i+r} \right )}{p} + \maxonedmdp{ \sum_{k=1}^{r-1} \tci{m}{i,i+k} p^{-\flr{m \Di{i,i+k}}} \sum_{s=k}^r   p^{-\flr{m \Di{i+k,i+r}}} \tcinversei{m}{i+k,i+s} \msi{\varphi}{m}{i+s, i+r}}{p}, \\
	%&\leq \maxonedmdp{p^{-\flr{m \Di{i,i+r}}} \left ( \msi{\varphi}{m}{i, i+r} - \tcinversei{m}{i+r,i+r} \right )}{p}+  \sum_{k=1}^{r-1} \left ( 1 + \maxonedmdp{p^{-\flr{m \Di{i,i+k}}} \tci{m}{i,i+k}  \msi{\gamma}{m}{i+k,i+r}}{p}\right ) \\
	&\leq r + \maxonedmdp{p^{-\flr{m \Di{i,i+r}}} \left ( \msi{\varphi}{m}{i, i+r} - \tcinversei{m}{i+r,i+r} \right )}{p} +  9 \sum_{k=1}^{r-1} \left ( \maxonedmdp{\tci{m}{i,i+k}}{p} + \maxonedmdp{p^{-\flr{m \Di{i,i+k}}}}{p}+ \maxonedmdp{ \msi{\gamma}{m}{i+k,i+r}}{p} \right )
	\end{align*}
	via the triangle inequality given in Lemma  \ref{lem:binnormtriangleineq} and the multiplication bound given in Lemma \ref{lem:binnormmulttriangleineq}.
\end{proof}
\begin{lemma}
	\label{lem:omofdifferencesnegative}
	Suppose that $i$ and $r$ are natural numbers such that $i$ is less than $n$ and $i+r$ is less than $n$. Suppose that $\mu$ is a measure such that $\mti{D}{i,i+r}$ is negative. Let
	\[ \msi{\theta}{m}{i,i+r} = p^{-\flr{m \Di{i,i+r}}} \left ( \msi{\varphi}{m}{i, i+r} -\tci{m}{i,i+r} \right ) \]
	Then, $	\maxonedmdp{\mti{\theta}{i,i+r}}{p} $ is almost surely in $o(m)$.
\end{lemma}
\begin{proof}
	As $\mti{D}{i,i+r}$ is negative, $\msi{\varphi}{m}{i, i+r} $ is convergent to a real number. Suppose $m$ is sufficiently large, such that Equation \eqref{eqn:expandedrecurrence} gives
	\[ \msi{\varphi}{m}{i,i+r} = \sum_{a \in \MFI{r}} S_{a}^{(m)} \]
	where
	\[ S_{a}^{(m)} = \sum_{0 \leq b_1 < \dots < b_{\left | a \right |-1} < m} \prod_{n=0}^{\left | a \right | -2} \msi{f}{b_{n+1}+1}{i+a_n,i+a_{n+1}} p^{\mdsi{y}{b_{n+1}}{i + a_n} - \mdsi{y}{b_{n+1}}{i + a_{n+1}}}. \]
	such that
	\begin{align}
	\msi{\theta}{m}{i,i+r} = & \sum_{a \in \MFI{r}} p^{-\flr{m \Di{i,i+r}}}\sum_{0 \leq b_1 < \dots < b_{\left | a \right |-1} < m} \prod_{n=0}^{\left | a \right | -2} \msi{f}{b_{n+1}+1}{i+a_n,i+a_{n+1}} p^{\mdsi{y}{b_{n+1}}{i + a_n} - \mdsi{y}{b_{n+1}}{i + a_{n+1}}} \nonumber \\
	&- \flr{\sum_{a \in \MFI{r}} p^{-\flr{m \Di{i,i+r}}} \sum_{0 \leq b_1 < \dots < b_{\left | a \right |-1} < \infty} \prod_{n=0}^{\left | a \right | -2} \msi{f}{b_{n+1}+1}{i+a_n,i+a_{n+1}} p^{\mdsi{y}{b_{n+1}}{i + a_n} - \mdsi{y}{b_{n+1}}{i + a_{n+1}}}} \label{eqref:thetasimplified}
	\end{align}
	Taking absolute values and application of the properties of the floor function gives 
	%	\begin{align*}
	%	\left | \msi{\theta}{m}{i,i+r} \right | \leq&\ 1 + \left | \sum_{a \in \MFI{r}} p^{-\flr{m \Di{i,i+r}}} \sum_{0 \leq b_1 < \dots < b_{\left | a \right |-1} < m} \prod_{n=0}^{\left | a \right | -2} \msi{f}{b_{n+1}+1}{i+a_n,i+a_{n+1}} p^{\mdsi{y}{b_{n+1}}{i + a_n} - \mdsi{y}{b_{n+1}}{i + a_{n+1}}} \right. \\
	%	&- \left .  \flr{\sum_{a \in \MFI{r}} p^{-\flr{m \Di{i,i+r}}} \sum_{0 \leq b_1 < \dots < b_{\left | a \right |-1} < m} \prod_{n=0}^{\left | a \right | -2} \msi{f}{b_{n+1}+1}{i+a_n,i+a_{n+1}} p^{\mdsi{y}{b_{n+1}}{i + a_n} - \mdsi{y}{b_{n+1}}{i + a_{n+1}}}} \right .\\
	%	&- \left .  \flr{\sum_{a \in \MFI{r}} p^{-\flr{m \Di{i,i+r}}} \sum_{m \leq b_1 < \dots < b_{\left | a \right |-1} < \infty} \prod_{n=0}^{\left | a \right | -2} \msi{f}{b_{n+1}+1}{i+a_n,i+a_{n+1}} p^{\mdsi{y}{b_{n+1}}{i + a_n} - \mdsi{y}{b_{n+1}}{i + a_{n+1}}}} \right |
	%	\end{align*}
	%	%
	%	Further application of the properties of the floor function gives that
	%
	\begin{align*}
	\left | \msi{\theta}{m}{i,i+r} \right | \leq& \ 2 + \left |  \flr{\sum_{a \in \MFI{r}} p^{-\flr{m \Di{i,i+r}}} \sum_{m \leq b_1 < \dots < b_{\left | a \right |-1} < \infty} \prod_{n=0}^{\left | a \right | -2} \msi{f}{b_{n+1}+1}{i+a_n,i+a_{n+1}} p^{\mdsi{y}{b_{n+1}}{i + a_n} - \mdsi{y}{b_{n+1}}{i + a_{n+1}}}} \right | \\
	\leq& \ 3 + \left |  \sum_{a \in \MFI{r}} p^{-\flr{m \Di{i,i+r}}} \sum_{m \leq b_1 < \dots < b_{\left | a \right |-1} < \infty} \prod_{n=0}^{\left | a \right | -2} \msi{f}{b_{n+1}+1}{i+a_n,i+a_{n+1}} p^{\mdsi{y}{b_{n+1}}{i + a_n} - \mdsi{y}{b_{n+1}}{i + a_{n+1}}} \right | \\
	\leq& \ 3 +  \sum_{a \in \MFI{r}} p^{-\flr{m \Di{i,i+r}}} \sum_{m \leq b_1 < \dots < b_{\left | a \right |-1} < \infty} \prod_{n=0}^{\left | a \right | -2} \left | \msi{f}{b_{n+1}+1}{i+a_n,i+a_{n+1}} \right | p^{\mdsi{y}{b_{n+1}}{i + a_n} - \mdsi{y}{b_{n+1}}{i + a_{n+1}}},
	\end{align*}
	and repeated application of Lemma \ref{lem:driftbounds} gives
	\begin{align*}
	\left | \msi{\theta}{m}{i,i+r} \right | \leq& \ 3 +  p^{-\flr{m \Di{i,i+r}}}\sum_{a \in \MFI{r}} \sum_{m \leq b_1 < \dots < b_{\left | a \right |-1} < \infty} \prod_{n=0}^{\left | a \right | -2} \left | \msi{f}{b_{n+1}+1}{i+a_n,i+a_{n+1}} \right | p^{b_{n+1} \left ( D_{i+a_n,i+a_{n+1}} + \varepsilon \right )  },
	\end{align*}
	for any positive $\varepsilon$ and all sufficiently large $m$.  Since, by Lemma \ref{lem:subexpentries},  $\log \left ( 1 + \left | \msi{f}{b_{n+1}+1}{i+a_n,i+a_{n+1}} \right | \right )$ is in $o(b_{n+1})$, if we show that
	\[ U_{a}^{(m)} := p^{-\flr{m \Di{i,i+r}}} \sum_{m \leq b_1 < \dots < b_{\left | a \right |-1} < \infty} \prod_{n=0}^{\left | a \right | -2} p^{b_{n+1} \left ( D_{i+a_n,i+a_{n+1}} + \varepsilon \right )  } \]
	is in $p^{o(m)}$ for each $a$ in $\MFI{r}$, then $d_+^{p} \left ( \msi{\varphi}{m}{i,i+r} \right )$ is $o(m)$. If $ D_{i,i+a_{1}} + \varepsilon > 0$, then 
	\begin{align*}
	\sum_{m \leq b_1 < b_2} p^{b_{n+1} \left ( D_{i,i+a_{1}} + \varepsilon \right )} 
	&= \frac{p^{m (D_{i,i+a_{1}} + \varepsilon)} - p^{b_2 (D_{i,i+a_{1}} + \varepsilon)}}{p^{D_{i,i+a_{1}} + \varepsilon} - 1} \\
	&\leq p^{m (D_{i,i+a_{1}} + \varepsilon)}
	\end{align*}
	and if $ D_{i,i+a_{1}} + \varepsilon$ is negative, then 
	\begin{align*}
	\sum_{m \leq b_1 < b_2} p^{b_{1} \left ( D_{i,i+a_{1}} + \varepsilon \right )} 
	&= \frac{p^{b_2 (D_{i,i+a_{1}} + \varepsilon)} - p^{m (D_{i,i+a_{1}} + \varepsilon)}}{1- p^{D_{i,i+a_{1}} + \varepsilon} } \\
	&\leq K_1 p^{b_2 (D_{i,i+a_{1}} + \varepsilon)}
	\end{align*}
	where $K_1 = \frac{1}{1- p^{D_{i,i+a_{1}} + \varepsilon}}$. Suppose there is a natural number $k$ less than $|a|-1$ and a positive, real number $K_k$, such that
	\begin{align*}
	&\sum_{0 \leq b_1 < \dots < b_{\left | a \right |-1} < \infty} \prod_{n=0}^{\left | a \right | -2} p^{b_{n+1} \left ( D_{i+a_n,i+a_{n+1}} + \varepsilon \right )  } \\ &\leq  K_k  p^{m (D_{i,i+a_k} + k \varepsilon)}  \sum_{m+k \leq b_{k+1} < \dots < b_{\left | a \right |-1} < \infty} \prod_{n=k}^{\left | a \right | -2} p^{b_{n+1} \left ( D_{i+a_n,i+a_{n+1}} + \varepsilon \right )  } 
	\end{align*}
	Then, by the same argument as the base case, the statement is true for $k+1$. Hence, for all sufficiently large $m$, there is a non-negative integer $k$ less than $|a|-1$  and a positive, real number $K$, such that
	\[ U_{a}^{(m)} \leq K p^{m D_{i,i+r} -\flr{m \Di{i,i+r}}} p^{m (|a|-1) \varepsilon} ,   \]
	which is in $p^{o(m)}$ as $\varepsilon$ may be chosen to be arbitrarily small. Hence, $d_+^{p} \left ( \msi{\theta}{m}{i,i+r} \right )$ is $o(m)$.
	
	We now turn our attention to $d_-^{p} \left ( \msi{\varphi}{m}{i,i+r} \right )$. It follows from Equation \eqref{eqref:thetasimplified} that
	\begin{align*}
	d_-^{p}{(\theta_{i,i+r})} &\geq \min \left \{ 0,  d_-^{p} \left ( \sum_{a \in \MFI{r}}  p^{-\flr{m \Di{i,i+r}}} S_{a}^{(m)} \right ) \right \} \\
	&\geq  \min_{a \in \MFI{r}} \left \{ 0, d_-^{p} \left (  p^{-\flr{m \Di{i,i+r}}}  S_{a}^{(m)} \right ) \right \} \\
	&\geq  \min_{a \in \MFI{r}} \left \{ 0,  \min d_-^{p} \left ( p^{-\flr{m \Di{i,i+r}}}
	\prod_{n=0}^{\left | a \right | -2} \msi{f}{b_{n+1}+1}{i+a_n,i+a_{n+1}} p^{\mdsi{y}{b_{n+1}}{i + a_n} - \mdsi{y}{b_{n+1}}{i + a_{n+1}} }  \right ) \right \} \\
	%	&= \min_{a \in \MFI{r}}   \left \{ 0, \min_{0 \leq b_1 < \dots < b_{\left | a \right |-1} < m} \sum_{n=0}^{\left | a \right | -2}\left (\mdsi{y}{b_{n+1}}{i + a_n} - \mdsi{y}{b_{n+1}}{i + a_{n+1}} +  d_-^{p} (\msi{f}{b_{n+1}+1}{i+a_n,i+a_{n+1}}) \right ) -\flr{m \Di{i,i+r}} \right \} \\
	&= \min_{a \in \MFI{r}}  \left \{ 0,  \min \sum_{n=0}^{\left | a \right | -2}\left (\mdsi{y}{b_{n+1}}{i + a_n} - \mdsi{y}{b_{n+1}}{i + a_{n+1}} +  d_-^{p} (\msi{f}{b_{n+1}+1}{i+a_n,i+a_{n+1}}) \right ) -\flr{m \Di{i,i+r}} \right \},
	\end{align*}
	%_{0 \leq b_1 < \dots < b_{\left | a \right |-1} < m}
	where the second minimum in each line is taken over the set of points $b_1, \dots, b_{\left | a \right |-1}$ such that $0 \leq b_1 < \dots < b_{\left | a \right |-1} < m$.	Using Lemma \ref{lem:driftbounds},
	\begin{align*}
	d_-^{p}{(\theta_{i,i+r})} &\geq \min_{a \in \MFI{r}}   \left \{ 0,\min \sum_{n=0}^{\left | a \right | -2}\left ( b_{n+1} \mti{D}{i + a_n, i + a_{n+1}} +  d_-^{p} (\msi{f}{b_{n+1}+1}{i+a_n,i+a_{n+1}}) \right ) -\flr{m \Di{i,i+r}}  - m \varepsilon \right \}
	\end{align*}
	for any positive, real $\varepsilon$, where the second minimum in each line is taken over the set of points $b_1, \dots, b_{\left | a \right |-1}$ such that $0 \leq b_1 < \dots < b_{\left | a \right |-1} < m$. This bound is clearly in $o(m)$ for the all negative displacement case because 
	\[ m \Di{i,i+r} = m \sum_{n=0}^{\left | a \right | -2}\left (  \mti{D}{i + a_n, i + a_{n+1}} \right ). \]
	Hence $d_-^{p}(\theta_{i,i+r})$ is in $o(m)$. The statement of the lemma in the negative displacement case then follows from the definition of $\maxonedmdp{\cdot}{p}$.
\end{proof}
\begin{lemma}
	\label{lem:omofdifferencespositive}
	Suppose that $i$ and $r$ are natural numbers so that $i$ is less than $n$ and $i+r$ is less than $n$. Suppose that $\mu$ is a measure so that $\mti{D}{i,i+r}$ is positive. Let
	\[ \msi{\theta}{m}{i,i+r} = p^{-\flr{m \Di{i,i+r}}} \left ( \msi{\varphi}{m}{i, i+r} -\tci{m}{i,i+r} \right ) \]
	Then, $	\maxonedmdp{\mti{\theta}{i,i+r}}{p} $ is almost surely in $o(m)$.
\end{lemma}
\begin{proof}
	
	As $\mti{D}{i,i+r}$ is negative, $\msi{\varphi}{m}{i, i+r} $ is convergent to a real number. Suppose $m$ is sufficiently large, so that Equation \eqref{eqn:expandedrecurrence} gives
	\[ \msi{\varphi}{m}{i,i+r} = \sum_{a \in \MFI{r}} S_{a}^{(m)} \]
	where
	\[ S_{a}^{(m)} = \sum_{0 \leq b_1 < \dots < b_{\left | a \right |-1} < m} \prod_{n=0}^{\left | a \right | -2} \msi{f}{b_{n+1}+1}{i+a_n,i+a_{n+1}} p^{\mdsi{y}{b_{n+1}}{i + a_n} - \mdsi{y}{b_{n+1}}{i + a_{n+1}}}. \]
	so that
	\begin{align}
	\msi{\theta}{m}{i,i+r} = & \sum_{a \in \MFI{r}} p^{-\flr{m \Di{i,i+r}}}\sum_{0 \leq b_1 < \dots < b_{\left | a \right |-1} < m} \prod_{n=0}^{\left | a \right | -2} \msi{f}{b_{n+1}+1}{i+a_n,i+a_{n+1}} p^{\mdsi{y}{b_{n+1}}{i + a_n} - \mdsi{y}{b_{n+1}}{i + a_{n+1}}} \nonumber \\
	&- \fract \left ( \sum_{a \in \MFI{r}} p^{-\flr{m \Di{i,i+r}}} \sum_{0 \leq b_1 < \dots < b_{\left | a \right |-1} < \infty} \prod_{n=0}^{\left | a \right | -2} \msi{f}{b_{n+1}+1}{i+a_n,i+a_{n+1}} p^{\mdsi{y}{b_{n+1}}{i + a_n} - \mdsi{y}{b_{n+1}}{i + a_{n+1}}} \right )
	\end{align}
	Hence,
	\begin{align*}
	|\msi{\theta}{m}{i,i+r}| \leq 1 + p^{-\flr{m \Di{i,i+r}}}\sum_{a \in \MFI{r}} \sum_{0 \leq b_1 < \dots < b_{\left | a \right |-1} < m} \prod_{n=0}^{\left | a \right | -2} \left | \msi{f}{b_{n+1}+1}{i+a_n,i+a_{n+1}} \right | p^{b_{n+1} \left ( D_{i+a_n,i+a_{n+1}} + \varepsilon \right )  }
	\end{align*}
	for any positive, real $\varepsilon$.
	
	Let $\msi{g}{l}{ij} = \sum_{r=0}^l \left | \msi{f}{r}{ij} \right |$ for every pair of integers $i$ and $j$ less than $n$ and every natural number $l$. The sequence $\msi{g}{l}{ij}$ is in $o(l)$ almost surely by strong law of large numbers. Since the absolute value of the fractional part of any number is less than or equal to $1$, we may conclude that 
	\begin{align*} 
	\left | \msi{\gamma}{m}{i,i+1} \right | &\leq  1 + \sum_{l=0}^{m-1}  \left | \msi{f}{l+1}{i,i+1} \right | p^{\mdsi{y}{l}{i} -  \mdsi{y}{l}{i+1}-\flr{m \Di{i,i+1}}} \\
	&\leq 1 + \sum_{l=0}^{m-1}  \left | \msi{g}{l+1}{i,i+1} \right | p^{\mdsi{y}{l}{i} -  \mdsi{y}{l}{i+1}-\flr{m \Di{i,i+1}}}
	\end{align*}
	is almost surely in $o(m)$. Arguing as in the negative displacement case, 
	\begin{align*}
	d_-^{p} \left( \msi{\gamma}{m}{i,i+1} \right ) %&= d_-^{p}  \left ( 
	%\sum_{l=0}^{m-1}  \msi{f}{l+1}{i,i+1} p^{\mdsi{y}{l}{i} -  %\mdsi{y}{l}{i+1}-\flr{m \Di{i,i+1}}} \right . %\\
	%& \qquad - \left . \fract \left ( \sum_{l=0}^{\infty}  \msi{f}{l+1}{i,i+1} p^{\mdsi{y}{l}{i} - \mdsi{y}{l}{i+1}-\flr{m \Di{i,i+1}}} \right ) \right ) \\
	&\geq \min \left \{ 0, d_-^{p} \left ( \sum_{l=m}^{\infty}  \msi{f}{l+1}{i,i+1} p^{\mdsi{y}{l}{i} -  \mdsi{y}{l}{i+1}-\flr{m \Di{i,i+1}}} \right ) \right \} \\
	&\geq \min_{m \leq l} \left \{ 0, d_-^{p} \left ( \msi{f}{l+1}{i,i+1} p^{\mdsi{y}{l}{i} -  \mdsi{y}{l}{i+1}-\flr{m \Di{i,i+1}}} \right ) \right \} \\
	&\geq \min_{m \leq l} \left \{ 0, \mdsi{y}{l}{i} -  \mdsi{y}{l}{i+1}- \flr{m \Di{i,i+1}} + d_-^{p} \left ( \msi{f}{l+1}{i,i+1}  \right ) \right \},
	\end{align*}
	which is in $o(m)$ almost surely by the strong law of large numbers and Lemma \ref{lem:subexpentries}. 	We conclude that the lemma must be true in the positive displacement case.
\end{proof}
We have seen that $\maxonedmdp{p^{-\flr{m \Di{i,i+r}}} \left( \msi{\varphi}{m}{i, i+r} - \tci{m}{i,i+r} \right )}{p}$ is in $o(m)$ almost surely for all natural numbers $i$ and $r$ less than $n$. In particular, $ \maxonedmdp{\msi{\gamma}{m}{i,i+1}}{p}  $ is in $o(m)$. Inducting with the bound in Lemma \ref{lem:gammaupperbound} gives   $ \maxonedmdp{\msi{\gamma}{m}{i,i+r}}{p}  $ is in $O(m)$.

\subsection{Non-trivial boundaries for walks with homogeneous displacement}
\label{sec:maximality}
In \cite{kaimanovich91}, Kaimanovich described the Poison boundary of the affine group
\[  \Aff(\pyad{2}) = \left \{ \begin{pmatrix} 2^n & f \\ 0 & 1 \end{pmatrix} : n \in \I, f \in \pyad{2}  \right \}, \]
showing that for measures of finite first moment with respect to word length and non-zero drift that it is either isomorphic to the real line, or the $2$-adic numbers. The group $\Aff(\pyad{p})$ is isomorphic to the solvable Baumslag-Solitar group $BS(1,p)$, see McLaury \cite{mclaury12}. The proof generalizes easily to $\Aff(\pyad{p})$ and $\Aff(\pyad{p})$ is an undistorted subgroup of $\Gex$ for any $n$ greater than or equal to $2$.

Also in \cite{kaimanovich91}, Kaimanovich stated that it would be `interesting to investigate the \PF boundary for higher-dimensional solvable groups over diadics' and that the boundary would be mixed, `probably consisting of real and $2$-adic components'. 

Brofferio extended this description $\Aff \left ( \pyad{p} \right )$ to finitely generated subgroups of $\Aff \left ( \mathbb{Q} \right )$ in \cite{brofferio09} for measures with finite first moment with respect to an \emph{adelic length} on the group. The boundary is a product of $p$-adic fields with a hitting measure. He was later able to give more information in the finitely generated case in \cite{brofferio09}, proving maximality with Kaimanovich's strip criterion. 

As previously mentioned, Brofferio and Schapira \cite{brofferio2011poisson} showed that the boundary of $GL_n (\mathbb{Q})$ for measures of finite first moment with respect to the left-invariant pseudometric $d^a$ is a product of flag manifolds over $p$-adic fields, but did not give much information about the support of the measure on the boundary, except a triviality condition.

Let $(\Gamma, \nu)$ be the $\mu$-boundary from the previous subsection. Suppose that the entries of the displacement matrix are either all positive or all negative and that $n=3$. In this case, we can use we use  Kaimanovich's ray criterion, Theorem \ref{thm:kaimanovichapproxthm}, combined with our metric estimate to show that $(\Gamma, \nu)$ is maximal and that the boundary is indeed a product of real or $p$-adic components. Let $m$ be a natural number. We now define `approximation' maps $\Pi^{(m)}$ from $\Gamma$ to $\Gex$ of the form
\[ \Pi^{(m)} (b) = \left ( \vects{t}{m}, \tc{m}(b) \right ) \]
in order to use that theorem. Each function $\vects{t}{m}$ is only dependent on $\mu$, not the boundary point. We choose
\[ \vects{t}{m} = \left ( \flr{ m \left ( \mm{\mu_{x_1}} \right )}, \dots, \flr{m \left ( \mm{\mu_{x_n}} \right ) } \right ) \]
where $\flr{\cdot}$ is the floor function. Writing $\mti{M}{ij} = \tie{i}-\tie{j}$ for brevity, let
\begin{align*}
T^{(m)}(b)_{12} &= p^{\mti{M}{12}} \flr{p^{-\mti{M}{12}} \mti{b}{12} } \\
T^{(m)}(b)_{23} &= p^{\mti{M}{23}} \flr{p^{-\mti{M}{23}} \mti{b}{23} } \\
T^{(m)}(b)_{13} &= p^{\mti{M}{13}} \flr{p^{-\mti{M}{13}} \mti{b}{13}  } - p^{\mti{M}{13}}  \flr{\fracc{p^{-\mti{M}{12}}\mti{b}{12} } p^{-\mti{M}{23}} \mti{b}{23} } 
\end{align*}
if all displacements are negative and
\begin{align*}
T^{(m)}(b)_{12} &= p^{\mti{M}{12}} \fracc{p^{-\mti{M}{12}} \mti{b}{12} } \\
T^{(m)}(b)_{23} &= p^{\mti{M}{23}} \fracc{p^{-\mti{M}{23}} \mti{b}{23} } \\
T^{(m)}(b)_{13} &= p^{\mti{M}{13}} \fracc{p^{-\mti{M}{13}} \mti{b}{13}  } - p^{\mti{M}{13}}  \fracc{\flr{p^{-\mti{M}{12}}\mti{b}{12} } p^{-\mti{M}{23}} \mti{b}{23} } 
\end{align*}
if all displacements are positive, for each $b$ in $\Gamma$, where $\{ x \}$ is the fractional part of $x$. 

If the boundary point $b$ in $\Gamma$ is clear from the context, we write
%
% Let $T^{(m)}$ be the map from $\Gamma$ to $\Gex$ given by 
%%
%\[  \mti{T^{(m)}\left (\ms{\varphi}{\infty} \right )}{ij} =  p^{- \left | \flr{m D_{ij}} \right | }  \flr{p^{\left | \flr{m D_{ij}} \right |} \phi_{\left | \flr{m D_{ij}} \right |} \circ \pi_{\left | \flr{m D_{ij}} \right |} \left ( \msi{\varphi}{\infty}{ij} \right )  } \]
%
%%
%for each $\ms{\varphi}{\infty}$ in $\Gamma$, where $\flr{\cdot}$ is the floor function\nomenclature{$\flr{\cdot}$}{floor function}, $\pi_m$ is the canonical projection mapping from $\mathbb{S}_p$ to $\R / p^i \I$ and $\phi_m$ maps each coset $\R / p^i \I$ to its unique representative in $[0,p^i)$. If the boundary point $\varphi^{(\infty)}$ is clear from the context, we will use the following notation to
%
\begin{enumerate}[(i)]
	\item $\tc{m}$ to mean $\mt{T^{(m)}\left ( b \right )}$, 
	\item $\tcinverse{m}$ to mean $\left ( \mt{T^{(m)}\left ( b \right )} \right )^{-1}$ and
\end{enumerate}

Suppose that $\left ( \mds{y}{m}, \ms{\varphi}{m} \right ) $ is convergent to a matrix $\ms{\varphi}{\infty} $ in $\Gamma$. Then, 
\[ \left ( \Pi^{(m)} \left ( \ms{\varphi}{\infty} \right )   \right )^{-1} = \left ( -\vects{t}{m}, \zeta_{-\vects{t}{m}} \left ( \tcinverse{m} \right ) \right ), \]
and therefore
\begin{align*}
\left ( \Pi^{(m)} \left ( \ms{\varphi}{\infty} \right )   \right )^{-1} \left ( \mds{y}{m},  \ms{\varphi}{m} \right )
&= \left ( \vects{x}{m} - \vects{t}{m}, \ms{\gamma}{m} \right ),
\end{align*}
where
\[\ms{\gamma}{m} = \zeta_{-\vects{t}{m}} \left (\tcinverse{m} \ms{\varphi}{m}  \right ). \]
From the bound given in Proposition \ref{prop:wordlengthbound}, there is a positive, real constant $L'$ which is not dependent on the particular path so that 
\begin{align*}
\left | \left ( \Pi^{(m)} \left ( \ms{\varphi}{\infty} \right )   \right )^{-1} \left ( \mds{y}{m},  \ms{\varphi}{m} \right ) \right | &= \left | \left ( \vects{x}{m} - \vects{t}{m}, \ms{\gamma}{m} \right ) \right | \nonumber \\
&\leq L' \left \llbracket \left ( \vects{x}{m} - \vects{t}{m}, \ms{\gamma}{m} \right ) \right  \rrbracket \nonumber \nonumber \\
&\leq L' \left ( \sum_{i=1}^n \left | \mdsi{x}{m}{i} - \flr{m \left ( \mm{\mu_{x_i}} \right ) } \right | + \sum_{i=1}^{n-1} \sum_{r=1}^{n-i} \maxonedmdp{\msi{\gamma}{m}{i,i+r}}{p} \right). 
\end{align*}
For each natural number $i$ less than $n$, Lemma \ref{lem:meanimeasureisfinite} implies that the first moment of each measure $\mm{\imesi{i}}$ is finite, hence 
$ \left | \mdsi{x}{m}{i} - \flr{m \left ( \mm{\imesi{i}} \right ) } \right | $
is in $o(m)$ almost surely by the strong law of large numbers. By Proposition \ref{prop:recursiveutinverseformula}, 
\begin{align*}
\ms{\gamma}{m} &= \zeta_{-\vects{t}{m}} \left ( \begin{pmatrix} 1 & -\tc{m}_{12} & \tc{m}_{12}\tc{m}_{23} - \tc{m}_{13} \\ 0 & 1 & -\tc{m}_{12} \\ 0  & 0 & 1  \end{pmatrix} \ms{\varphi}{m} \right )
\end{align*}
In particular, 
\begin{align*}
\msi{\gamma}{m}{12} &= p^{-\mti{M}{12}} \left ( \msi{\varphi}{m}{12} - \tc{m}_{12} \right ), \\
\msi{\gamma}{m}{23} &= p^{-\mti{M}{23}} \left ( \msi{\varphi}{m}{23} - \tc{m}_{23} \right  ), \ \textrm{and} \\
\msi{\gamma}{m}{13} &= p^{-\mti{M}{13}} \left ( \msi{\varphi}{m}{13} - \tc{m}_{13} + \tc{m}_{12}\tc{m}_{23} - \msi{\varphi}{m}{12} \tc{m}_{23} \right ). \\
\end{align*}
The two propositions which follow show that $\maxonedmdp{\msi{\gamma}{m}{ij}}{p}$ is in $o(m)$ for the negative and positive displacement cases, respectively. This will allow use of the ray criterion. 
\begin{proposition}
	Suppose that all displacements are negative. Then, the sequence $\maxonedmdp{\msi{\gamma}{m}{ij}}{p}$ is almost surely in $o(m)$ for all natural numbers $i$ and $j$, such that $i < j \leq n$.
\end{proposition}
\begin{proof}
	The argument for $\maxonedmdp{\msi{\gamma}{m}{12}}{p}$ and $\maxonedmdp{\msi{\gamma}{m}{23}}{p}$ being in $o(m)$ follows from Lemma \ref{lem:gammaupperbound} and Lemma \ref{lem:omofdifferencesnegative}. An easy modification of the argument made by Kaimanovich in \cite{kaimanovich91} describing the \PF boundary of $\Aff (\pyad{2})$ could also be used. The remainder of the proof is concerned with showing that $\maxonedmdp{\msi{\gamma}{m}{13}}{p}$ is also in $o(m)$. 
	
	Without loss of generality, we may suppose that $m$ is large, so that Equation \eqref{eqn:expandedrecurrence} gives   
	\begin{align*}
	\msi{\varphi}{m}{12} &= \sum_{l=0}^m \msi{f}{l+1}{12} p^{\why{l}{1}{2}}, \\ \msi{\varphi}{m}{23} &= \sum_{l=0}^m \msi{f}{l+1}{23} p^{\why{l}{2}{3}}, \ \textrm{and} \\
	%a = (0,1,2), |a|=3
	\msi{\varphi}{m}{13} 
	&= \sum_{l=0}^m \msi{f}{l+1}{13} p^{\why{l}{1}{3}} %a = (0,2), |a|=2
	+  \sum_{l=1}^m \sum_{r=0}^l \msi{f}{r+1}{12} \msi{f}{l+1}{23} p^{\why{r}{1}{2}} p^{\why{l}{2}{3}}.
	%a = (0,1,2), |a|=3
	\end{align*}
	Substituting these expressions into $\msi{\gamma}{m}{13}$ 
	\begin{align*}
	\msi{\gamma}{m}{13} =& \sum_{0 \leq l < m} \msi{f}{l+1}{13} p^{\why{l}{1}{3} - \mti{M}{13}} %a = (0,2), |a|=2
	+  \sum_{0 \leq r < l < m} \msi{f}{r+1}{12} \msi{f}{l+1}{23} p^{\why{r}{1}{2}} p^{\why{l}{2}{3}  - \mti{M}{13}} \\
	&- \flr{\sum_{0 \leq l < \infty} \msi{f}{l+1}{13} p^{\why{l}{1}{3} - \mti{M}{13}} %a = (0,2), |a|=2
		+  \sum_{0 \leq r < l < \infty} \msi{f}{r+1}{12} \msi{f}{l+1}{23} p^{\why{r}{1}{2}} p^{\why{l}{2}{3} - \mti{M}{13}} } \\
	&+ \flr{ \sum_{0 \leq l < \infty} \msi{f}{l+1}{12} p^{\why{l}{1}{2}- \mti{M}{12} } } \flr { \sum_{0 \leq l < \infty} \msi{f}{l+1}{23} p^{\why{l}{2}{3} - \mti{M}{23}}} \\
	&- \left (  \sum_{0 \leq l < m} \msi{f}{l+1}{12} p^{\why{l}{1}{2}- \mti{M}{12}} \right ) \flr{ \sum_{0 \leq l < \infty} \msi{f}{l+1}{23} p^{\why{l}{2}{3} - \mti{M}{23}}  } \\
	&+  \flr { \fracc{ \sum_{0 \leq l < \infty} \msi{f}{l+1}{12} p^{\why{l}{1}{2}- \mti{M}{12} } } \sum_{0 \leq l < \infty} \msi{f}{l+1}{23} p^{\why{l}{2}{3} - \mti{M}{23}}} 
	\end{align*}
	using the identity $\flr{x} = x - \fracc{x}$,
	\begin{align*}
	\msi{\gamma}{m}{13} =& \sum_{0 \leq l < m} \msi{f}{l+1}{13} p^{\why{l}{1}{3} - \mti{M}{13}} %a = (0,2), |a|=2
	+  \sum_{0 \leq r < l < m} \msi{f}{r+1}{12} \msi{f}{l+1}{23} p^{\why{r}{1}{2}} p^{\why{l}{2}{3}  - \mti{M}{13}} \\
	&- \flr{\sum_{0 \leq l < \infty} \msi{f}{l+1}{13} p^{\why{l}{1}{3} - \mti{M}{13}} %a = (0,2), |a|=2
		+  \sum_{0 \leq r < l < \infty} \msi{f}{r+1}{12} \msi{f}{l+1}{23} p^{\why{r}{1}{2}} p^{\why{l}{2}{3} - \mti{M}{13}} } \\
	&+ \bracc{ \sum_{0 \leq l < \infty} \msi{f}{l+1}{12} p^{\why{l}{1}{2}- \mti{M}{12} } } \flr { \sum_{0 \leq l < \infty} \msi{f}{l+1}{23} p^{\why{l}{2}{3} - \mti{M}{23}}} \\
	&- \fracc{ \sum_{0 \leq l < \infty} \msi{f}{l+1}{12} p^{\why{l}{1}{2}- \mti{M}{12} } } \flr { \sum_{0 \leq l < \infty} \msi{f}{l+1}{23} p^{\why{l}{2}{3} - \mti{M}{23}}} \\
	&- \left (  \sum_{0 \leq l < m} \msi{f}{l+1}{12} p^{\why{l}{1}{2}- \mti{M}{12}} \right ) \flr{ \sum_{0 \leq l < \infty} \msi{f}{l+1}{23} p^{\why{l}{2}{3} - \mti{M}{23}}  } \\
	&+  \flr { \fracc{ \sum_{0 \leq l < \infty} \msi{f}{l+1}{12} p^{\why{l}{1}{2}- \mti{M}{12} } } \sum_{0 \leq l < \infty} \msi{f}{l+1}{23} p^{\why{l}{2}{3} - \mti{M}{23}}} 
	\end{align*}
	The last line and the third last line cancel from our lemma and so we can collect the terms from the second last and fourth last line to get
	\begin{align*}
	\msi{\gamma}{m}{13} =& \sum_{0 \leq l < m} \msi{f}{l+1}{13} p^{\why{l}{1}{3} - \mti{M}{13}} %a = (0,2), |a|=2
	+  \sum_{0 \leq r < l < m} \msi{f}{r+1}{12} \msi{f}{l+1}{23} p^{\why{r}{1}{2}} p^{\why{l}{2}{3}  - \mti{M}{13}} \\
	&- \flr{\sum_{0 \leq l < \infty} \msi{f}{l+1}{13} p^{\why{l}{1}{3} - \mti{M}{13}} %a = (0,2), |a|=2
		+  \sum_{0 \leq r < l < \infty} \msi{f}{r+1}{12} \msi{f}{l+1}{23} p^{\why{r}{1}{2}} p^{\why{l}{2}{3} - \mti{M}{13}} } \\
	&+ \bracc{ \sum_{m \leq l < \infty} \msi{f}{l+1}{12} p^{\why{l}{1}{2}- \mti{M}{12} } } \flr { \sum_{0 \leq l < \infty} \msi{f}{l+1}{23} p^{\why{l}{2}{3} - \mti{M}{23}}} 
	\end{align*}
	Then, we can use the fact that $\flr{x} = x - \fracc{x}$ again on the bottom two lines, to get
	\begin{align*}
	\msi{\gamma}{m}{13} =& -\sum_{m \leq l < \infty} \msi{f}{l+1}{13} p^{\why{l}{1}{3} - \mti{M}{13}} %a = (0,2), |a|=2
	- \sum_{l=m}^{\infty} \sum_{r=0}^l \msi{f}{r+1}{12} \msi{f}{l+1}{23} p^{\why{r}{1}{2}} p^{\why{l}{2}{3}  - \mti{M}{13}} \\
	&+ \fracc{\sum_{0 \leq l < \infty} \msi{f}{l+1}{13} p^{\why{l}{1}{3} - \mti{M}{13}} %a = (0,2), |a|=2
		+  \sum_{0 \leq r < l < \infty} \msi{f}{r+1}{12} \msi{f}{l+1}{23} p^{\why{r}{1}{2}} p^{\why{l}{2}{3} - \mti{M}{13}} } \\
	&+ \bracc{ \sum_{m \leq l < \infty} \msi{f}{l+1}{12} p^{\why{l}{1}{2}- \mti{M}{12} } } \bracc { \sum_{0 \leq l < \infty} \msi{f}{l+1}{23} p^{\why{l}{2}{3} - \mti{M}{23}}} \\
	&+ \bracc{ \sum_{m \leq l < \infty} \msi{f}{l+1}{12} p^{\why{l}{1}{2}- \mti{M}{12} } } \fracc { \sum_{0 \leq l < \infty} \msi{f}{l+1}{23} p^{\why{l}{2}{3} - \mti{M}{23}}} 
	\end{align*}
	Cancelling the third line with terms from the first line in the above we have, 
	\begin{align*}
	\msi{\gamma}{m}{13} =& -\sum_{m \leq l < \infty} \msi{f}{l+1}{13} p^{\why{l}{1}{3} - \mti{M}{13}} %a = (0,2), |a|=2
	+ \sum_{l=m}^{\infty} \sum_{r=l+1}^\infty \msi{f}{r+1}{12} \msi{f}{l+1}{23} p^{\why{r}{1}{2}} p^{\why{l}{2}{3}  - \mti{M}{13}} \\
	&+ \fracc{\sum_{0 \leq l < \infty} \msi{f}{l+1}{13} p^{\why{l}{1}{3} - \mti{M}{13}} %a = (0,2), |a|=2
		+  \sum_{0 \leq r < l < \infty} \msi{f}{r+1}{12} \msi{f}{l+1}{23} p^{\why{r}{1}{2}} p^{\why{l}{2}{3} - \mti{M}{13}} } \\
	&+ \bracc{ \sum_{m \leq l < \infty} \msi{f}{l+1}{12} p^{\why{l}{1}{2}- \mti{M}{12} } } \fracc { \sum_{0 \leq l < \infty} \msi{f}{l+1}{23} p^{\why{l}{2}{3} - \mti{M}{23}}} 
	\end{align*}
	Since the fractional part of any number is less than or equal to one,  Lemma \ref{lem:driftbounds} gives
	\begin{align*}
	\left | \msi{\gamma}{m}{13} \right | &\leq 3 + 2\sum_{m+1 \leq l < \infty} \left | \msi{f}{l+1}{13} \right | p^{l \mti{D}{13} + l\epsilon -  \mti{M}{13}} %a = (0,2), |a|=2
	+  \sum_{m+1 \leq r < l < \infty} \left | \msi{f}{r+1}{12} \right | \left | \msi{f}{l+1}{23} \right | p^{r \mti{D}{12}  + r\epsilon } p^{l \mti{D}{23} + l\epsilon - \mti{M}{13}}  
	\end{align*}
	almost surely for any positive $\varepsilon$ and all sufficiently large $m$. Since, by Lemma \ref{lem:subexpentries},  $\log \left ( 1 + \left | \msi{f}{m}{ij} \right | \right )$ is in $o(m)$, it follows that $\msi{\gamma}{m}{13} $ is in $p^{o(m)}$ and hence that $d_+^{p} \left ( \left | \msi{\gamma}{m}{13} \right | \right ) $ is in $o(m)$. On the other hand, 
	\begin{align*}
	\msi{\gamma}{m}{13} =& \sum_{0 \leq l < m} \msi{f}{l+1}{13} p^{\why{l}{1}{3} - \mti{M}{13}} %a = (0,2), |a|=2
	+  \sum_{0 \leq r < l < m} \msi{f}{r+1}{12} \msi{f}{l+1}{23} p^{\why{r}{1}{2}} p^{\why{l}{2}{3}  - \mti{M}{13}} \\
	&- \flr{\sum_{0 \leq l < \infty} \msi{f}{l+1}{13} p^{\why{l}{1}{3} - \mti{M}{13}} %a = (0,2), |a|=2
		+  \sum_{0 \leq r < l < \infty} \msi{f}{r+1}{12} \msi{f}{l+1}{23} p^{\why{r}{1}{2}} p^{\why{l}{2}{3} - \mti{M}{13}} } \\
	&+ \flr{ \sum_{0 \leq l < \infty} \msi{f}{l+1}{12} p^{\why{l}{1}{2}- \mti{M}{12} } } \flr { \sum_{0 \leq l < \infty} \msi{f}{l+1}{23} p^{\why{l}{2}{3} - \mti{M}{23}}} \\
	&- \left (  \sum_{0 \leq l < m} \msi{f}{l+1}{12} p^{\why{l}{1}{2}- \mti{M}{12}} \right ) \flr{ \sum_{0 \leq l < \infty} \msi{f}{l+1}{23} p^{\why{l}{2}{3} - \mti{M}{23}}  } \\
	&+ \flr { \fracc{ \sum_{0 \leq l < \infty} \msi{f}{l+1}{12} p^{\why{l}{1}{2}- \mti{M}{12} } } \sum_{0 \leq l < \infty} \msi{f}{l+1}{23} p^{\why{l}{2}{3} - \mti{M}{23}}} 
	%&+ \comment {  \flr{\sum_{0 \leq l < \infty} \msi{f}{l+1}{12} p^{\why{l}{1}{2} - \mti{M}{12}} } \fracc{\sum_{0 \leq l < \infty} \msi{f}{l+1}{23} p^{\why{l}{2}{3} - \mti{M}{23}}  }}
	\end{align*}
	Since the floor of any number has $d_-^{p}$ less than or equal to zero and $d_-^{p}$ of a sum is greater than the minimum $d_-^p$ value of the summands, we have 
	\begin{align*}
	d_-^{p}(\msi{\gamma}{m}{13}) \geq\min &\left \{ \min_{0 \leq l < m} \why{l}{1}{3} - \mti{M}{13} + d_-^{p} \left ( \msi{f}{l+1}{13} \right ), \right . \\
	&\quad \left . \min_{0 \leq l < m} \why{l}{1}{2} - \mti{M}{12} + d_-^{p} \left ( \msi{f}{l+1}{12} \right ), 0, \right . \\
	&\quad \left . \min_{0 \leq r < l < m}  \why{r}{1}{2} +  \why{l}{2}{3} - \mti{M}{13} + d_-^{p} \left ( \msi{f}{l+1}{12} \right ) + d_-^{p} \left ( \msi{f}{l+1}{23} \right ) \right \}
	\end{align*}
	Application of Lemma \ref{lem:driftbounds}, shows that this bound is in $o(m)$ almost surely.
\end{proof}
\begin{proposition}
	Suppose that all displacements are positive. The sequence $\maxonedmdp{\msi{\gamma}{m}{ij}}{p}$ is almost surely in $o(m)$ for all natural numbers $i$ and $j$, such that $i < j \leq n$.
\end{proposition}
\begin{proof}
	The argument is very similar to the negative displacement case, but we include it for completeness. For similar reasons, $\maxonedmdp{\msi{\gamma}{m}{12}}{p}$ and $\maxonedmdp{\msi{\gamma}{m}{23}}{p}$ are already in $o(m)$, so the remainder of the proof is concerned with showing that $\maxonedmdp{\msi{\gamma}{m}{13}}{p}$ is in $o(m)$. 
	
	Without loss of generality, we may suppose that $m$ is large, so that substituting the expressions for $\msi{\varphi}{m}{12}, \msi{\varphi}{m}{13} $ and $\msi{\varphi}{m}{23}$ we have
	\begin{align*}
	\msi{\gamma}{m}{13} =& \sum_{0 \leq l < m} \msi{f}{l+1}{13} p^{\why{l}{1}{3} - \mti{M}{13}} %a = (0,2), |a|=2
	+  \sum_{0 \leq r < l < m} \msi{f}{r+1}{12} \msi{f}{l+1}{23} p^{\why{r}{1}{2}} p^{\why{l}{2}{3}  - \mti{M}{13}} \\
	&- \fracc{\sum_{0 \leq l < \infty} \msi{f}{l+1}{13} p^{\why{l}{1}{3} - \mti{M}{13}} %a = (0,2), |a|=2
		+  \sum_{0 \leq r < l < \infty} \msi{f}{r+1}{12} \msi{f}{l+1}{23} p^{\why{r}{1}{2}} p^{\why{l}{2}{3} - \mti{M}{13}} } \\
	&+ \fracc{ \sum_{0 \leq l < \infty} \msi{f}{l+1}{12} p^{\why{l}{1}{2}- \mti{M}{12} } } \fracc { \sum_{0 \leq l < \infty} \msi{f}{l+1}{23} p^{\why{l}{2}{3} - \mti{M}{23}}} \\
	&- \left (  \sum_{0 \leq l < m} \msi{f}{l+1}{12} p^{\why{l}{1}{2}- \mti{M}{12}} \right ) \fracc{ \sum_{0 \leq l < \infty} \msi{f}{l+1}{23} p^{\why{l}{2}{3} - \mti{M}{23}}  } \\
	&+  \fracc { \flr{ \sum_{0 \leq l < \infty} \msi{f}{l+1}{12} p^{\why{l}{1}{2}- \mti{M}{12} } } \sum_{0 \leq l < \infty} \msi{f}{l+1}{23} p^{\why{l}{2}{3} - \mti{M}{23}}} 
	\end{align*}
	making similar arguments to the negative displacement case,
	\begin{align*}
	\msi{\gamma}{m}{13} =& -\sum_{m \leq l < \infty} \msi{f}{l+1}{13} p^{\why{l}{1}{3} - \mti{M}{13}} %a = (0,2), |a|=2
	+ \sum_{l=m}^{\infty} \sum_{r=l+1}^\infty \msi{f}{r+1}{12} \msi{f}{l+1}{23} p^{\why{r}{1}{2}} p^{\why{l}{2}{3}  - \mti{M}{13}} \\
	&+ \flr{\sum_{0 \leq l < \infty} \msi{f}{l+1}{13} p^{\why{l}{1}{3} - \mti{M}{13}} %a = (0,2), |a|=2
		+  \sum_{0 \leq r < l < \infty} \msi{f}{r+1}{12} \msi{f}{l+1}{23} p^{\why{r}{1}{2}} p^{\why{l}{2}{3} - \mti{M}{13}} } \\
	&+ \bracc{ \sum_{m \leq l < \infty} \msi{f}{l+1}{12} p^{\why{l}{1}{2}- \mti{M}{12} } } \flr { \sum_{0 \leq l < \infty} \msi{f}{l+1}{23} p^{\why{l}{2}{3} - \mti{M}{23}}} 
	\end{align*}
	Since $d_-^{p}(z)$ is positive for any integer,  Lemma \ref{lem:driftbounds} gives
	\begin{align*}
	d_-^{p} (\msi{\gamma}{m}{13}) = \min ( 0, & \min_{m \leq l < \infty} \why{l}{1}{3} - \mti{M}{13} + d_-^{p}(\msi{f}{l+1}{13}),  \\
	& \min_{m \leq l < \infty} \min_{l < r < \infty} (\why{l}{1}{2} + \why{l}{2}{3} - \mti{M}{13} + d_-^{p}(\msi{f}{r+1}{12}) + d_-^{p}(\msi{f}{l+1}{23})), \\
	& \min_{m \leq l < \infty} \why{l}{1}{2} - \mti{M}{12} + d_-^{p}(\msi{f}{l+1}{12}) )
	\end{align*}
	almost surely for any positive $\varepsilon$ and all sufficiently large $m$. Since, by Lemma \ref{lem:subexpentries},  $\log \left ( 1 + \left | \msi{f}{m}{ij} \right | \right )$ is in $o(m)$, it follows that $\msi{\gamma}{m}{13} $ is in $p^{o(m)}$ and hence that $d_+^{p} \left ( \left | \msi{\gamma}{m}{13} \right | \right ) $ is in $o(m)$. On the other hand, 
	\begin{align*}
	\msi{\gamma}{m}{13} =& \sum_{0 \leq l < m} \msi{f}{l+1}{13} p^{\why{l}{1}{3} - \mti{M}{13}} %a = (0,2), |a|=2
	+  \sum_{0 \leq r < l < m} \msi{f}{r+1}{12} \msi{f}{l+1}{23} p^{\why{r}{1}{2}} p^{\why{l}{2}{3}  - \mti{M}{13}} \\
	&- \fracc{\sum_{0 \leq l < \infty} \msi{f}{l+1}{13} p^{\why{l}{1}{3} - \mti{M}{13}} %a = (0,2), |a|=2
		+  \sum_{0 \leq r < l < \infty} \msi{f}{r+1}{12} \msi{f}{l+1}{23} p^{\why{r}{1}{2}} p^{\why{l}{2}{3} - \mti{M}{13}} } \\
	&+ \fracc{ \sum_{0 \leq l < \infty} \msi{f}{l+1}{12} p^{\why{l}{1}{2}- \mti{M}{12} } } \fracc { \sum_{0 \leq l < \infty} \msi{f}{l+1}{23} p^{\why{l}{2}{3} - \mti{M}{23}}} \\
	&- \left (  \sum_{0 \leq l < m} \msi{f}{l+1}{12} p^{\why{l}{1}{2}- \mti{M}{12}} \right ) \fracc{ \sum_{0 \leq l < \infty} \msi{f}{l+1}{23} p^{\why{l}{2}{3} - \mti{M}{23}}  } \\
	&+ \fracc { \flr{ \sum_{0 \leq l < \infty} \msi{f}{l+1}{12} p^{\why{l}{1}{2}- \mti{M}{12} } } \sum_{0 \leq l < \infty} \msi{f}{l+1}{23} p^{\why{l}{2}{3} - \mti{M}{23}}} 
	%&+ \comment {  \fracc{\sum_{0 \leq l < \infty} \msi{f}{l+1}{12} p^{\why{l}{1}{2} - \mti{M}{12}} } \flr{\sum_{0 \leq l < \infty} \msi{f}{l+1}{23} p^{\why{l}{2}{3} - \mti{M}{23}}  }}
	\end{align*}
	Since the fractional part of any number is less than or equal to one, 
	\begin{align*}
	|\msi{\gamma}{m}{13}| \leq  3 &+ \sum_{0 \leq l < m} | \msi{f}{l+1}{13} | p^{\why{l}{1}{3} - \mti{M}{13}} %a = (0,2), |a|=2 
	\\
	&+  \sum_{0 \leq r < l < m} | \msi{f}{r+1}{12} | | \msi{f}{l+1}{23} | p^{\why{r}{1}{2}} p^{\why{l}{2}{3}  - \mti{M}{13}} \\
	&+  \sum_{0 \leq l < m} | \msi{f}{l+1}{12} | p^{\why{l}{1}{2}- \mti{M}{12}} 
	\end{align*}
	Application of Lemma \ref{lem:driftbounds}, shows that this bound is in $o(m)$ almost surely.
\end{proof}
\begin{theorem}
	\label{thm:GnPFB}
	Suppose that $\mu$ is a probability measure on $\Gex$ that has finite first moment with respect to word length and non-zero displacements so that the \PF boundary of $(\Gex, \mu)$ is non-trivial, as given by the conditions in Proposition \ref{prop:Gnboundarytrivcondition}. Suppose that $n$ is less than or equal to $3$. Then, the \PF boundary of $(\Gex, \mu)$ is $(\Gamma, \nu)$ provided that all entries in the displacement  matrix $D$ associated with $\mu$ are of the same sign. 
\end{theorem}
\begin{proof}
	Recall that
	\begin{align*}
	\left ( \Pi^{(m)} \left ( \ms{\varphi}{\infty} \right )   \right )^{-1} \left ( \mds{y}{m},  \ms{\varphi}{m} \right )
	&= \left ( \vects{x}{m} - \vects{t}{m}, \ms{\gamma}{m} \right ),
	\end{align*}
	where
	\[\ms{\gamma}{m} = \zeta_{-\vects{t}{m}} \left (\tcinverse{m} \ms{\varphi}{m}  \right ). \]
	Further recall that there is a positive, real constant $L$, such that the summands of the right hand side of this inequality are almost surely in $o(m)$. Therefore, $(\Gamma, \nu)$ is the \PF boundary by Kaimanovich's ray criterion.  
\end{proof}
\subsection{Trivial boundaries}

\label{sec:triviality}

In this section we give necessary and sufficient conditions for probability measures $\mu$ on $\Gex$ to have a trivial \PF boundary.

\begin{lemma}
	\label{lem:supppropogation}
	Suppose that $\mu$ is a probability measure on $\Gex$. Let $\nproji{pq}$ be the projection map from $\Gex$ to $\dyad$ given by $ \nproji{pq}(x,f) = \mti{f}{pq} $ and let
	\[ (\mds{y}{m}, \ms{\varphi}{m}) = \prod_{i=1}^m (\mds{x}{m}, \ms{f}{m}) \]
	be a path in the random walk $(\Gex, \mu)$. If
	\begin{equation}
	\sgr \left (\pf{\nproji{pq}}{\mu}(\Gex)  \right )  = \{0\} \label{eqn:sgrconditiononentires}
	\end{equation}
	for natural numbers $p$ and $q$ so that $p$ is less than $q$, which is less than or equal to $n$, then $\msi{\varphi}{m}{ij}$ is almost surely zero for all integers $i$ and $j$ satisfying $p \leq i < j \leq q$.
	%
	%
	%
	%\[ \supp \left ( \mu \circ \nproji{ij}^{-1} (\Gex) \right ) = \{0\}  \]
	%%
	%for all $i,j$ satisfying $p \leq i < j \leq q$. In particular, i
\end{lemma}
\begin{proof}
	Suppose that $\msi{f}{m}{ij} \neq 0$ for some natural numbers $i$ and $j$ satisfying $p \leq i < j \leq q$ and some natural number $m$. It follows from Equation \eqref{eqn:expandedrecurrence} that our supposition is in contradiction with the assumption in the statement of the lemma. 
\end{proof}

\begin{proposition}
	
	\label{prop:Gnboundarytrivcondition}
	Suppose that $\mu$ is a probability measure on $\Gex$. Let $D$ be the associated displacement matrix. The \PF boundary of $(\Gex, \mu)$ is trivial if and only if for every pair of natural numbers $p$ and $q$ satisfying $p < q \leq n$ any of following conditions are satisfied:
	\begin{enumerate}
		\item $\mti{D}{pq} = 0$, or
		\item $\sgr \left (\pf{\nproji{pq}}{\mu} (\Gex)  \right )  = \{0\}$.
	\end{enumerate}
	In particular, if $\D$ is of all zero entries or $\sgr \mu$ generates an Abelian group then the \PF boundary of $(\Gex, \mu)$ is trivial.
\end{proposition}

\begin{proof}
	Suppose that $p$ and $q$ are a pair of natural numbers so that $p < q \leq n$ and at least one of the two conditions is satisfied. Let $R_n$ be the subgroup of $\Hn$ given by 
	\[ R_n = \{ (x_1, \dots, x_n) : x_i = x_j \ \textrm{for all} \ i, \ \textrm{and} \ j \ \textrm{such that} \ \Di{ij} = 0 \}. \] 
	Then, $\Gexnp{0}{2} = R_n \ltimes N_n$ is a subgroup which is recurrent in $\Hn \ltimes \Nn$. The action of $R_n$ is trivial on $\Nn$ by Lemma \ref{lem:supppropogation}, Equation \eqref{eqn:expandedrecurrence} and the condition imposed on the support of $\mu$. Consequently, $\Gexnp{0}{2}$ is nilpotent.
	%
	%
	%isomorphic to a nilpotent group as 
	%
	%The direct product of two nilpotent groups is nilpotent,.
	Let $\mu_0$ be the hitting measure on $\Gexnp{0}{2}$. As $\Gexnp{0}{2}$ is nilpotent, the \PF boundary of $(\Gex, \mu_0)$ is trivial. As $\Gexnp{0}{2}$ is a recurrent subgroup,  Lemma 2.2 in Kaimanovich \cite{kaimanovich91} gives that the boundaries $(\Gexnp{0}{2}, \mu_0)$ and $(\Gex, \mu_0)$ are isomorphic, where $\mu_0$ is the hitting measure on $\Gexnp{0}{2}$. Hence the \PF boundary of $(\Gex, \mu)$ is trivial.
	
	Suppose that there are natural numbers $p$ and $q$ so that $\mti{D}{pq} \neq 0$ and $\sgr \left (\pf{\nproji{pq}}{\mu} (\Gex)  \right ) \neq \{0\}$ and, to obtain a contradiction, that the \PF boundary is trivial, i.e. that the limit $\ms{\varphi}{\infty}$ is the same for almost every pair of paths $(\mds{y}{m}, \ms{\varphi}{m})$.
	
	Then, the group generated by $\supp \mu$ is non-abelian and there must exist a pair of points in $\supp \mu$ which have commutator $(I_n, \mt{h})$, such that $\mti{h}{pq}$ is non-zero. The remainder of the argument is the same as the one given by Kaimanovich in \cite{kaimanovich91}. 	Since $\mti{D}{pq} \neq 0$, it is the case that $\msi{\varphi}{\infty}{pq} \neq 0 $ almost surely. Considering the action of the commutator $(I_n, \mt{h})$ on the supposedly unique boundary point $\ms{\varphi}{\infty}$, we see that
	\[ (I_n, \mt{h}) \cdot \ms{\varphi}{\infty} = \mt{h} I_n \ms{\varphi}{\infty} I_n^{-1} = \ms{\varphi}{\infty}. \]
	which is a contradiction, because it implies that $\mti{h}{pq} = 0$.   %Hence, the boundary must be non-trivial.
\end{proof}
If all the displacements in the random walk are zero, then an induction argument may be given to show that the boundary is trivial.  

\begin{proposition}
	\label{prop:altbndrytrivargument}
	Let $\mu$ be a probability measure on $\Gex$. Suppose that each random walk $(\I, \mu_{x_i})$ is recurrent for all natural numbers $i$ so that $i \leq n$. Then, the \PF boundary of $(\Gex, \mu)$ is trivial.
\end{proposition}
\begin{proof}
	Identify $\Gexnp{{k}}{2}$ with the subgroup in $\Gex$ consisting of elements $(\md{x}, \mt{f})$ so that $x_r = 0$ whenever $r>k$ and $\mti{f}{ij} = 0$ whenever $j > k$ and $j \neq i$. Each subgroup $\Gexnp{{k-1}}{2}$ is a recurrent set in $\Gexnp{{k}}{2}$ for the random walk associated with $(\Gexnp{{k}}{2}, \mu_k)$,  where $\mu_n = \mu$ and $\mu_{k-1}$ is the hitting measure of $(\Gexnp{{k}}{2}, \mu_k)$ on $\Gexnp{{k-1}}{2}$. For this reason we may identify the \PF boundary of $(\Gex, \mu)$ with the \PF boundary of $(\Gexnp{{1}}{2}, \mu_1)$ recursively applying Lemma 2.2 in Kaimanovich \cite{kaimanovich91} to the subgroups $\Gexnp{k}{2}$. However,  $\Gexnp{1}{2}$ is abelian because it is isomorphic to $\I$, so it has trivial boundary. 
\end{proof}

\bibliographystyle{plain}

\end{document}